\documentclass[a4paper,10pt]{amsart}
\usepackage{amsmath,amsthm,amssymb}
\usepackage[latin1]{inputenc}
\usepackage[T1]{fontenc}
\usepackage[all]{xy}
\usepackage{mathrsfs}
\usepackage{hyperref}

\numberwithin{equation}{section}
\newtheorem{thm}[equation]{Theorem}
\newtheorem*{thm*}{Theorem}
\newtheorem{cor}[equation]{Corollary}
\newtheorem{lem}[equation]{Lemma}
\newtheorem{prop}[equation]{Proposition}

\theoremstyle{definition}

\newtheorem{rem}[equation]{Remark}
\newtheorem{hyp}[equation]{Hypotheses}

\newtheorem{op}[equation]{Optimistic Comparison}
\usepackage{bm}
\usepackage{upgreek}
\makeatletter
\newcommand{\bfgreek}[1]{\bm{\@nameuse{up#1}}}

\newcommand{\triv}{{\mathbf{1}}}
\def\R{\mathbb R}

\def\Z{\mathbb Z}
\def\ZZ{\mathbb Z}
\def\A{\mathbb A}
\def\Q{\mathbb Q}
\def\QQ{\mathbb Q}
\def\C{\mathbb C}

\def\E{\mathcal E}
\def\K{\mathcal K}
\def\J{\mathcal J}

\def\ira{\stackrel{\sim}{\rightarrow}}
\def\hra{\hookrightarrow}

\def\ra{\rightarrow}

\def\g{\mathfrak g}

\def\a{\mathfrak a}

\def\mc{\mathbf c}

\def\O{\mathcal O}
\def\CC{\mathcal C}
\def\CE{\mathcal E}
\def\CW{\mathcal W}

\def\k{\mathfrak k}
\def\p{\mathfrak p}

\def\<{\langle}
\def\>{\rangle}
\def\triv{\mathbf 1}

\title[]%
{Whittaker periods, motivic periods, and special values of tensor product $L$-functions}
\author{Harald Grobner and Michael Harris}
\address{Harald Grobner: Fakult\"at f\"ur Mathematik\\ Universit\"at Wien\\ Oskar--Morgenstern--Platz 1\\ A-1090 Wien\\Austria}
\email{harald.grobner@univie.ac.at}

\address{Michael Harris: Univ Paris Diderot, Sorbonne Paris Cit\'e, UMR 7586, Institut de
Math\'ematiques de Jussieu-Paris Rive Gauche, Case 247, 4 place Jussieu
F-75005, Paris, France;~\\
Sorbonne Universit\'es, UPMC Univ Paris 06, UMR 7586, IMJ-PRG, F-75005
Paris, France;~\\
CNRS, UMR7586, IMJ-PRG, F-75013 Paris, France;~  \\ Department of
Mathematics, Columbia University, New York, NY  10027, USA. }
\email{}

\subjclass[2010]{Primary: ; Secondary: }
\thanks{H.G. is supported in parts by the Austrian Science Fund (FWF) Erwin Schr\"odinger grant, J 3076-N13, and the FWF-project, project number P 25974-N25.}
\thanks{The research leading to these results has received funding from the European Research Council under the European Community's Seventh Framework Programme (FP7/2007-2013) / ERC Grant agreement n¡ 290766 (AAMOT)}

\begin{document}

\begin{abstract} Let $\K$ be an imaginary quadratic field.  Let $\Pi$ and $\Pi'$ be irreducible generic cohomological automorphic
representation of $GL(n)/{\K}$ and $GL(n-1)/{\K}$, respectively.  Each of them can be given two natural rational structures over number
fields.  One is defined by the rational structure on topological cohomology, the other is given
in terms of the Whittaker model.  The ratio between these rational structures is called
a {\it Whittaker period}. An argument presented by Mahnkopf and Raghuram shows that, at least
if $\Pi$ is cuspidal and the weights of $\Pi$ and $\Pi'$ are in a standard relative position,
the critical values of the Rankin-Selberg product $L(s,\Pi \times \Pi')$
are essentially algebraic multiples of the product of the Whittaker periods of $\Pi$ and $\Pi'$.
We show that, under certain regularity and polarization hypotheses, the Whittaker period of a cuspidal $\Pi$
can be given a motivic interpretation, and can also be related to a critical value of
the adjoint $L$-function of related automorphic representations of unitary groups.  The resulting expressions for critical values of
the Rankin-Selberg and adjoint $L$-functions are compatible with Deligne's conjecture.

\end{abstract}
\maketitle

\setcounter{tocdepth}{1}
\tableofcontents

\section{Introduction}

$L$-functions can be attached both to automorphic representations and to arithmetic objects such as
Galois representations or motives, and one implication of the Langlands program is that $L$-functions
of the second kind are examples of $L$-functions of the first kind.  Very few results of arithmetic
interest can be proved about the second kind of $L$-functions until they have been identified
with automorphic $L$-functions.  For example, there is an extraordinarily deep web of conjectures
relating the values at integer points of arithmetic (motivic) $L$-functions to cohomological
invariants of the corresponding geometric (motivic) objects.  In practically all the instances \footnote{Whether or not the analytic formulas
for special values of the Riemann zeta function and Dirichlet $L$-series should be considered exceptions
is beyond the scope of this introduction.}  where these
conjectures have been proved, automorphic methods have proved indispensable.

At the same time, there is a growing number of results on special values of automorphic
$L$-functions that make no direct reference to arithmetic.  Instead, the special values are
written as algebraic multiples of complex invariants defined by means of representation theory.
Examples relevant to the present paper include \cite{mahnk}, \cite{raghuram-imrn}, \cite{ragh-gln} and \cite{grob-ragh}, where the complex invariants
are defined for representations of cohomological type by reference to the uniqueness of Whittaker
or Shalika models.  The present paper continues this series, proving a version of the
main results of \cite{mahnk}, \cite{raghuram-imrn} and \cite{ragh-gln} for the Rankin-Selberg $L$-function of a
pair $(\Pi, \Pi')$ of a cuspdial automorphic representation $\Pi$ of $GL_n$ and a cuspidal {\it or} abelian automorphic representation $\Pi'$ of $GL_{n-1}$, where the general linear groups are over an imaginary quadratic field $\K$. Here, the notion ``abelian automorphic'' refers to a representation, which is a tempered Eisenstein representation induced from a Borel subgroup of $GL_{n-1}$. In view of Raghuram's recent preprint  \cite{ragh-gln}, which we
received after writing a first version of this paper, the inclusion of such automorphic representations $\Pi'$ is the new feature of this result. %As the proof of Thm.\ \ref{thm:whittaker-periods-general} for cuspidal $\Pi'$ happens to be a side-effect of the proof of our Thm.\ \ref{thm:whittaker-periods} for abelian $\Pi'$, we decided to leave Thm.\ \ref{thm:whittaker-periods-general} in its present form.
The theorem applies in particular when $\Pi$ and $\Pi'$ are obtained by base change from cohomological cuspidal representations $\pi$ and $\pi'$ of unitary groups.  As in \cite{mahnk}, \cite{raghuram-imrn} and \cite{ragh-gln} the critical values of these $L$-functions are then expressed in terms of {\it Whittaker periods}, which are purely representation-theoretic invariants.   Here is a statement.  In what follows, $p(\Pi)$ and $p(\Pi')$ are
the Whittaker periods, which belong to $\C^{\times}$; this and the remaining notation will be explained in the subsequent sections.

\begin{thm}\label{thm:whittaker-periods-general}
Let $\Pi=BC(\pi)\|\cdot\|^{\textbf{\sf m}}$ be a cuspidal automorphic representation of $GL_n(\A_\K)$, ${\textbf{\sf m}}\in\Z$, and let $\Pi'=BC(\pi')$ be a cuspidal or abelian automorphic representation of $GL_{n-1}(\A_\K)$ obtained by base change from unitary groups as in Sections \ref{sect:Pi} and \ref{sect:Pi'}. In particular, $\Pi$ is cohomological with respect to $E_\mu$ and $\Pi'$ is cohomological with respect to $E_\lambda$.
We assume the parameters $\mu$ and $\lambda$ satisfy hypothesis \ref{hyp:coeff}.  Then the following holds:
\begin{enumerate}
\item For all critical values $\tfrac12+m\in\textrm{\emph{Crit}}(\Pi\times\Pi')$ with $m\geq 0$ and every $\sigma\in {\rm Aut}(\C)$,
$$
\sigma\left(
\frac{L(\tfrac 12+m,\Pi_f \times \Pi'_f)}{p(\Pi)p(\Pi')p(m,\Pi_\infty,\Pi'_\infty) \mathcal{G}(\omega_{\Pi'_{f,0}})}\right) \ = \
\frac{L(\tfrac 12+m,{}^\sigma\Pi_f \times {}^\sigma\Pi'_f)}{p({}^\sigma\Pi)p({}^\sigma\Pi')p(m,{}^\sigma\Pi_\infty,{}^\sigma\Pi'_\infty) \mathcal{G}(\omega_{{}^\sigma\Pi'_{f,0}})}.
$$
\item
$$
L(\tfrac 12+m,\Pi_f \times \Pi'_f) \ \sim_{\Q(\Pi_f)\Q(\Pi'_f)} p(\Pi)p(\Pi')p(m,\Pi_\infty,\Pi'_\infty) \mathcal{G}(\omega_{\Pi'_{f,0}}),
$$
where ``$\sim_{\Q(\Pi_f)\Q(\Pi'_f)}$'' means up to multiplication by an element in the composition of number fields $\Q(\Pi_f)\Q(\Pi'_f)$.
\end{enumerate}
\end{thm}

The main purpose of the present paper is to bring this purely automorphic result closer to
the motivic expression for the critical values of $L(s,\Pi\times \Pi')$ embodied in Deligne's
conjecture \cite{deligne}.  The main step in the proof relies in fact on the extension of the results of \cite{mahnk}, \cite{raghuram-imrn}, \cite{ragh-gln} to the case, when $\Pi'$ is allowed to be abelian automorphic. In that case, the $L$-function factors as a product of $n-1$ $L$-functions of $GL_n$ of the type considered in \cite{harcrelle} (completed by \cite{har2007}).  There, the special values are expressed in terms of Petersson norms of arithmetically normalized
holomorphic automorphic forms on Shimura varieties attached to unitary groups.  While Deligne's conjecture provides
an expression of the critical values in terms of Deligne periods, which are defined in terms of motives, and whose
status is therefore at least partially hypothetical, the Petersson norms are invariants of authentically arithmetic
objects.  The purpose of Section \ref{sect:motives} is to provide plausible reasons to identify Deligne's periods
in our situation with the expressions derived using \cite{harcrelle}.

It should be noted that \cite{harcrelle} was written with applications to self-dual cohomological automorphic representations in mind,
although many of the results are valid more generally.  For this reason, the standing self-duality
hypothesis of \cite{harcrelle} is sometimes invoked tacitly, simplifying some statements but making it
difficult for the reader (including the author of \cite{harcrelle}, at more than 15 years' distance) to determine which
statements need to be modified to treat the general case.  We have therefore attempted in Section \ref{sect:motives},
and especially in \ref{unitaryperiods}, to provide complete statements with precise parameters.  The authors
hope that definitive statements, over general CM fields, will be available in the near future.

When $\Pi'$ is an Eisenstein representation as in Thm.\ \ref{thm:whittaker-periods}, its Whittaker invariant
 $p(\Pi')$ can also be identified with a motivic period, using Shahidi's calculation of Whittaker coefficients for
 Eisenstein series.   A similar identification was already exploited in \cite{mahnk}, for Eisenstein classes
 attached to other parabolic subgroups. In our case, the motivic period $p(\Pi')$ exactly cancels the terms in
 the expression calculated in \cite{harcrelle}, yielding an expression (Thm.\ \ref{heckethm}) for $p(\Pi)$ solely in terms of Petersson norms of
 holomorphic forms on Shimura varieties, multiplied by purely archimedean invariants that have not yet been
 calculated explicitly:

\begin{thm}\label{thm:67}
Let $\Pi$ be a cuspidal automorphic representation of $GL_n(\A_\K)$, which can be obtained by base change from unitary groups of all signatures (i.e., more precisely, satisfies Hypothesis \ref{descent}). Assume that $\Pi$ is cohomological with respect to the irreducible, finite-dimensional, algebraic representation $E_\mu$ of highest weight $\mu_G=(\mu_1,...,\mu_n;-\mu_n,...,-\mu_1)$ and suppose that $\mu_i - \mu_{i+1} \geq 2$ for all $i$, so that there is a highest weight $\lambda_G=(\lambda;\lambda^{\sf v})$ for $G'_\infty$ and an integer $m > 0$ such that $s_0 = \tfrac12 +m$ is a critical value of $L(s,\Pi\otimes \Pi')$ for any cuspidal automorphic representation $\Pi'$ of $GL_{n-1}(\A_\K)$, which is cohomological with respect to $E_\lambda$.  Suppose that $\Pi$ is obtained by base change from unitary groups of all signatures (i.e., satisfies Hypothesis \ref{descent}). Then there is a non-zero constant $Z(\Pi_{\infty})$, which depends only on the local representation $\Pi_{\infty}$ and such that
$$p(\Pi)  \sim_{\K E(\Pi)}
Z(\Pi_{\infty})\prod_{j = 1}^{n-1} P^{(j)}(\Pi).$$
\end{thm}

The factors $P^{(j)}(\Pi)$ are the arithmetic Petersson norms mentioned above. Invoking Hypothesis \ref{descent} guarantees that they can be defined.  The constant
$Z(\Pi_{\infty})$, like the expression $p(m,\Pi_\infty,\Pi'_\infty)$ in Thm.\ \ref{thm:whittaker-periods-general}, is well-defined up to a non-zero scalar in $\K$.

It should be mentioned that the latter theorem (or equivalently Thm.\ \ref{heckethm}) is conditional on a global non-vanishing hypothesis. This hypothesis appears to be extremely difficult in general, but fortunately can always be satisfied provided the infinitesimal character of $\Pi$ is sufficiently regular.

It should also be mentioned that the main contribution to the archimedean invariants just mentioned is a
 Rankin-Selberg zeta integral involving explicit cohomological vectors.  For nearly 40 years, results of
 this type, for $GL_n$ with $n > 2$, were presented subject to a local non-vanishing hypothesis, which
 in our present situation comes down to the claim that the archimedean zeta integral does not equal zero for
 the choice of vectors forced upon us by the method.  Just last year, however, Binyong Sun discovered an
 abstract representation-theoretic method that allowed him to prove that such integrals never vanish.
 In his companion paper \cite{sun} he applies this method to cohomological Rankin-Selberg zeta
 integrals for $GL_n(\R)\times GL_{n-1}(\R)$ and for  $GL_n(\C)\times GL_{n-1}(\C)$.  Without Sun's
 breakthrough, the results of the present paper would still be conditioned on the non-vanishing
 hypothesis.

 In summary, this allows us to prove our main theorem (Thm.\ \ref{tensorproduct})
 expressing critical values of Rankin-Selberg $L$-functions in terms of the cohomology of algebraic
 varieties:

\begin{thm}\label{mainthm}  Let $\Pi$ and $\Pi'$ be cuspidal automorphic representations of $GL_n(\A_{\K})$ and
$GL_{n-1}(\A_{\K})$ which are cohomological with respect to $E_{\mu}$ and $E_{\lambda}$, respectively,
assumed to satisfy the equivalent conditions of Lem.\ \ref{lem:coeff} (i.e., hypothesis \ref{hyp:coeff} for $\textbf{\sf m}=0$).   We assume
$\Pi = BC(\pi)$ and $\Pi' = BC(\pi')$ are obtained by base change from definite unitary groups and let $E(\Pi)$ (resp.\ $E(\Pi')$) be the fields of definition of the corresponding motives $M(\Pi)$ (resp.\ $M(\Pi')$).
Suppose
\begin{enumerate}
\item  $\mu_i - \mu_{i+1} \geq 2$ for all $i$ and $\lambda_j - \lambda_{j+1} \geq 2$ for all $j$.
\item Both $\Pi$ and $\Pi'$ can be obtained by base change from holomorphic discrete series representations of unitary groups of arbitrary signature (i.e., satisfy Hypothesis \ref{descent}).
\end{enumerate}
Then for every critical point $s_0 =\tfrac12+ m$ of $L(s,\Pi\times\Pi')$ with $m \geq 0$,
$$L(\tfrac12 +m,\Pi_f \times \Pi'_f) \ \sim_{\K E(\Pi) E(\Pi')} p(m,\Pi_\infty,\Pi'_\infty) Z(\Pi_{\infty})Z(\Pi'_{\infty}) \prod_{j = 1}^{n-1} P^{(j)}(\Pi)
\prod_{k = 1}^{n-2} P^{(k)}(\Pi'),$$
where ``$\sim_{\K E(\Pi) E(\Pi')}$'' means up to multiplication by an element in the composition of number fields $\K E(\Pi) E(\Pi')$.
Equivalently, for every critical point $s_0 = n-1 + m$ of $L(s,R(M(\Pi)\otimes M(\Pi')))$ with $m \geq 0$,
$$
L(n-1+m,R(M(\Pi) \otimes M(\Pi'))) \sim_{\K E(\Pi) E(\Pi')} p(m,\Pi_\infty,\Pi'_\infty) Z(\Pi_{\infty})Z(\Pi'_{\infty}) \prod_{j = 1}^{n-1} P^{(j)}(\Pi)\prod_{k = 1}^{n-2} P^{(k)}(\Pi'). $$
\end{thm}

For the convenience of the reader, we wrote the above theorem from the automorphic, as well as from the motivic point of view.
It is shown in Optimistic Comparison \ \ref{comparison} that
the period that appears in Thm. \ \ref{mainthm}, namely
$$ \prod_{j = 1}^{n-1} P^{(j)}(\Pi)
\prod_{k = 1}^{n-2}P^{(k)}(\Pi'), $$
is at least formally consistent with Deligne's conjecture, provided the Hodge types of the motives
attached to $\Pi$ and $\Pi'$ satisfy the relations derived from Lemma \ \ref{lem:coeff}.  The article
\cite{haradj} computes Deligne periods for general Rankin-Selberg products and obtains
different products of the same basic invariants  $P^{(j)}(\Pi)$ and $P^{(k)}(\Pi')$
(under the identifications of Optimistic Comparison \ \ref{comparison}).  The final section
explains how one might hope to express the critical values of general Rankin-Selberg products,
using normalized Rankin-Selberg integrals of Eisenstein cohomology classes. \\\\

Consider the constant $u(m,\Pi_\infty,\Pi'_\infty):=p(m,\Pi_\infty,\Pi'_\infty) Z(\Pi_{\infty})Z(\Pi'_{\infty})$ appearing in Thm.\ \ref{mainthm}.  The individual
terms in the product depend on normalizations of archimedean Whittaker models, as do the Whittaker periods $p(\Pi)$ and $p(\Pi')$
themselves, but the product is independent of normalizations, up to $\K^{\times}$.  In subsequent work, Lin Jie has identified
$u(m,\Pi_\infty,\Pi'_\infty)$ as the explicit power of $2 \pi i$ predicted by the automorphic version of Deligne's conjecture described
in section \ref{nnminus1}; see Theorem \ref{lin} for a precise statement.

In the absence of the regularity hypothesis, we can prove a rather different kind of result, using Wei Zhang's
recent proof of a version of Neal Harris's Ichino-Ikeda conjecture for unitary groups.  Suppose
$\Pi$ and $\Pi'$ satisfy the hypotheses of \cite{weizhang}; in particular, $\Pi$ and $\Pi'$
are both supercuspidal at some prime of $\K$ split over $\Q$.   Then the conclusion of
Thm.\ \ref{thm:whittaker-periods-general} holds, with $p(\Pi)$ and $p(\Pi')$ replaced by the
values at $s = 1$ of the adjoint $L$-functions of $\pi$ and $\pi'$ respectively, and with slightly
different archimedean factors (See Cor.\ \ref{Ladjoint}).

\begin{thm}
Let $\Pi$ and $\Pi'$ be cuspidal automorphic representations of $GL_n(\A_{\K})$ and $GL_{n-1}(\A_{\K})$ which are cohomological with respect to $E_{\mu}$ and $E_{\lambda}$, respectively, assumed to satisfy the equivalent conditions of Lem.\ \ref{lem:coeff}.  We assume
$\Pi = BC(\pi)$ and $\Pi' = BC(\pi')$ are obtained by base change from definite unitary groups and that the irreducible unitary cuspidal automorphic representations $\pi$ and $\pi'$ satisfy the hypotheses of \cite{weizhang} (more precisely, hypothesis \ref{hypotheses}).
%\begin{enumerate}
%\item  $\mu_i - \mu_{i+1} \geq 2$ for all $i$ and $\lambda_j - \lambda_{j+1} \geq 2$ for all $j$.
%\item Both $\Pi$ and $\Pi'$ satisfy Hypothesis \ref{descent}.
%\end{enumerate}
Then for every critical point $s_0 = \tfrac12 + m$ of $L(s,\Pi\times\Pi')$ with $m\geq 0$, there is a non-zero complex constant $a(m,\Pi_{\infty},\Pi'_{\infty})$,
depending only on the archimedean components of $\Pi$ and $\Pi'$, and an integer $a(n)$ only depending on $n$, such that,
$$L(\tfrac12 +m,\Pi_f \times \Pi'_f) \ \sim_{\Q(\varepsilon_\K) \Q(\pi)\Q(\pi')} a(m,\Pi_\infty,\Pi'_\infty) \mathcal{G}(\varepsilon_{\K,f})^{a(n)}L(1,\pi_f,Ad)L(1,\pi'_f,Ad),$$
where ``$\sim_{\Q(\varepsilon_\K) \Q(\pi)\Q(\pi')}$'' means up to multiplication by an element in the composition of number fields $\Q(\varepsilon_\K) \Q(\pi)\Q(\pi')$.
Equivalently, for every critical point $s_0 = n-1 + m$ of $L(s,R(M(\Pi)\otimes M(\Pi')))$ with $m\geq 0$,
$$L(n-1+m,R(M(\Pi) \otimes M(\Pi'_f))) \ \sim_{\Q(\varepsilon_\K) \Q(\pi)\Q(\pi')} a(m,\Pi_\infty,\Pi'_\infty) \mathcal{G}(\varepsilon_{\K,f})^{a(n)}L(1,\pi_f,Ad)L(1,\pi'_f,Ad).$$
\end{thm}

The Whittaker period invariants $p(\Pi)$ are defined by comparing
two rational structures on $\Pi_f$, one coming from the global realization of $\Pi_f$ in the cohomology
of the locally symmetric space $S_n$ attached to $GL_n$, the other from the Whittaker model.  It is well-known,
however, that $\Pi_f$ occurs in a range of cohomological degrees $b_n \leq i \leq c_n$, and each one gives rise to one (or more)
period invariants.  Our $p(\Pi)$ is attached to the lowest cohomological degree $b_n$;
the proof of Theorem \ \ref{thm:whittaker-periods-general} is based on the numerical coincidence
that $1+b_n + b_{n-1}$ equals the dimension of $S_{n-1}$.  In \cite{urban}, Eric Urban showed that, when $n = 2$,
the Whittaker period attached to  $\Pi$ in the top cohomological degree $c_2$ is related to a non-critical
special value of one of the two Asai $L$-functions attached to $\Pi$.  It can be deduced from the
arguments of \cite{weizhang} and the work of Jacquet, Rallis, and Flicker cited there that the phenomenon
Urban discovered remains valid for general $n$; this is the subject of our forthcoming article with E. Lapid
\cite{GHL}.

We gratefully acknowledge the help of F. Shahidi and W. Zhang, for their answers to our
questions about different steps in the proof of the main theorem.  A. Raghuram helped us clear
up an apparent contradiction between our results and Deligne's conjecture.  We are especially grateful
to Binyong Sun, for proving the non-vanishing theorem that makes our results unconditional and that
will inevitably make the whole field of special values considerably more attractive.

\section{Unitary groups and base change}
\subsection{Some standing assumptions and notation}
\subsubsection{Number fields}
Let $\K/\Q$ denote an imaginary quadratic extension of $\Q$. We write $S(\K)$ for the set of places of $\K$ and $S$ for the set of places of $\Q$. The ring of adeles over $\K$ (resp.\ over $\Q$) is denoted $\A_\K$ (resp.\ $\A$). The ring of integers of $\K$ is denoted $\mathcal O$ and we let $\mathfrak{D}_\K$ be the different of $\K/\Q$, i.e., $\mathfrak{D}_\K^{-1} = \{x \in \K : Tr_{\K/\Q}(x \O) \subset \Z\}$, where $Tr_{\K/\Q}$ is the trace map from $\K$ to $\Q$. The normalized absolute value of $\A_\K$ is denoted $\|\cdot\|$. We extend the Hecke character
$\varepsilon_{\K}: \Q^\times\backslash\A^\times\ra\C^\times$, associated to the extension $\K/\Q$ via class field theory, to a Hecke character $\eta:\K^\times\backslash\A^\times_\K\ra\C^\times$. It is easy to see that $\eta$ is unitary and $\eta_\infty(z)=z^t\bar z^{-t}$ for a certain $t=t(\eta)\in\tfrac 12+\Z$.

For convenience, we will also fix, once and for all, a non-trivial, continuous, additive character $\psi:\K\backslash \A_\K\ra\C^\times$. To this end, let $\psi_{\Q}$ be the additive character of ${\mathbb Q} \backslash {\mathbb A}$, as in Tate's thesis, namely, $\psi_{\Q}(x) = e^{2\pi i \Lambda(x)}$ with the $\Lambda = \sum_{p \leq \infty} \Lambda_p$, where $\Lambda_{\infty}(x) = -x$ for $x \in \R$ and $\Lambda_p(x_p)$, $x_p \in \Q_p$, is the finite, negative tail of $x_p$'s $p$-adic expansion. That is, if $x_p=\sum_{j\geq v}a_jp^j$, then $\Lambda_p(x_p)=\sum_{j= v}^{-1}a_jp^j$. If we write
$\psi_{\Q} =  \psi_{\R} \otimes \otimes_p \psi_{\Q_p}$,
then $\psi_{\R}(x) = e^{-2\pi ix}$ and $\psi_{\Q_p}$ is trivial on $\Z_p$ and non-trivial on $p^{-1}\Z_p$.

Now, we define $\psi: = \psi_{\Q} \circ Tr_{\K/\Q}$. If $\psi = \otimes_w \psi_w$, then the local characters are determined analogously. In particular, if $\mathfrak{D}_\K = \prod_{\wp} \wp^{r_{\wp}}$, the product running over
all prime ideals $\wp\subset\mathcal O$, then the conductor of the local character $\psi_{w}$ is $\wp_w^{-r_{\wp}}$, i.e., $\psi_{w}$ is trivial on $\wp_w^{-r_{\wp}}$ and non-trivial on $\wp_w^{-r_{\wp}-1}.$ We let $S(\psi)=\{\wp \nmid\mathfrak D_\K\}$.

\subsubsection{Hermitian spaces and unitary groups}
We let $V$ be an $n$-dimensional Hermitian space over $\K$, relative to the non-trivial element in Gal$(\K/\Q)$. Depending on the parity of $n$, we define
$$\gamma:=\left\{\begin{array}{ll}
 \triv & \textrm{if $n$ is even} \\
 \eta & \textrm{if $n$ is odd}
\end{array}
\right.$$
Hence, $\gamma_\infty(z)=z^c\bar z^{-c}$, where $c=0$, if $n$ is even and $c=t\in \tfrac12+\Z$, if $n$ is odd.

We let $V'\subset V$ be a subspace of codimension one, on which the restriction of the Hermitian form is non-degenerate so that $V=V'\oplus V'^\bot$. At infinity, we want $V$ (and $V'$) to be definite.\\
We let $H:=U(V)$ and $H':=U(V')$ be the corresponding unitary groups over $\Q$. In particular, $H(\R)\cong U(n)$, the compact unitary group of rank $n$ and similar $H'(\R)\cong U(n-1)$.

Furthermore, we let $G:=GL_n/\K$ and $G':=GL_{n-1}/\K$. We also define real Lie groups $K:=GU(V)(\R)$ (resp.\ $K':=GU(V')(\R)$), which we view inside $G(\C)$ (resp.\ $G'(\C)$) by their natural embedding.

The Lie algebra of a real Lie group is denoted by the same but gothic, lower case letter.

\subsubsection{Finite-dimensional representations of real groups}
Let $E^{\sf unt}_\mu$ (resp.\ $E^{\sf unt}_\lambda)$) be an irreducible, finite-dimensional, algebraic representation of $H(\R)$ (resp.\ $H'(\R)$) on a complex vector space. Having fixed a maximal $\Q$-split torus in $H(\C)\cong GL_n(\C)$ (resp.\ $H'(\C)\cong GL_{n-1}(\C)$) and an ordering on the set of its dominant algebraic characters, we may think of $E^{\sf unt}_\mu$ (resp.\ $E^{\sf unt}_\lambda)$) as being given by its highest weight $\mu$ (resp.\ $\lambda$). We may arrange that $\mu=(\mu_1,...,\mu_n)$, where $\mu_1\geq...\geq\mu_n$, $\mu_i\in\Z$ for all $1\leq i\leq n$ and similar $\lambda=(\lambda_1,...,\lambda_{n-1})$ with $\lambda_1\geq...\geq\lambda_{n-1}$, $\lambda_i\in\Z$ for all $1\leq i\leq n-1$.

Furthermore, we write $E^{\sf unt}_\mu\otimes E^{\sf unt}_{\mu^{\sf v}}$ (resp.\ $E^{\sf unt}_\lambda\otimes E^{\sf unt}_{\lambda^{\sf v}}$) for the irreducible, finite-dimensional, algebraic representation of the {\it real} Lie group $G(\C)$ (resp.\ $G'(\C)$) of highest weight $(\mu,\mu^{\sf v})$ (resp.\ $(\lambda,\lambda^{\sf v})$). (Here the check ``${\sf v}$'' denotes taking the contragredient representation.) It is isomorphic to its complex conjugate contragredient.

\begin{lem}[See, e.g., \cite{goodman-wall}, Thm.\ 8.1.1]\label{lem:coeff}
With the above notation, the following assertions are equivalent:
\begin{enumerate}
\item $\mu_1\geq -\lambda_{n-1}\geq \mu_2\geq -\lambda_{n-2}\geq...\geq -\lambda_1\geq \mu_n$
\item \emph{Hom}$_{H'(\C)}(E^{\sf unt}_\mu\otimes E^{\sf unt}_\lambda,\C)\neq 0$
\end{enumerate}
If any of these two conditions is satisfied, then also
\emph{Hom}$_{G'(\C)}((E^{\sf unt}_\mu\otimes E^{\sf unt}_{\mu^{\sf v}})\otimes (E^{\sf unt}_\lambda\otimes E^{\sf unt}_{\lambda^{\sf v}}),\C)\neq 0$
\end{lem}

\subsection{Base change}\label{sect:cuspautom}
We let $\pi$ (resp.\ $\pi'$) be an irreducible, unitary cuspidal automorphic representation of $H(\A)$ (resp.\ $H'(\A)$). We shall always assume that $\pi$ (resp.\ $\pi'$) is cohomological with respect to $E^{\sf unt}_\mu$ (resp.\ $E^{\sf unt}_\lambda)$). Since $H(\R)$ and $H'(\R)$ are both compact and connected, this simply means that $\pi_\infty\cong E^{\sf unt}_{\mu^{\sf v}}$ and $\pi'_\infty\cong E^{\sf unt}_{\lambda^{\sf v}}$. We let $S(\pi)$ (resp.\ $S(\pi')$) be the set of all finite places of $\Q$, where $\pi$ (resp.\ $\pi'$) ramifies, together with the infinite place of $\Q$. The following is essentially a theorem of Labesse:

\begin{thm}[Base Change]\label{thm:base_change}
There exist irreducible automorphic representations $\Pi$ of $G(\A_\K)$ and $\Pi'$ of $G'(\A_\K)$, which are of the form
$$\Pi\cong \Pi_1\boxplus ...\boxplus\Pi_s\quad\quad\textrm{and}\quad\quad \Pi'\cong \Pi'_1\boxplus ...\boxplus\Pi'_r$$
where $\Pi_j$ (resp.\ $\Pi'_j$) is a square-integrable automorphic representation of $GL_{s_j}(\A_\K)$ (resp.\ $GL_{r_j}(\A_\K)$), for all $j$, such that:
\begin{itemize}
\item For all $p\notin S(\pi)$, the representation $\Pi_p:=\otimes_{w|p}\Pi_w$ is the local base change of $\pi_p$.
\item The archimedean component $\Pi_\infty$ of $\Pi$ is cohomological with respect to $E^{\sf unt}_\mu\otimes E^{\sf unt}_{\mu^{\sf v}}$, i.e.,  \\ $H^*(\g,K,\Pi_\infty\otimes (E^{\sf unt}_\mu\otimes E^{\sf unt}_{\mu^{\sf v}}))\neq 0$.
\end{itemize}
and analogously
\begin{itemize}
\item For all $p\notin S(\pi')$, the representation $\Pi'_p:=\otimes_{w|p}\Pi'_w$ is the local base change of $\pi'_p$.
\item The archimedean component $\Pi'_\infty$ of $\Pi'$ is cohomological with respect to $E^{\sf unt}_\lambda\otimes E^{\sf unt}_{\lambda^{\sf v}}$, i.e., \\ $H^*(\g',K',\Pi'_\infty\otimes (E^{\sf unt}_\lambda\otimes E^{\sf unt}_{\lambda^{\sf v}}))\neq 0$.
\end{itemize}
If all the isobaric summands $\Pi_j$ (resp.\ $\Pi'_j$) are cuspidal, then $\Pi_\infty$ (resp.\ $\Pi'_\infty$) is the local base change of $\pi_\infty$ (resp.\ $\pi'_\infty$).
\end{thm}
\begin{proof}
This is essentially, Labesse Cor.\ \cite{labesse-book} 5.3. Strictly speaking, the proof given in the aforementioned reference only works for totally real number fields $F$ different from $\Q$. However, there is no doubt that the theorem holds in the above form. See \cite{labesse-book}, Remarque 5.2.
\end{proof}

\subsection{Descent to unitary groups and base change}

When $BC(\pi)$ is cuspidal, it is expected, and has been proved in a great many cases (cf. \cite{labesse-book})
that $\pi$ occurs {\it with multiplicity one} in the discrete spectrum of $H$.  Moreover, if $H^{\flat}$ is any inner
form of $H$ then $\Pi$ descends to an $L$-packet  $\{\pi^{\flat}\}$ of discrete (cuspidal) automorphic
representations of $H^{\flat}$ whose archimedean components lie
in the discrete series; again it is expected that each member of the packet occurs with multiplicity one.   %(The statement of Corollary 11.2 in \cite{white} is
%a bit misleading:  it asserts that some member of the $L$-packet occurs with multiplicity one, but the first
%paragraph of the proof guarantees that this holds for every member of the $L$-packet when $\Pi$ is cuspidal.)

Suppose $H^{\flat}(\R) \cong U(r,s)$, the unitary group of signature $(r,s)$.  Then any $\pi^{\flat}$ on $H^{\flat}$
whose base change is $BC(\pi)$, and whose archimedean component is a holomorphic (discrete series) representation, will
be denoted $\pi_{r,s}$.  Thus, $\pi = \pi_{n,0}$ or $\pi_{0,n}$. It doesn't matter whether $H$ is
positive or negative definite; but when $rs > 0$, the difference between $\pi_{r,s}$ and $\pi_{s,r}$ needs to be
respected, because they define holomorphic automorphic forms on non-isomorphic Shimura varieties (attached
to the similitude groups containing $H^{\flat}$).  The normalizations relevant to these Shimura varieties
can be found in \cite{harcrelle}.

To $\pi_{r,s}$ we associate a period invariant $P^{(s)}(BC(\pi)) \in \C^{\times}$, as in \S 4.2 of \cite{haradj}.
 %, well-defined up to multiplication by a scalar in $E(BC(\pi))$.
Roughly speaking, $P^{(s)}(BC(\pi))$ is the square of the Petersson norm of an arithmetically normalized holomorphic
automorphic form in $\pi_{r,s}$.   Precise definitions are in \cite{harcrelle}; see also section \ref{unitaryperiods} below,
especially Hypothesis \ref{descent}, as well as Remark \ref{substitutes}.

\subsection{The cuspidal automorphic representation $\Pi$}\label{sect:Pi}
Recall the automorphic representation $BC(\pi)$ from Thm.\ \ref{thm:base_change}. We make the additional assumption that it is cuspidal and take an arbitrary, but henceforth fixed, integer ${\textbf{\textbf{\sf m}}}\in\Z$. We define
$$\Pi:=BC(\pi)\|\cdot\|^{\textbf{\sf m}},$$
suppressing its dependence on ${\textbf{\sf m}}$. By assumption, $\Pi$ is cuspidal automorphic and hence $\Pi$ is automatically globally $\psi^{-1}$-generic. We denote by $W(\Pi_f)$ the corresponding Whittaker model of $\Pi_f$, again suppressing the dependence on the fixed additive character $\psi^{-1}$. By Thm.\ \ref{thm:base_change}, $\Pi_\infty$ is cohomological with respect to the finite-dimensional, algebraic representation
$$E_{\mu}:=(E^{\sf unt}_\mu\cdot\textrm{det}^{-\textbf{\sf m}})\otimes (E^{\sf unt}_{\mu^{\sf v}}\cdot\textrm{det}^{-\textbf{\sf m}})=E^{\sf unt}_{\mu-{\textbf{\sf m}}}\otimes E^{\sf unt}_{\mu^{\sf v}-{\textbf{\sf m}}}$$
of $G(\C)$ of highest weight $\mu_G:=(\mu-{\textbf{\sf m}},\mu^{\sf v}-{\textbf{\sf m}})$, hence, we have
$$\Pi_\infty\cong {\rm Ind}_{B(\C)}^{G(\C)}[z_1^{\ell_1+{\textbf{\sf m}}}\bar z_1^{-\ell_1+{\textbf{\sf m}}}\otimes...\otimes z_{n}^{\ell_{n}+{\textbf{\sf m}}}\bar z_{n}^{-\ell_{n}+{\textbf{\sf m}}}],$$
where $B=TN$ is the standard Borel subgroup of $G$, and
$$\ell_j=-\mu_{n-j+1}+\frac{n+1}{2} -j.$$
In analogy to the case of unitary groups, we let $S(\Pi)$ be the set of finite places of $\K$, where $\Pi$ ramifies, together with the infinite place of $\K$.

\subsection{The abelian automorphic representation $\Pi'$}\label{sect:Pi'}
Now recall the automorphic representations $BC(\pi')$ from Thm.\ \ref{thm:base_change}. We set
$\Pi':=BC(\pi')$ and make the following additional assumption: We suppose that $\Pi'\cong \Pi'_1\boxplus ...\boxplus\Pi'_r$ is abelian (and therefore not square-integrable), i.e.,
$$\Pi'_j=BC(\chi_j)\cdot\gamma\quad\quad 1\leq j\leq n-1,$$
where $BC(\chi_j)$ is the base change of an algebraic Hecke character $\chi_j: U(1)(\Q)\backslash U(1)(\A)\ra\C^\times$. More explicitly, $BC(\chi_j)(g)=\chi_j(g/\eta(g))$ for all $g\in\A^\times_\K$. At the archimedean place of $\Q$ (resp.\ $\K$) we obtain $\chi_{j,\infty}(e^{i\theta})=e^{ik_j\theta}$, $e^{i\theta}\in U(1)$, (resp.\ $BC(\chi_j)_\infty(z)=z^{k_j} \bar z^{-k_j}$, $z\in\C^\times$) for $k_j\in\Z$. Therefore,
$$\Pi'_\infty\cong {\rm Ind}_{B'(\C)}^{G'(\C)}[z_1^{k_1+c}\bar z_1^{-k_1-c}\otimes...\otimes z_{n-1}^{k_{n-1}+c}\bar z_{n-1}^{-k_{n-1}-c}],$$
where $B'=T'N'$ is the standard Borel subgroup of $G'$ and so, since $\Pi'_\infty$ is cohomological with respect to the finite-dimensional, algebraic representation
$$E_\lambda:=E^{\sf unt}_\lambda\otimes E^{\sf unt}_{\lambda^{\sf v}}$$ of $G'(\C)$ of highest weight $\lambda_{G'}:=(\lambda,\lambda^{\sf v})$, necessarily,
$$k_j=-\lambda_{n-j}+\frac n2 -j-c.$$
Hence, we can further see that $k_1 > k_2 > ... > k_{n-1}$. In particular, this implies, using Shahidi \cite{shahidi-book} Prop.\ 7.1.3, Thm.\ 3.5.12 and Rem.\ 3.5.14, that $\Pi'$ is globally $\psi$-generic, i.e., the $\psi$-Whittaker coefficient of $\Pi'$ defines a non-vanishing intertwining $\Pi'\hra$ Ind$_{N'(\A)}^{G'(\A)}[\psi]$ (unnormalized induction). As for $\Pi$, we denote by $W(\Pi'_f)$ the corresponding Whittaker model of $\Pi'_f$ and let $S(\Pi')$ be the set of finite places of $\K$, where $\Pi'$ ramifies, together with the infinite place of $\K$.\\

We will make the following assumption, valid throughout the paper:

\begin{hyp}\label{hyp:coeff}
The highest weights $\mu$ and $\lambda$ satisfy
$$\mu_1\pm {\textbf{\sf m}}\geq -\lambda_{n-1}\geq \mu_2\pm {\textbf{\sf m}}\geq -\lambda_{n-2}\geq...\geq -\lambda_1\geq \mu_n\pm {\textbf{\sf m}}$$
\end{hyp}

Lem.\ \ref{lem:coeff} then guarantees that Hom$_{G'(\C)}(E_\mu\otimes E_\lambda,\C)\neq 0$, which will be important from Sect.\ \ref{sect:critical} on.

\subsection{An action of ${\rm Aut}(\C)$}\label{sect:autc-act}
For a moment, let $k\geq 1$ be any integer and let $E$ be an irreducible, finite-dimensional, algebraic representation of $GL_k/\K$. As a representation of the real Lie group $GL_k(\C)$, $E$ factors as $E=E_\iota\otimes E_{\bar\iota}$, for a fixed embedding $\iota:\K\hra\C$. For a given $\sigma\in$ Aut$(\C)$, we define the $GL_k(\C)$-representation
$${}^\sigma\! E:=E_{\sigma\circ\iota}\otimes E_{\overline{\sigma\circ\iota}}.$$
Hence, $^\sigma\! E$ is identical to $E$, if $\sigma\circ\iota=\iota$ and it is the representation ${}^\sigma\! E =E_{\bar\iota}\otimes E_\iota$, otherwise, i.e., if $\sigma\circ\iota=\bar\iota$. Furthermore, for a representation of $GL_k(\C)$, which is induced from the Borel subgroup $B_k=T_kN_k$ of $GL_k$, $\Delta_\infty={\rm Ind}_{B_k(\C)}^{GL_k(\C)}[\omega_1\otimes...\otimes\omega_k]$, we let
$${}^\sigma\Delta_\infty:={\rm Ind}_{B_k(\C)}^{GL_k(\C)}[{}^\sigma\!\omega_1\otimes...\otimes{}^\sigma\!\omega_k].$$
Note that if $\Delta_\infty$ is cohomological with respect to $E$, then ${}^\sigma\Delta_\infty$ is cohomological with respect to ${}^\sigma\! E$.\\
If $\Delta_f$ (resp.\ $\Delta_w$, $w\in S(\K)_f$) is an irreducible representation of $GL_k(\A_f)$ (resp.\ $GL_k(\K_w)$), then we denote by ${}^\sigma\Delta_f$ (resp.\ ${}^\sigma\Delta_w$) the representation defined by Waldspurger in \cite{waldsp}, I.1. Analogously, if $\delta_\infty$ is an irreducible representation of $H(\R)$ or $H'(\R)$, we let $^\sigma\! \delta_\infty:=\delta_\infty$, since there is only the identity embedding of $\Q$ into $\C$. For representations $\delta_f$ (resp.\ $\delta_p$, $p\in S_f$) of $H(\A_f)$ or $H'(\A_f)$ (resp.\ $H(\Q_p)$ or $H'(\Q_p)$), we use again the definition of Waldspurger, \cite{waldsp}, I.1.

If $\delta$ is any of the above representations, we denote by $\mathfrak S(\delta):=\{\sigma\in {\rm Aut}(\C)| {}^\sigma\delta\cong\delta\}$ and define the rationality field of $\delta$ to be
\begin{center}
\framebox{$\Q(\delta):=\{z\in\C| \sigma(z)=z \quad\forall \sigma\in\mathfrak S(\delta)\}.$}
\end{center}
See again Waldspurger, \cite{waldsp}, I.1. We say that a representation $\delta$ is defined over a field $F$, if there is an $F$-subspace $\delta_0$ of the representation space of $\delta$, which is invariant under the group action and such that the natural map $\delta_0\otimes_F\C\ra\delta$ is an isomorphism.

\begin{prop}\label{prop:sigmatwistabelian}
For all $\sigma\in$\emph{Aut}$(\C)$, there is an isomorphism
$${}^\sigma\Pi'\cong \boxplus_{j=1}^{n-1} {}^\sigma\!BC(\chi_j) {}^\sigma\!\!\left(\gamma\|\cdot\|^{\frac{n-2j}{2}}\right)\|\cdot\|^{\frac{2j-n}{2}}. $$
In other words, ${}^\sigma\Pi'$ is again the isobaric automorphic sum of Hecke characters. Similarly, ${}^\sigma\Pi$ is again a cuspidal automorphic representation.
\end{prop}
\begin{proof}
The first assertion results from a straight-forward calculation, using the fact that since $BC(\chi_j)$ and $\gamma\|\cdot\|^{\frac{n-2j}{2}}$ are both algebraic Hecke characters, so are ${}^\sigma\!BC(\chi_j)$ and ${}^\sigma\!\!\left(\gamma\|\cdot\|^{\frac{n-2j}{2}}\right)$. The last assertion on $\Pi$ can be shown as follows: By \cite{clozel} Thm.\ 3.13, there is a cuspidal automorphic representation $\Xi$ of $G(\A_\K)$, cohomological with respect to ${}^\sigma\! E_\mu$ and such that $\Xi_f\cong{}^\sigma\Pi_f$. As a cohomological cuspidal automorphic representation has an essentially tempered archimedean component, it follows from the classification of cohomological representations of $G(\C)$ that necesssarily $\Xi_\infty\cong{}^\sigma\Pi_\infty$.
\end{proof}

\begin{prop}\label{prop:regalg}
The finite part $\Pi_f$ of $\Pi=BC(\pi)\|\cdot\|^{\textbf{\sf m}}$ (resp.\ $\Pi'_f$ of $\Pi'=BC(\pi')$) is defined over its rationality field $\Q(\Pi_f)$ (resp.\ $\Q(\Pi'_f)$), which is an extension of $\Q(E_\mu)$ (resp.\ of $\Q(E_\lambda)$). This structure is unique up to homotheties. Both rationality fields are number fields. As a representation of $GL_k(\K)\subset GL_k(\C)$, an irreducible, finite-dimensional, algebraic representation $E$ of $GL_k(\C)$ is defined over $\Q(E)$.
\end{prop}
\begin{proof}
The fields $\Q(\Pi_f)$ and $\Q(\Pi'_f)$ extend $\Q(E_\mu)$, respectively $\Q(E_\lambda)$, because of Strong Multiplicity One for isobaric automorphic representations. Here we use Prop.\ \ref{prop:sigmatwistabelian}. Moreover, as a consequence of \cite{clozel} Thm.\ 3.13, $\Q(\Pi_f)$ is a number field. (For a detailed proof, one may also have a look at \cite{grob-rag-I}, Thm.\ 8.1.) By Prop.\ \ref{prop:sigmatwistabelian}
$$\bigcap_{j=1}^{n-1} \mathfrak S\left(BC(\chi_{j})_f\gamma_f\|\cdot\|_f^{\frac{n-2j}{2}}\right) \subseteq \mathfrak S(\Pi'_f),$$
so $\Q(\Pi'_f)\subseteq \prod_{j=1}^{n-1} \Q\left(BC(\chi_{j})_f\gamma_f\|\cdot\|_f^{\frac{n-2j}{2}}\right)$. As all $BC(\chi_{j})$ and $\gamma\|\cdot\|^{\frac{n-2j}{2}}$ are algebraic Hecke characters, $\prod_{j=1}^{n-1} \Q\left(BC(\chi_{j})_f\gamma_f\|\cdot\|_f^{\frac{n-2j}{2}}\right)$ is a number field, whence so is $\Q(\Pi'_f)$. The rest follows now from Clozel \cite{clozel}, Prop.\ 3.1 and p.\ 122.
\end{proof}

\begin{rem}\label{rem:induced}  The representation $\Pi'_f$ is induced from the Hecke character
$\beta := \Pi'_1\otimes \dots \otimes \Pi'_r$ of the (Levi quotient $T'(\A_f)$ of the) Borel subgroup $B'(\A_f)$, and can thus be
identified with a space of functions on $G'(\A_f)$:
$$\Pi'_f = \{\phi:  G'(\A_f) \ra \C~|~ \phi(b'g') = \beta(b')\cdot \delta_{B'}^{\frac{1}{2}}(b')\cdot\phi(g'), ~b' \in B'(\A_f), g' \in G'(\A_f)\}.$$
Here $\delta_{B'}$ is the modulus character of $B'(\A_f)$.   Define $\gamma^+(x) = \gamma(x)\cdot \|x\|^{\frac{1}{2}}$ if $n$ is odd,
$\gamma^+(x) = 1$ if $n$ is even.   For any $n$, $(\Pi'_j)^+ = BC(\chi_j)\cdot \gamma^+$ is an algebraic Hecke character, defined over
the number field $\Q((\Pi'_j)^+)$.  On the other hand, the $j$-th component of $\beta\cdot \delta_{B'}^{\frac{1}{2}}$ equals
$BC(\chi_j)\cdot\gamma\|\cdot\|^{\frac{n-2j}{2}}$ which is $(\Pi'_j)^+$ multiplied by an integral power of the norm character.
Thus $\beta\cdot\delta_{B'}^{\frac12}$ takes values in a number field, say $\Q(\beta,\gamma^+)$.  It follows that the subspace
$$\Pi'_f[\Q(\beta,\gamma^+)] = \{\phi:  G'(\A_f) \ra \Q(\beta,\gamma^+) ~|~ \phi(b'g') = \beta(b')\cdot \delta_{B'}^{\frac{1}{2}}(b')\cdot\phi(g'), ~b' \in B'(\A_f), g' \in G'(\A_f)\}$$
is a $G'(\A_f)$-invariant $\Q(\beta,\gamma^+)$-rational structure on $\Pi'_f$,
and it is clear that the action of Aut$(\C)$ on $\Pi'_f$ can be read directly in terms of the actions on the values
of functions.
\end{rem}

Next, let ${\bf t}_{\sigma,k}$ be the diagonal matrix $diag(t_{\sigma}^{-(k-1)}, t_{\sigma}^{-(k-2)},...,1)\in GL_k(\A_f)$, as in \cite{raghuram-shahidi-imrn}, 3.2 or \cite{mahnk}, 3.3. Here, $t_\sigma$ is an element of $\widehat\O^*$, assigned to $\sigma\in$ Aut$(\C)$ by the map Aut$(\C)\ra$ Gal$(\Q(\mu_\infty)/\Q)\ra\widehat\Z^*\hra\widehat\O^*$. Let $\xi_f\in {\rm Ind}_{N_k(\A_f)}^{GL_k(\A_f)}[\phi_f]$ (unnormalized induction), $\phi$ an additive, continuous character $\K\backslash\A_\K\ra\C^\times$. Then, ${}^\sigma\!\xi_f(g_f):=\sigma(\xi({\bf t}_{\sigma,k} g_f))$ is again an element in ${\rm Ind}_{N_k(\A_f)}^{GL_k(\A_f)}[\phi_f]$. The analogous definition applies locally at a finite place $w\in S(\K)_f$.

\begin{prop}\label{lem:rational-whittaker}
The map $\xi_f\mapsto {}^\sigma\xi_f$ defines a $\sigma$-linear $G(\A_f)$-equivariant isomorphism from $W(\Pi_f)$ onto $W({}^\sigma\Pi_f)$ as well as a $\sigma$-linear $G'(\A_f)$-equivariant isomorphism from $W(\Pi'_f)$ onto $W({}^\sigma\Pi'_f)$. For any finite extension $F$ of the rationality field in question, we have
an $F$-structure on $W(\Pi_f)$ and $W(\Pi'_f)$ by taking the ${\rm Aut}(\C/F)$-invariants.
\end{prop}
\begin{proof}
For the cuspidal representation $\Pi$, this is \cite{raghuram-shahidi-imrn}, Lem.\ 3.2. The same proof goes through for $\Pi'$, using Jacquet--Piatestski-Shapiro--Shalika's theorem (5.1.(i)) in \cite{jac-ps-shalika-mathann}.
\end{proof}

\subsection{Rational structures on cohomological Harish-Chandra modules} \label{rationalgk}
The universal enveloping algebra $\mathcal U(\g')$ has a canonical
$\Q$-rational (vector space-)structure determined by the $\Q$-reductive group $R_{\K/\Q}(G')$.  Moreover, $K'$ has a compatible
$\Q$-rational structure, and determines a $\Q$-rational Cartan decomposition $\g' = \k'\oplus \p'$.
It thus makes sense to define the field of rationality
of the $(\mathcal U(\g'),K')$-module $\Pi'_\infty$ as the field $\Q(\Pi'_{\infty})$ fixed by the action of
the subgroup of Aut$(\C)$ that fixes $\Pi'_\infty$ up to isomorphism.  Now $\Pi'_{\infty}$ is the $\mathcal U(\g')$-module
induced from a $\Q(E_\lambda)$-rational character of the Lie algebra of the Borel subgroup $B'(\C)$.  In fact, it is given by
non-normalized induction from an algebraic
character of the torus $T'(\C)$, by the arguments already seen in \ref{rem:induced}. It follows that the subspace of $\Pi'_{\infty}$ isotypic
for the representation $E_\lambda$ of $K'$ is defined over the number field $\Q(E_\lambda)$.  Using the $\Q$-rational structure on
$\p'$, we find that the relative Lie algebra cohomology complex
$C^{\bullet}(\g',K', \Pi'_\infty\otimes E_\lambda) = {\rm Hom}_{K'}(\Lambda^{\bullet}(\p'), \Pi'_\infty\otimes E_\lambda)$ has a natural
$\Q(E_{\lambda})$-rational (vector space-)structure. Hence, the same holds true for $(\g',K')$-cohomology.

Although the cuspidal automorphic representation $\Pi$ is not globally induced, it follows from
the classification of cohomological representations that the archimedean component
$\Pi_{\infty}$ is again isomorphic to the non-normalized induction to $G(\C)$ of an algebraic character of the maximal torus $T(\C)$, see Sect.\ \ref{sect:Pi}.  We find again that $C^{\bullet}(\g,K, \Pi_\infty\otimes E_\mu)= {\rm Hom}_{K}(\Lambda^{\bullet}(\p), \Pi_\infty\otimes E_\mu)$ -- and so also the space of $(\g,K)$-cohomology -- has a natural $\Q(E_{\mu})$-rational (vector space-)structure.

In subsequent chapters we will be working with the Whittaker models $W(\Pi_{\infty})$ and $W(\Pi'_{\infty})$.  It follows
from the above discussion that their subspaces of $K$ (resp. $K'$)-finite vectors have rational models over the appropriate fields.
Choices of complex isomorphisms $i_{\infty}: \Pi_{\infty} \overset{\sim}\ra W(\Pi_{\infty})$ and $i'_{\infty}: \Pi'_{\infty} \overset{\sim}\ra W(\Pi'_{\infty})$ identify
the rational models of the two sides, up to complex factors of proportionality that depend on the choices.  In the proof of Proposition \ref{WhitEis2}, the
factor of proportionality is denoted $\Omega(\Pi'_{\infty})$; it is attached to the explicit choice of  $i'_{\infty}$ defined by the Whittaker integral.
It is tempting to use the Whittaker integral to define $i_{\infty}$ as well.  This would provide natural normalizations for all the archimedean constants
that appear in our final formulas.  However, other normalizations  -- for example, normalizations in terms of the archimedean local zeta integrals --
may turn out to be more natural.

\section{Whittaker periods for the general linear group}
\subsection{Automorphic cohomology of locally symmetric spaces}\label{sect:autcoh}
We define
$$S_n:= G(\K)\backslash G(\A_\K) / K\quad\quad\textrm{and}\quad\quad S_{n-1}:= G'(\K)\backslash G'(\A_\K)/ K'$$
Moreover, we let
$$\tilde S_{n-1}:= G'(\K)\backslash G'(\A_\K)/ U(n-1)$$
Observe that $K\cong U(n)\C^\times=U(n)\R_+$ and similarly $K'\cong U(n-1)\C^\times=U(n-1)\R_+$, where $\R_+$ denotes the topological connected component of the identity of the split component of the center of $G(\C)$ and $G'(\C)$. In this way, the group $U(n-1)$ in the definition of $\tilde S_{n-1}$ is $K'$ without the contribution of the center. Consider the map $\iota: G'\hra G$, realizing an element $g'\in G'$ as block diagonal matrix $diag(g',1)\in G$. If $K_f$ is an open compact subgroup of $G(\A_f)$, $K'_f:=\iota^{-1}(K_f)$ is open compact in $G'(\A_f)$ and we define
$$S_n(K_f):= G(\K)\backslash G(\A_\K) /K K_f\quad\quad\textrm{and}\quad\quad S_{n-1}(K'_f):= G'(\K)\backslash G'(\A_\K)/ K' K'_f$$
and $\tilde S_{n-1}(K'_f):=G'(\K)\backslash G'(\A_\K)/ U(n-1) K'_f$.
Our finite-dimensional modules $E_\mu$ and $E_\lambda$ naturally define sheaves $\E_\mu$ on $S_n$ and $\E_\lambda$ on $S_{n-1}$, respectively, and we let
$$H^q(S_n,\E_\mu)\quad\quad\textrm{and}\quad\quad H^q(S_{n-1},\E_\lambda)$$
be the corresponding cohomology spaces. They carry a $G(\A_f)$-, resp.\ a $G'(\A_f)$-module structure. With respect to this module-structure, there are isomorphisms
$$H^q(S_n,\E_\mu)\cong H^q(\g,K,\mathcal A_{\J}(G)\otimes E_\mu)$$
and
$$H^q(S_{n-1},\E_\lambda)\cong H^q(\g',K',\mathcal A_{\J'}(G')\otimes E_\lambda).$$
This needs some explanation: First, the space $\mathcal A_{\J}(G)$ (resp.\ $\mathcal A_{\J'}(G')$) denotes the space of all automorphic forms of $G(\A_\K)$ (resp.\ $G'(\A_\K)$) which are annihilated by some power of $\J$ (resp.\ $\J'$). Here, $\J$ (resp.\ $\J'$) is the ideal of the center of the universal enveloping algebra $\mathcal U(\g_\C)$ (resp.\ $\mathcal U(\g'_\C)$), which is annihilated by the contragredient representation of $E_\mu$ (resp.\ $E_\lambda$). Now, since $U(n)\C^\times=U(n)\R_+$ and $U(n-1)\C^\times=U(n-1)\R_+$, the assertion follows from Franke, \cite{franke}, Thm.\ 18. \\

From this we obtain decompositions
$$H^q(S_n,\E_\mu)\cong H^q_{cusp}(S_n,\E_\mu)\oplus H^q_{Eis}(S_n,\E_\mu)$$
and
$$H^q(S_{n-1},\E_\lambda)\cong H^q_{cusp}(S_{n-1},\E_\lambda)\oplus H^q_{Eis}(S_{n-1},\E_\lambda)$$
as follows:
The space $\mathcal A_{\J}(G)$ may be decomposed along the associate classes $\{P\}$ of parabolic $\K$-subgroups $P=LN$ of $G$ and the cuspidal supports $\varphi_P$, i.e., associate classes of irreducible cuspidal automorphic representations $\tau=\tilde\tau e^{\<d\Lambda,H_P(.)\>}$ of $L(\A_\K)$, with $\tilde\tau$ a unitary cuspidal automorphic representation and $\Lambda:\R_+^{{\sf rank}(P)}\rightarrow\C^\times$ a Lie group character of the split component $\R_+^{{\sf rank}(P)}$ of $L(\C)$. Indeed, if $\mathcal A_{\J,\{P\}}(G)$ denotes the space of all automorphic forms in $\mathcal A_{\J}(G)$, which are negligible along every parabolic subgroup $Q\notin\{P\}$, and, moreover, if $\mathcal A_{\J,\{P\},\varphi_P}(G)$ denotes the subspace of $\mathcal A_{\J,\{P\}}(G)$, which is generated as a $G(\A)$-module by all possible holomorphic values or residues of all Eisenstein series attached to $\tilde\tau$, evaluated at the point $d\Lambda$, together with all their derivatives, then
$$\mathcal A_{\J}(G) = \bigoplus_{\{P\}}\mathcal A_{\J,\{P\}}(G) =\bigoplus_{\{P\},\varphi_P}\mathcal A_{\J,\{P\},\varphi_P}(G).$$
For a detailed description of this decomposition, we refer the reader to the original paper, Franke-Schwermer \cite{schwfr}, 1.1--1.4.
If we set
$$H^q_{cusp}(S_n,\E_\mu):=H^q(\g,K,\mathcal A_{\J,\{G\}}(G)\otimes E_\mu)$$
$$H^q_{Eis}(S_{n},\E_\mu):=\bigoplus_{\{P\}\neq\{G\}}H^q(\g',K',\mathcal A_{\J,\{P\}}(G)\otimes E_\mu)$$
we obtain the above decomposition of cohomology for $G$. We remark that because $\mathcal A_{\J,\{G\}}(G)$ consists precisely of all cuspidal automorphic forms in $\mathcal A_{\J}(G)$, $H^q_{cusp}(S_n,\E_\mu)$ is called the cuspidal cohomology of $G$ (with respect to $E_\mu$), and since $\mathcal A_{\J,\{P\}}(G)$ is defined by means of Eisenstein series, supported in $\{P\}$, $H^q_{Eis}(S_{n},\E_\mu)$ is called the Eisenstein cohomology of $G$ (with respect to $E_\mu$). Clearly, putting a prime everywhere, gives the analogous result for the cohomology of $G'$.\\

Let $H^q_c(S_n,\E_\mu)$ (resp.\ $H^q_c(\tilde S_{n-1},\E_\lambda)$) be the cohomology with compact support. As cusp forms are rapidly decreasing, one has $H^q_{cusp}(S_n,\E_\mu)\subseteq H^q_c(S_n,\E_\mu)$. Put $b_k:=\frac{k(k-1)}{2}$. Then $b_k$ is the smallest degree in which a cohomological generic automorphic representation of $GL_k(\A_\K)$ has non-zero cohomology. Note that this degree is independent of the given representation, as well as of the finite-dimensional coefficient module. Furthermore,
$$H^{b_{n}}(\g,K,W(\Pi_\infty)\otimes E_\mu)\cong \C$$
and also
$$H^{b_{n-1}}(\g',K',W(\Pi'_\infty)\otimes E_\lambda)\cong \C.$$

\subsection{A diagram}
Similar to the ideas in Mahnkopf \cite{mahnk}, Raghuram \cite{raghuram-imrn}, \cite{ragh-gln} and Grobner-Raghuram \cite{grob-ragh}, we are going to consider the following diagram:

$$
\xymatrix{
H^{b_n}_c(S_n,\E_\mu)\times H^{b_{n-1}}(S_{n-1},\E_\lambda)\ar[r]^{\iota\times p} & H^{b_n}_c(\tilde S_{n-1},\E_\mu)\times H^{b_{n-1}}(\tilde S_{n-1},\E_\lambda)\ar[d]^{\wedge}\\
H^{b_n}_{cusp}(S_n,\E_\mu)\times H^{b_{n-1}}_{Eis}(S_{n-1},\E_\lambda)\ar@{^{(}->}[u] & H^{b_n+b_{n-1}}_c(\tilde S_{n-1},\E_\mu\otimes\E_\lambda)\ar[d]^{\mathcal T^*}\\
H^{b_{n}}(\g,K,\Pi\otimes E_\mu)\times H^{b_{n-1}}(\g',K',\Pi'\otimes E_\lambda) \ar[u]^{\Psi=\Psi^{\sf cusp}_{\Pi}\times\Psi^{\sf Eis}_{\Pi'}}& H^{b_n+b_{n-1}}_c(\tilde S_{n-1},\C)\ar[d]^{\int}\\
W(\Pi_f)\times W(\Pi'_f)\ar[u]^{\Theta_0=\Theta^{\sf cusp}_0\times\Theta^{\sf Eis}_0}\ar[r]^{\rm Dia} & \C}
$$
We are going to define the various maps, which appear in this diagram, in the next sections. Observe that ``Dia'' denotes the composition of all these maps.

\subsection{The map $\Theta_0$}\label{sect:Theta}
This section is in analogy with Mahnkopf \cite{mahnk}, Raghuram-Shahidi, \cite{raghuram-shahidi-imrn}, Raghuram \cite{raghuram-imrn}, \cite{ragh-gln} and Grobner-Raghuram \cite{grob-ragh}.

As a first step, we choose and fix generators of the one-dimensional spaces $H^{b_{n}}(\g,K,W(\Pi_\infty)\otimes E_\mu)$ and $H^{b_{n-1}}(\g',K',W(\Pi'_\infty)\otimes E_\lambda)$. They are of the form
$$
[\Pi_\infty]:=\sum_{\underline i=(i_1,...,i_{b_n})}\sum_{\alpha=1}^{\dim E_\mu} X^*_{\underline i}\otimes\xi_{\Pi_\infty,\underline i, \alpha}\otimes e_\alpha,
$$
and
$$
[\Pi'_\infty]:=\sum_{\underline j=(j_1,...,j_{b_{n-1}})}\sum_{\beta=1}^{\dim E_\lambda} X'^*_{\underline j}\otimes\xi_{\Pi'_\infty,\underline j, \beta}\otimes e'_\beta,
$$
where the following data has been fixed:
\begin{enumerate}
\item A basis $\{X_j\}$ of $\g/\k$, which fixes the dual-basis $\{X_j^*\}$ for $\left(\g/\k\right)^*$. By our concrete choice of $K$, $\k$ is defined over $\Q$, whence we may assume that $\{X_j\}$ is a $\Q$-basis.
For $\underline i=(i_1,...,i_{b_n})$, let $X^*_{\underline i}=X^*_{i_1}\wedge...\wedge X^*_{i_{b_n}}\in \bigwedge^{b_n} \left(\g/\k\right)^*.$
\item Elements $e_1,...,e_{\dim E_\mu}$, which form a $\Q(E_\mu)$-basis of $E_\mu$.
\item To each $\underline i$ and $\alpha$, $\xi_{\Pi_\infty,\underline i, \alpha}\in W(\Pi_\infty)$.
\end{enumerate}
and
\begin{enumerate}
\item A basis $\{X'_j\}$ of $\g'/\k'$, which fixes the dual-basis $\{X_j'^*\}$ for $\left(\g'/\k'\right)^*$. Similar to the case above, by our concrete choice of $K'$, $\k'$ is defined over $\Q$, whence we may assume that $\{X'_j\}$ is a $\Q$-basis.
For $\underline j=(j_1,...,j_{b_{n-1}})$, let $X'^*_{\underline j}=X'^*_{j_1}\wedge...\wedge X'^*_{j_{b_{n-1}}}\in \bigwedge^{b_{n-1}} \left(\g'/\k'\right)^*.$
\item Elements $e'_1,...,e'_{\dim E_\lambda}$, which form a $\Q(E_\lambda)$-basis of $E_\lambda$.
\item To each $\underline j$ and $\beta$, $\xi_{\Pi'_\infty,\underline j, \beta}\in W(\Pi'_\infty)$.
\end{enumerate}

We may assume that the bases are compatible in the way, that $\{X_j\}$ extends $\{X'_j\}$ (along the embedding $\g'/\k'\hra\g/\k$ defined by $\iota$) and that the non-compact part of the center of $\g'$ is spanned by $X_r$, $r=\dim_\R\tilde S_{n-1}$.
The choice of the generators $[\Pi_\infty]$ and $[\Pi'_\infty]$ fixes generators $[{}^\sigma\Pi_\infty]$ and $[{}^\sigma\Pi'_\infty]$ of the one-dimensional spaces $H^{b_{n}}(\g',K', W({}^\sigma\Pi_\infty) \otimes {}^\sigma\!E_\mu)$ and $H^{b_{n-1}}(\g',K', W({}^\sigma\Pi'_\infty) \otimes {}^\sigma\!E_\lambda)$ for all $\sigma\in$ Aut$(\C)$ as well as isomorphisms $\Theta^{\sf cusp}$ and $\Theta^{\sf Eis}$:
\begin{eqnarray*}
 W(\Pi_f) &\ira & W(\Pi_f) \otimes H^{b_n}(\g,K, W(\Pi_\infty) \otimes E_\mu)\\
&\ira & H^{b_n}(\g,K, W(\Pi) \otimes E_\mu)\\
&\ira & H^{b_n}(\g,K, \Pi \otimes E_\mu),
\end{eqnarray*}
and
\begin{eqnarray*}
W(\Pi'_f)&\ira & W(\Pi'_f)\otimes H^{b_{n-1}}(\g',K',W(\Pi'_\infty)\otimes E_\lambda) \\
&\ira & H^{b_{n-1}}(\g',K',W(\Pi')\otimes E_\lambda)\\
&\ira & H^{b_{n-1}}(\g',K',\Pi'\otimes E_\lambda),
\end{eqnarray*}
where the last map in the respective diagram denotes the inverse of the corresponding Fourier coefficient. Next, recall from Prop.\ \ref{prop:regalg} that the $G(\A_f)$-module $H^{b_n}(\g,K, \Pi \otimes E_\mu)$ is defined over $\Q(\Pi_f)$ and that the $G'(\A_f)$-module $H^{b_{n-1}}(\g',K',\Pi'\otimes E_\lambda)$ is defined over $\Q(\Pi'_f)$. Both structures are unique up to multiplication by non-zero complex numbers. This leads us to the following

\begin{prop}[The Whittaker-periods]
\label{defprop}
There are non--zero complex numbers $p(\Pi)=p(\Pi_f,[\Pi_\infty])$ and $p(\Pi')=p(\Pi'_f,[\Pi'_\infty])$, such that the normalized maps
\begin{equation}\label{eq:Theta_0}
\Theta^{\sf cusp}_0:=p(\Pi)^{-1}\cdot\Theta^{\sf cusp}\quad\quad\textrm{and }\quad\quad \Theta^{\sf Eis}_0:=p(\Pi')^{-1}\cdot\Theta^{\sf Eis}
\end{equation}
are ${\rm Aut}({\mathbb C})$-equivariant, i.e.,
$$
\xymatrix{
W(\Pi_f) \ar[rrrr]^{\Theta^{\sf cusp}_0}\ar[d] & & & &
H^{b_n}(\g,K, \Pi \otimes E_\mu)
\ar[d]\\
W({}^\sigma\Pi_f)
\ar[rrrr]^{\Theta^{\sf cusp}_0}
& & & &
H^{b_n}(\g,K, {}^\sigma\Pi \otimes {}^\sigma\!E_\mu)
}
$$
and
$$
\xymatrix{
W(\Pi'_f) \ar[rrrr]^{\Theta^{\sf Eis}_0}\ar[d] & & & &
H^{b_{n-1}}(\g',K', \Pi' \otimes E_\lambda)
\ar[d]\\
W({}^\sigma\Pi'_f)
\ar[rrrr]^{\Theta^{\sf Eis}_0}
& & & &
H^{b_{n-1}}(\g',K', {}^\sigma\Pi' \otimes {}^\sigma\!E_\lambda)
}
$$
commute. The complex number $p(\Pi)$ (resp.\ $p(\Pi')$) is well-defined only up to multiplication by non-zero elements of the number field $\Q(\Pi_f)$ (resp.\ $\Q(\Pi'_f)$).
\end{prop}
\begin{proof}
For $\Pi$ this is shown in \cite{raghuram-shahidi-imrn}, Definition/Proposition 3.3. In order to obtain the result for $\Pi'$, we observe that the proof of \cite{raghuram-shahidi-imrn}, Definition/Proposition 3.3, resp.\ \cite{grob-ragh} Definition/Proposition 4.2.1 goes over word for word, keeping in mind our Prop.\ \ref{lem:rational-whittaker} and Jacquet--Piatestski-Shapiro's Thm.\ (5.1.(i)) in \cite{jac-ps-shalika-mathann}.
\end{proof}

Finally, we set
$$\Theta_0:=\Theta^{\sf cusp}_0\times \Theta^{\sf Eis}_0.$$

\subsection{The map $\Psi=\Psi^{\sf cusp}_{\Pi}\times\Psi^{\sf Eis}_{\Pi'}$}\label{sect:cuspEisinj}
It is well-known and follows from our r\'esum\'e in section \ref{sect:autcoh} together with Multiplicity One that
$$H^{b_n}(\g,K,\Pi\otimes E_\mu) = H^{b_n}(\g,K,\mathcal A_{\J,\{G\},\varphi_\Pi}\otimes E_\mu),$$
where $\varphi_\Pi$ is a singleton, represented by the cuspidal automorphic representation $\Pi$. Hence, there is a natural embedding
$$\Psi^{\sf cusp}_{\Pi}: H^{b_n}(\g,K,\Pi\otimes E_\mu)\hra H^{b_n}_{cusp}(S_n,\E_\mu).$$
It is the purpose of this section to construct an embedding
$$\Psi^{\sf Eis}_{\Pi'}: H^{b_{n-1}}(\g',K',\Pi'\otimes E_\lambda)\hra H^{b_{n-1}}_{Eis}(S_{n-1},\E_\lambda),$$
as well. This is more delicate. First, as a short remark, let us point out that -- in contrast to the case of cuspidal cohomology -- this is also a question of degrees of cohomology: For a {\it non-trivial} $G'(\A_f)$-morphism $H^{b_{n-1}}(\g',K',\Pi'\otimes E_\lambda)\ra H^{b_{n-1}}_{Eis}(S_{n-1},\E_\lambda)$ to exist, it is necessary that
$$q_{\textbf{\sf max}}:=\min_{\substack{P'= \textrm{ \emph{a max.}}\\ \emph{parabolic}/\K}}\left(\tfrac12 \dim_\R N_{P'}(\C)\right)\leq b_{n-1}.$$
This follows from Grobner \cite{grobner-EisRes}, Thm.\ 18, since $\Pi'$ is not square-integrable. Since for any maximal parabolic subgroup $P'$ of $G'$, the unipotent radical $N_{P'}$ is of dimension $\dim_\C N_{P'}(\C)=m_1\cdot m_2$, for some non-trivial partition $m_1+m_2=n-1$, we have without loss of generality $m_1\leq \tfrac{n-1}{2}$ and $m_2\leq n-2$. Hence,
$$q_{\textbf{\sf max}}=\tfrac12 \dim_\R N_{P'}(\C)= m_1\cdot m_2\leq \tfrac{(n-1)(n-2)}{2}=b_{n-1},$$
and therefore the above condition is always satisfied. We now construct such a non-trivial map and show that is an injection.\\\\

Since $b_{n-1}$ is the minimal degree, in which $\Pi'_\infty\cong {\rm Ind}_{B'(\C)}^{G'(\C)}[z_1^{k_1+c}\bar z_1^{-k_1-c}\otimes...\otimes z_{n-1}^{k_{n-1}+c}\bar z_{n-1}^{-k_{n-1}-c}]$ has non-vanishing $(\g',K')$-cohomology with respect to $E_\lambda$, there is an isomorphism of one-dimensional $\C$-vector-spaces:
$$H^{b_{n-1}}(\g',K',{\rm Ind}_{B'(\C)}^{G'(\C)}[z_1^{k_1+c}\bar z_1^{-k_1-c}\otimes...\otimes z_{n-1}^{k_{n-1}+c}\bar z_{n-1}^{-k_{n-1}-c}]\otimes E_\lambda)$$
$$\ira H^{b_{n-1}}(\g',K',{\rm Ind}_{B'(\C)}^{G'(\C)}[z_1^{k_1+c}\bar z_1^{-k_1-c}\otimes...\otimes z_{n-1}^{k_{n-1}+c}\bar z_{n-1}^{-k_{n-1}-c}\otimes S(\check\a^{G'}_{B',\C})]\otimes E_\lambda),$$
where $S(\check\a^{G'}_{B',\C})$ is the symmetric algebra of the orthogonal complement $\check\a^{G'}_{B',\C}$ of $\check\a_{G',\C}=X^*(G')\otimes_\Z\C$ in $\check\a_{B',\C}=X^*(T')\otimes_\Z\C$. Hence, if we write $\tilde\tau:=BC(\chi_1)\gamma\otimes...\otimes BC(\chi_{n-1})\gamma$, then
$$H^{b_{n-1}}(\g',K',\Pi'\otimes E_\lambda)\cong H^{b_{n-1}}(\g',K',{\rm Ind}_{B'(\A_\K)}^{G'(\A_\K)}[\tilde\tau\otimes S(\check\a^{G'}_{B',\C})]\otimes E_\lambda).$$
The algebra $\a^{G'}_{B',\C}$ operates trivially on $\tilde\tau$. Hence, one may check that $(B',\tilde\tau,0,0)$ is one of the quadruples, constructed in Grobner \cite{grobner-EisRes}, 3.3. Let $\varphi_{B'}$ be the associate class of unitary cuspidal automorphic representations of $T'(\A_\K)$, represented by $\tilde\tau$, so the summand $\mathcal A_{\J',\{B'\},\varphi_{B'}}(G')$ of the space of automorphic forms $\mathcal A_{\J'}(G')$, cf.\ Sect.\ \ref{sect:autcoh}, is well-defined. Interpreting the elements of the symmetric algebra as differential operators $\frac{\partial^m}{\partial\Lambda^m}$, we obtain an intertwining operator
$$Eis_{\Pi'}:{\rm Ind}_{B'(\A_\K)}^{G'(\A_\K)}[\tilde\tau\otimes S(\check\a^{G'}_{B',\C})]\ra \mathcal A_{\J',\{B'\},\varphi_{B'}}(G')$$
$$f\otimes \frac{\partial^m}{\partial\Lambda^m}\mapsto \frac{\partial^m}{\partial\Lambda^m}\left(E(f,\Lambda)\right)|_{\Lambda=0},$$
where $E(f,\Lambda)$ is the Eisenstein series attached to a $K'$-finite section $f\in {\rm Ind}_{B'(\A_\K)}^{G'(\A_\K)}[\tilde\tau]$ at $\Lambda\in\a^{G'}_{B',\C}$. Observe that this intertwining operator is well-defined, since all Eisenstein series are holomorphic at $\Lambda=0$.

\begin{prop}
The $G'(\A_f)$-homomorphism
$$\Psi_{\Pi'}: H^{b_{n-1}}(\g',K',{\rm Ind}_{B'(\A_\K)}^{G'(\A_\K)}[\tilde\tau\otimes S(\check\a^{G'}_{B',\C})]\otimes E_\lambda) \ra H^{b_{n-1}}(\g',K',\mathcal A_{\J',\{B'\},\varphi_{B'}}(G')\otimes E_\lambda)$$
induced from $Eis_{\Pi'}$ is an injection.
\end{prop}
\begin{proof}
As all Eisenstein series $E(f,\Lambda)$ attached to a $K'$-finite section $f\in {\rm Ind}_{B'(\A_\K)}^{G'(\A_\K)}[\tilde\tau]$ are holomorphic at $\Lambda=0$, $\Psi_{\Pi'}$ is injective by Schwermer \cite{schwLNM}, Satz 4.11 or Li-Schwermer \cite{lischw}, Thm.\ 3.3.
\end{proof}

As a consequence of our discussion in Sect.\ \ref{sect:autcoh}, we obtain an injection of $G'(\A_f)$-modules
$$\Psi^{\sf Eis}_{\Pi'}: H^{b_{n-1}}(\g',K',\Pi'\otimes E_\lambda) = H^{b_{n-1}}(\g',K',\Pi'_{\infty}\otimes E_\lambda)\otimes \Pi'_f \hra H^{b_{n-1}}_{Eis}(S_{n-1},\E_\lambda).$$
With respect to the $\Q(E_{\lambda})$-rational structure on relative Lie algebra cohomology defined in \ref{rationalgk}, this defines a $\Q(\Pi'_f)$-rational injection
$$\Pi'_f = {\rm Ind}_{B'(\A_f)}^{G'(\A_f)}[ \tilde{\tau}_f]  \hra H^{b_{n-1}}_{Eis}(S_{n-1},\E_\lambda)$$
where, for any field $L \supseteq \Q(\Pi'_f)$, the $L$-rational vectors on the left-hand side are just the $L$-valued functions on $G'(\A_f)$ that transform
on the left under the character $\delta^{\frac{1}{2}}_{B',f}\cdot \tilde{\tau}_f$.

\subsection{The map $\iota\times\phi$}
Recall the map $\iota: G'\hra G$, realizing an element $g'\in G'$ as block diagonal matrix $diag(g',1)\in G$.

\begin{lem}
The map $\iota: \tilde S_{n-1}(K'_f)\ra S_n(K_f)$ is proper.
\end{lem}
\begin{proof}
For the general linear group over $\Q$, this is stated in Mahnkopf \cite{mahnk}, p.\ 615, however, without proof. For sake of completeness, we sketch an argument here. Let $$S_{n-1,n}(K_f):=M(F)\backslash M(\A)/(K\cap M(\C))(K_f\cap M(\A_f)),$$
where $M\cong GL_{n-1}\times GL_1$, viewed as block diagonal matrices. The map $\iota$ factors as
$$\xymatrix{\tilde S_{n-1}(K'_f) \ar[r]^j &  S_{n-1,n}(K_f)\ar[r]^u & S_n(K_f)}.$$
Clearly, $j$ is proper. Hence, it suffices to show that $u$ is proper. This follows from \cite{ash}, Lem.\ 2.7.
\end{proof}

As a consequence, we obtain a map in cohomology with compact support:
$$\iota^q: H^q_c(S_n(K_f),\E_\mu)\ra H^q_c(\tilde S_{n-1}(K'_f),\E_\mu).$$
Similar, the projection $p: \tilde S_{n-1}(K'_f)\twoheadrightarrow S_{n-1}(K'_f)$ induces a map
$$p^q: H^q(S_{n-1}(K'_f),\E_\lambda)\ra H^q(\tilde S_{n-1}(K'_f),\E_\lambda).$$
In our diagram, we let $\iota\times p$ be the direct limit (over all open compact subgroups $K_f$ of $G(\A_f)$) of the maps $\iota^{b_n}\times p^{b_{n-1}}$.

\subsection{Critical points}\label{sect:critical}
Recall the definition of a point $s=\tfrac12+m$, $m\in\Z$ being {\it critical} for $L(s,\Pi\times\Pi')$ from Deligne, \cite{deligne} Prop.\--D\'ef.\ 2.3. We let
$$\textrm{Crit}(\Pi\times\Pi')\subset\tfrac12+\Z$$
be the set of critical points of $L(s,\Pi\times\Pi')$.

\begin{lem}\label{lem:1/2critical}
If $s = \tfrac12 \pm m$,  $m\geq 0$, is a critical point of $L(s,\Pi\times\Pi')$, then \emph{Hom}$_{H'(\C)}(E^{\sf unt}_{\mu- m-\textbf{\sf m}}\otimes E^{\sf unt}_\lambda,\C)\neq 0$ and \emph{Hom}$_{H'(\C)}(E^{\sf unt}_{\mu^{\sf v}- m-\textbf{\sf m}}\otimes E^{\sf unt}_{\lambda^{\sf v}},\C)\neq 0$. The set of critical points is $\textrm{\emph{Aut}}(\C)$-invariant, i.e., $\textrm{\emph{Crit}}(\Pi\times\Pi')=\textrm{\emph{Crit}}({}^\sigma\Pi\times {}^\sigma\Pi')$.
\end{lem}
\begin{proof}
A proof of these facts will be given in Sect.\ \ref{sect:motives} in the motivic context. See in particular the proof of Lem.\ \ref{critpoints}.
\end{proof}

\subsection{The map $\mathcal T^*$}\label{sect:T*}
Let $s=\tfrac12 + m\in\textrm{Crit}(\Pi\times\Pi')$, $m\geq 0$. Hence, by Lem.\ \ref{lem:1/2critical}, we obtain a non-trivial map
$$\mathcal T^{(m)}\in\textrm{Hom}_{G'(\C)}(E_{\mu-m}\otimes E_\lambda,\C),$$ which we will fix in a compatible way (i.e., $\mathcal T^{(m)}$ for $\Pi\otimes\Pi'$ shall be identical to $\mathcal T^{(0)}$ for $(\Pi\|\cdot\|^m)\otimes\Pi'$). Again by Lem.\ \ref{lem:1/2critical}, it factors as $\mathcal T^{(m)} =\mathcal T^{(m)}_{\iota}\otimes \mathcal T^{(m)}_{\bar\iota}$, and if $\sigma\in$ Aut$(\C)$, then we define ${}^\sigma(\mathcal T^{(m)})$ in the obvious way, i.e.,
${}^\sigma (\mathcal T^{(m)}) =\mathcal T^{(m)}_{\sigma\circ\iota}\otimes \mathcal T^{(m)}_{\overline{\sigma\circ\iota}}$. Finally, if $s=\tfrac12\in\textrm{Crit}(\Pi\times\Pi')$, we let $\mathcal T = \mathcal T^{(0)}$ and obtain a morphism in cohomology
$$\mathcal T^*: H^{b_n+b_{n-1}}_c(\tilde S_{n-1},\E_\mu\otimes\E_\lambda)\ra H^{b_n+b_{n-1}}_c(\tilde S_{n-1},\C),$$
as in our diagram.

\subsection{Poincar\'e duality}\label{sect:poinc}
It remains to define the last map, denoted ``$\int$'' in our diagram. Therefore, observe that
$$b_n+b_{n-1}=\dim_\R \tilde S_{n-1}.$$
Moreover, for any open compact subgroup $K_f$ of $G(\A_f)$, the set of connected components of the orbifold $\tilde S_{n-1}(K'_f)$ is parametrized by the finite set $\mathcal X:=\A^\times_f/\K^\times\det(K'_f)$. We assume to have fixed orientations on the various connected components $\tilde S_{n-1}(K'_f)_x$, $x\in\mathcal X$, given by the orientation on $G'(\C)/U(n-1)$, which is defined by $X_1\wedge...\wedge X_{r}$, $r=\dim_\R\tilde S_{n-1}$, cf.\ Sect.\ \ref{sect:Theta}, so
$$\int_{\tilde S_{n-1}(K'_f)}=\sum_{x\in\mathcal X}\int_{\tilde S_{n-1}(K'_f)_x}.$$
Let $\int_{\tilde S_{n-1}}$ stand for the direct limit (over all open compact subgroups $K_f$ of $G(\A_f)$) of the maps $\int_{\tilde S_{n-1}(K'_f)}$.
Next, observe that the trivial character $\triv$ of $G'(\A_\K)$ defines a non-trivial cohomology class in $H^0(\tilde S_{n-1},\C)$. Using this class, Poincar\'e--Duality between $H^{b_n+b_{n-1}}_c(\tilde S_{n-1},\C)$ and $H^0(\tilde S_{n-1},\C)$ gives rise to a surjection

\begin{equation}\label{eq:poincare}
\begin{array}{cccccl}
H^{b_n+b_{n-1}}_c(\tilde S_{n-1},\C) & \longrightarrow & \C &\\
\theta & \longmapsto  & \int(\theta):=\int_{\tilde S_{n-1}} \theta.
\end{array}
\end{equation}

\subsection{A non-archimedean, particular vector}
We will now choose a special vector in the product of Whittaker models $W(\Pi_f)\times W(\Pi'_f)$, which has the property that it transforms nicely, when plugged into our diagram. This vector will be fixed as in Raghuram \cite{raghuram-imrn}, 3.1.4, which is itself inspired by Mahnkopf \cite{mahnk}, 2.1.1. \\\\
First, we remark that for $w\in S(\K)_f$, $\Pi'_w$ is tempered. Hence, Jacquet--Shalika \cite{jac-shal-pjm}, (3.2) Proposition still holds for $\Pi'_w$. Therefore, any non-zero Whittaker functional $\xi'_w\in W(\Pi'_w)$ is non-vanishing on $T'(\K_w)^+=\{t\in T'(\K_w)| t_i t_{i+1}^{-1}\in\mathcal{O}_w, t_{n-1,n-1}=1 \}$. We let $K(m_w)$ (resp. $K'(m'_w)$) be the mirahoric subgroup of $G(\K_w)$ (resp.\ $G'(\K_w)$) of level $m_w$ (resp.\ $m'_w$). This is the subgroup of $G(\O_w)$ (resp.\ $G'(\O_w)$), consisting of those matrices, whose last row is congruent to $(0,...,0,\star)$ modulo $\wp^{m_w}_w$, where $\wp_w$ is the unique maximal ideal in $\mathcal{O}_w$. Suppose from now on that $m_w$ (resp.\ $m'_w$) is the conductor of $\Pi_w$ (resp.\ $\Pi'_w$). Then, by Jacquet--Piatestski-Shapiro--Shalika, \cite{jac-ps-shalika-mathann} (5.1) Theorem, the space of Whittaker vectors, transforming by the central character $\omega_{\Pi_w}$ of $\Pi_w$ (resp.\ $\omega_{\Pi'_w}$ of $\Pi'_w$) under the action of the mirahoric subgroup is one-dimensional and its elements are usually called new vectors. \\\\
Now, for $\Pi'_w$ fix an element $t_{\Pi'_w}\in T'(\K_w)^+$ on which the nontrivial new vectors of $\Pi'_w$ do not vanish. Observe that we may choose the same element for all $\sigma$-twists, i.e., such that $t_{\Pi'_w}=t_{{}^\sigma\Pi'_w}$. If $w\notin S(\Pi')$, then we may take $t_{\Pi'_w}:=id$. Depending on these choices, for all $w\in S(\K)_f$, we let
$$\xi_{\Pi'_w} := \textrm{ the unique new vector such that } \xi_{\Pi'_w}(t_{\Pi'_w})=1$$
This pins down a special Whittaker vector $\xi_{\Pi'_f}:=\otimes'_{w\in S(\K)_f}\xi_{\Pi'_w}\in W(\Pi'_f)$.\\\\
Our choice for $\Pi_w$ will depend on the data fixed for $\Pi'_w$. First, we fix an element $t_{\Pi_w}\in T(\K_w)^+$ analogously as for $G'(\K_w)$. Now, for $w\notin S:=S(\Pi')\cup S(\psi)$, we let $\xi_{\Pi_w}$ be the unique new vector of $\Pi_w$, which satisfies $\xi_{\Pi_w}(t_{\Pi_w})=1$. It is a certain non-zero multiple $c_{\Pi_w}$ of the essential vector, cf.\ \cite{jac-ps-shalika-mathann} (4.1) Th\'eor\`eme. If $w\in S_f$, we take $\xi_{\Pi_w}$ to be the unique Whittaker vector, whose restriction to $\iota(G'(\K_w))$ is supported on $N'(\K_w)t_{\Pi'_w} K'(m'_w)$ and there equal to $\psi_w\omega^{-1}_{\Pi'_w}$. This gives a special Whittaker vector $\xi_{\Pi_f}:=\otimes'_{w\in S(\K)_f}\xi_{\Pi_w}\in W(\Pi_f)$.

\begin{lem}\label{lem:sigma-xi}
Let $\xi_{\Pi_f}$ and $\xi_{\Pi'_f}$ be the above Whittaker vectors. For a non-archimedean place $w$ of $\K$, the integral
$$\Psi(s,\xi_{\Pi_w},\xi_{\Pi'_w})=\int_{N'(\K_w)\backslash G'(\K_w)} \xi_{\Pi_w}(\iota(g))\xi_{\Pi'_w}(g)\|\!\det(g)\!\|_w^{s-\tfrac12}dg$$
converges for Re$(s)\geq 1-{\textbf{\sf m}}$ and has a meromorphic continuation to all of $\C$. It equals
$$\Psi(s,\xi_{\Pi_w},\xi_{\Pi'_w})=\left\{
\begin{array}{cc}
c_{\Pi_w}\cdot L(s,\Pi_w\times\Pi'_w) & \textrm{if } w\notin S\\
\|\!\det(t_{\Pi'_w})\!\|_w^{s-\tfrac12} \cdot vol(K'(m'_w)) & \textrm{if } w\in S,
\end{array}\right.$$
For all $\sigma\in$\emph{Aut}$(\C)$, ${}^\sigma\!(\xi_{\Pi'_w})=\xi_{{}^\sigma\Pi'_w}$ for all non-archimedean places $w$ of $\K$, whereas
$${}^\sigma\!(\xi_{\Pi_w})=\left\{
\begin{array}{cc}
\xi_{{}^\sigma\Pi_w} & \textrm{if } w\notin S\\
\omega_{{}^\sigma\Pi'_w}(t_\sigma)\cdot\xi_{{}^\sigma\Pi_w} & \textrm{if } w\in S.
\end{array}\right.$$
\end{lem}
\begin{proof}
Any one of the above assertions is either well--known or follows from a direct calculation using the definition of the local Whittaker vectors.
\end{proof}

\subsection{An archimedean non-vanishing result}

Recall our choices of cohomology classes $[\Pi_\infty]$ and $[\Pi'_\infty]$ from Sect.\ \ref{sect:Theta}. Since $\xi_{\Pi_\infty,\underline i, \alpha}$ is $K$-finite and $\xi_{\Pi'_\infty,\underline j, \beta}$ is $K'$-finite, the integral

$$\Psi(s,\xi_{\Pi_\infty,\underline i, \alpha},\xi_{\Pi'_\infty,\underline j, \beta})=\int_{N'(\C)\backslash G'(\C)} \xi_{\Pi_\infty,\underline i, \alpha}(\iota(g)) \xi_{\Pi'_\infty,\underline j, \beta}(g) \|\!\det(g)\!\|_\infty^{s-\tfrac12}dg$$
converges for Re$(s)\gg 0$. By \cite{cogdell-ps}, Thm.\ 1.2.(i) $\Psi(s,\xi_{\Pi_\infty,\underline i, \alpha},\xi_{\Pi'_\infty,\underline j, \beta})$ is holomorphic at all critical values of $L(s,\Pi\times\Pi')$.\\
Assume that $\tfrac12\in\textrm{Crit}(\Pi\times\Pi')$. In this case, we define
$$c(\tfrac12,\Pi_\infty,\Pi'_\infty):=\sum_{\alpha,\beta}\sum_{\underline i, \underline j}s(\underline i,\underline j)\cdot\mathcal T(e_\alpha \otimes e'_\beta)\cdot\Psi(\tfrac12,\xi_{\Pi_\infty,\underline i, \alpha},\xi_{\Pi'_\infty,\underline j, \beta}),$$
where $s(\underline i,\underline j)$ is given by $\iota(X^*_{\underline i})\wedge p(X'^*_{\underline j})= s(\underline i,\underline j)\cdot X^*_1\wedge...\wedge X^*_{r}$, $r=\dim_\R \tilde S_{n-1}$. By what we said above, $c(\tfrac12,\Pi_\infty,\Pi'_\infty)$ is well-defined, i.e., finite.\\
Now, drop the assumption that $\tfrac12\in\textrm{Crit}(\Pi\times\Pi')$ and let $s=\tfrac12+m\in\textrm{Crit}(\Pi\times\Pi')$ be an arbitrary critical point of $L(s,\Pi\times\Pi')$ with $m\geq 0$. Then, $\tfrac12\in\textrm{Crit}((\Pi\|\cdot\|^m)\times\Pi')$ and we define
$$c(\tfrac12+m,\Pi_\infty,\Pi'_\infty):=c(\tfrac12,\Pi_\infty\|\cdot\|_\infty^m,\Pi'_\infty).$$
By our compatible choice of $\mathcal T^{(m)}$, this is well-defined.  The following theorem recently proved
by Binyong Sun \cite[Thm.\ A]{sun} is of the highest importance for the present paper:

\begin{thm}\label{thm:sun}
For all $s=\tfrac12+m\in\textrm{\emph{Crit}}(\Pi\times\Pi')$ with $m\geq 0$, $c(\tfrac12+m,\Pi_\infty,\Pi'_\infty)\neq 0$.
\end{thm}

\begin{proof}  Since Sun's theorem concerns properties of a non-trivial pairing ${\rm Hom}_{G'(\R)}(\Pi_{\infty}\otimes \Pi'_{\infty},\C)$, and makes no reference to
zeta integrals, Whittaker functions, or coefficient systems, we explain how to translate his theorem into our non-vanishing statement.   We return to the
notation of \ref{sect:Theta}.   Write $\p = \g/\k$ and $\p' = \g'/\k'$.  The expressions there for the generators
$[\Pi_{\infty}] \in H^{b_{n}}(\g,K,\Pi_\infty\otimes E_\mu)$ and $[\Pi'_{\infty}] \in H^{b_{n-1}}(\g',K',\Pi'_\infty\otimes E_\lambda)$ are
based on the following identifications:
$$H_n:= H^{b_{n}}(\g,K,\Pi_\infty\otimes E_\mu) = {\rm Hom}_K(\wedge^{b_n}\p\otimes E_{\mu}^*,\Pi_{\infty}), $$
$$H_{n-1}:= H^{b_{n-1}}(\g',K',\Pi'_\infty\otimes E_\lambda) = {\rm Hom}_{K'}(\wedge^{b_{n-1}}\p\otimes E_{\lambda}^*,\Pi'_{\infty}).$$
Letting $h_n$, $h_{n-1}$ denote generators of the one-dimensional spaces $H_n$ and $H_{n-1}$, respectively,
it is known that $\rho = Im(h_n)$ (resp. $\rho' = Im(h_{n-1})$) is an irreducible $K$-type in $\Pi_{\infty}$ (resp. $K'$-type in $\Pi'_{\infty}$).
Sun's theorem states that, if $u$  is a generator of the
one-dimensional space ${\rm Hom}_{G'(\C)}([\Pi_{\infty}\otimes \Pi'_{\infty}]\hat{},\C)$, where $\hat{}$ denotes the corresponding
Casselman-Wallach (Frechet) completion, then $u|_{\rho\otimes \rho'} \neq 0$.   This implies in particular that if we define
$RS_m:  \Pi_{\infty}\otimes \Pi'_{\infty} \ra \C$ as the composite of the arrows in the
diagram
$$\Pi_{\infty}\otimes \Pi'_{\infty} \longrightarrow  W(\Pi_{\infty})\otimes W(\Pi'_{\infty})~~~\overset{\Psi_m}\longrightarrow  \C,$$
where
$$\Psi_m(\xi\otimes \xi'): =\Psi(\tfrac12+m,\xi,\xi')= \left[\int_{N'(\C)\backslash G'(\C)} \xi(\iota(g))\,\xi'(g)\,\|\!\det(g)\!\|_w^{s-\tfrac12}dg\right]_{s = \tfrac12+m}$$
is the Rankin-Selberg integral, then $RS_m$ defines a non-zero linear form on $\rho\otimes \rho'$.  Here we are using the fact,
due to Jacquet, Shalika, and Piatetski-Shapiro, that the Rankin-Selberg integral is not identically zero.

Now if we let $e_{\alpha}^*$ and $e_{\beta}^{\prime,*}$ denote the dual bases to the bases $e_{\alpha}$ and $e'_{\beta}$ of
$E_{\mu}$ and $E_{\lambda}$ introduced in  \ref{sect:Theta}, we find that (in the obvious notation)
$$\xi_{\Pi_\infty,\underline i, \alpha} = h_n(X_{\underline i}\otimes e_{\alpha}^*) \in \rho \quad\quad\quad \xi_{\Pi'_\infty,\underline j, \beta} = h_{n-1}(X_{\underline j}\otimes e_{\beta}^{\prime,*}) \in \rho'$$
are the matrix coefficients of $h_n$ and $h_{n-1}$ in the chosen bases.   Then the non-vanishing of $c(\tfrac12+m,\Pi_\infty,\Pi'_\infty)$ comes
down to the non-vanishing of $RS_m$ on $\rho\otimes \rho'$.

\end{proof}
%\begin{rem}
%For unitary representations $\Pi=BC(\pi)$ and the central critical value $s=\tfrac12$, the above conjecture was proved by B. Sun. The general case shall be treated in ...
%\end{rem}

Observe that Theorem\ \ref{thm:sun} implies that $c(\tfrac12+m,{}^\sigma\Pi_\infty,{}^\sigma\Pi'_\infty)\neq 0$, for all $\sigma\in$Aut$(\C)$. We denote by $p(m,{}^\sigma\Pi_\infty,{}^\sigma\Pi'_\infty)$ the inverse of $c(\tfrac12+m,{}^\sigma\Pi_\infty,{}^\sigma\Pi'_\infty)$. In the special case when $m=0$, we will abbreviate $p({}^\sigma\Pi_\infty,{}^\sigma\Pi'_\infty):=p(0,{}^\sigma\Pi_\infty,{}^\sigma\Pi'_\infty)$.

\subsection{A theorem on Whittaker periods}
%Let us recall the Gau\ss -sum of a Hecke character $\omega:\K^\times\backslash\A^\times_\K\ra\C^\times$. Therefore, we let $\mathfrak{c}$ stand for the conductor ideal of $\omega_f$ and recall the different $\mathfrak{D}_\K$ of $\K/\Q$. We take $y = (y_w)_{w \in S(\K)_f} \in {\mathbb A}_f^{\times}$ such that ${\rm ord}_w(y_w) = -{\rm ord}_w(\mathfrak{c}) -{\rm ord}_w(\mathfrak{D}_\K)$. Then, the Gau\ss ~sum of $\omega_f$ is defined as $\mathcal{G}(\omega_f,\psi_f,y) = \prod_{w \in S(\K)_f} \mathcal{G}(\omega_w,\psi_w,y_w)$, where the local Gau\ss ~sum $\mathcal{G}(\omega_w,\psi_w,y_w)$ is defined as
%$$
%\mathcal{G}(\omega_w,\psi_w,y_w) = \int_{\mathcal{O}_w^{\times}} \omega_w(u_w)^{-1}\psi_w(y_wu_w) du_w.
%$$
%If all inputs are unramified at $w$, then $\mathcal{G}(\omega_w,\psi_w,y_w) =1$, so the above product is finite. Suppressing the dependence on $\psi$ and $y$, we denote %$\mathcal{G}(\omega_f,\psi_f,y)$ simply by $\mathcal{G}(\omega_f)$.

Let us recall the Gau\ss -sum of a Hecke character $\omega :\Q^\times\backslash\A^\times_\Q\ra\C^\times$. Let $\mathfrak{c}$ stand for the conductor ideal of $\omega_f$.
We take $y = (y_p)_{p \in S_f} \in {\mathbb A}_f^{\times}$ such that ${\rm ord}_p(y_p) = -{\rm ord}_p(\mathfrak{c})$. Then, the Gau\ss ~sum of $\omega_f$ is
defined as $\mathcal{G}(\omega_f,\psi_f,y) = \prod_{p \in S_f} \mathcal{G}(\omega_p,\psi_p,y_p)$, where the local Gau\ss ~sum $\mathcal{G}(\omega_p,\psi_p,y_p)$ is defined as
$$
\mathcal{G}(\omega_p,\psi_p,y_p) = \int_{\Z_p^{\times}} \omega_p(u_p)^{-1}\psi_p(y_pu_p) du_p.
$$
If all inputs are unramified at $p$, then $\mathcal{G}(\omega_p,\psi_p,y_p) =1$, so the above product is finite. Suppressing the dependence on $\psi$ and $y$, we denote $\mathcal{G}(\omega_f,\psi_f,y)$ simply by $\mathcal{G}(\omega_f)$.

If $\omega: \K^\times\backslash\A^\times_\K\ra\C^\times$ is a Hecke character of $\K$, let
$\omega_0$ denote its restriction to the id\`eles of $\Q$. We are now ready to state and prove the main result of this section, which is analogous to the main result in \cite{raghuram-imrn}.

\begin{thm}\label{thm:whittaker-periods}
Let $\Pi=BC(\pi)\|\cdot\|^{\textbf{\sf m}}$ be a cuspidal automorphic representation of $G(\A_\K)$, ${\textbf{\sf m}}\in\Z$, and let $\Pi'=BC(\pi')$ be an abelian automorphic representation of $G'(\A_\K)$ obtained by base change from unitary groups as in Sections \ref{sect:Pi} and \ref{sect:Pi'}. In particular, $\Pi$ is cohomological with respect to $E_\mu$ and $\Pi'$ is cohomological with respect to $E_\lambda$. We assume that the highest weights $\mu$ and $\lambda$ satisfy hypothesis \ref{hyp:coeff}. In view of the archimedean non-vanishing result, cf.\ Theorem \ \ref{thm:sun}, the following holds:
\begin{enumerate}
\item For all critical values $\tfrac12+m\in\textrm{\emph{Crit}}(\Pi\times\Pi')$ with $m\geq 0$ and every $\sigma\in {\rm Aut}(\C)$,
$$
\sigma\left(
\frac{L(\tfrac 12+m,\Pi_f \times \Pi'_f)}{p(\Pi)p(\Pi')p(m,\Pi_\infty,\Pi'_\infty) \mathcal{G}(\omega_{\Pi'_{f,0}})}\right) \ = \
\frac{L(\tfrac 12+m,{}^\sigma\Pi_f \times {}^\sigma\Pi'_f)}{p({}^\sigma\Pi)p({}^\sigma\Pi')p(m,{}^\sigma\Pi_\infty,{}^\sigma\Pi'_\infty) \mathcal{G}(\omega_{{}^\sigma\Pi'_{f,0}})}.
$$
\item
$$
L(\tfrac 12+m,\Pi_f \times \Pi'_f) \ \sim_{\Q(\Pi_f)\Q(\Pi'_f)} p(\Pi)p(\Pi')p(m,\Pi_\infty,\Pi'_\infty) \mathcal{G}(\omega_{\Pi'_{f,0}}),
$$
where ``$\sim_{\Q(\Pi_f)\Q(\Pi'_f)}$'' means up to multiplication by an element in the composition of number fields $\Q(\Pi_f)\Q(\Pi'_f)$.
\end{enumerate}
\end{thm}
\begin{rem}\label{nogauss}  In fact, the central character of $\Pi'$ is the twisted base change of the central character of $\pi'$, and therefore its restriction to
the id\`eles of $\Q$ is trivial!   The term $\mathcal{G}(\omega_{{}^\sigma\Pi'_{f,0}})$ is nevertheless retained in the statement of the above theorem
because it applies, with minor modifications, to cohomological automorphic representations not obtained by base change, in which case the presence of the Gau\ss {} sum, as in \cite{raghuram-imrn}, is indispensable.
\end{rem}
\begin{proof}
Since (1) implies (2) by Strong Multiplicity One for isobaric automorphic representations, we only prove (1). We will proceed in two steps.\\\\
\emph{Step 1: Assume $\tfrac12$ is critical.}\\\\
For any $\underline i$, $\underline j$, $\alpha$ and $\beta$, denote
$$\varphi_{\Pi,\underline i,\alpha}:=W^{-1}(\xi_{\Pi_\infty,\underline i,\alpha}\otimes\xi_{\Pi_f})\in\Pi\quad\quad\textrm{and}\quad\quad \varphi_{\Pi',\underline j,\beta}:=W^{-1}(\xi_{\Pi'_\infty,\underline j,\beta}\otimes\xi_{\Pi'_f})\in\Pi',$$
i.e., the inverse image of our particular Whittaker vectors. It follows directly from the definition of the various maps in our diagram, that if we chase $\xi_{\Pi_f}\times \xi_{\Pi'_f}$ through the diagram, we obtain
$${\rm Dia}(\xi_{\Pi_f}\times \xi_{\Pi'_f})=\sum_{\alpha,\beta}\sum_{\underline i, \underline j}p(\Pi)^{-1}p(\Pi')^{-1}\cdot s(\underline i,\underline j)\cdot\mathcal T(e_\alpha \otimes e'_\beta)\cdot\int_{G'(\K)\backslash G'(\A_\K)} \varphi_{\Pi,\underline i,\alpha}|_{G'(\A_\K)}\cdot \varphi_{\Pi',\underline j,\beta} dg.$$
We are left to compute the latter integral. Since $\varphi_{\Pi,\underline i,\alpha}$ is cuspidal, it is well-known that
$$\int_{G'(\K)\backslash G'(\A_\K)} \varphi_{\Pi,\underline i,\alpha}(\iota(g))\cdot \varphi_{\Pi',\underline j,\beta}(g) \|\!\det(g)\!\|^{s-\tfrac12}dg$$
converges for all $s\in\C$ and equals
$$\int_{N'(\A_\K)\backslash G'(\A_\K)} (\xi_{\Pi_\infty,\underline i,\alpha}\otimes\xi_{\Pi_f})(\iota(g))\cdot (\xi_{\Pi'_\infty,\underline j,\beta}\otimes\xi_{\Pi'_f})(g) \|\!\det(g)\!\|^{s-\tfrac12} dg$$
for Re$(s)\gg 0$. For such $s$, this is furthermore equal to
$$\Psi(s,\xi_{\Pi_\infty,\underline i,\alpha},\xi_{\Pi'_\infty,\underline j,\beta})\cdot L^{S}(s,\Pi\times\Pi') \prod_{w\notin S} c_{\Pi_w}\prod_{w\in S_f} \|\!\det(t_{\Pi'_w})\!\|_w^{s-\tfrac12} vol(K'(m'_w))$$
with $S= S(\Pi')\cup S(\psi)$ by Lemma \ref{lem:sigma-xi}. By analytic continuation, we obtain
$${\rm Dia}(\xi_{\Pi_f}\times \xi_{\Pi'_f})=\frac{L(\tfrac12,\Pi_f\times \Pi'_f)}{p(\Pi)p(\Pi')p(\Pi_\infty,\Pi'_\infty)}\cdot\frac{\prod_{w\notin S} c_{\Pi_w}\prod_{w\in S_f} vol(K'(m'_w))}{\prod_{w\in S} L(\tfrac12,\Pi_w\times \Pi'_w)},$$
so
$$\sigma({\rm Dia}(\xi_{\Pi_f}\times \xi_{\Pi'_f}))=\sigma\left(\frac{L(\tfrac12,\Pi_f\times \Pi'_f)}{p(\Pi)p(\Pi')p(\Pi_\infty,\Pi'_\infty)}\right)\cdot\frac{\prod_{w\notin S} \sigma(c_{\Pi_w})\prod_{w\in S_f} vol(K'(m'_w))}{\prod_{w\in S_f} \sigma(L(\tfrac12,\Pi_w\times \Pi'_w))}.$$
Here we note that $\prod_{w\in S_f} vol(K'(m'_w))$ is a rational number. Next, observe that $\sigma(L(\tfrac12,\Pi_w\times \Pi'_w))=L(\tfrac12,{}^\sigma\Pi_w\times {}^\sigma\Pi'_w)$, which is proved in Raghuram \cite{raghuram-imrn}, Prop.\ 3.17. Moreover, since the results of Jacquet--Piatetski-Shapiro--Shalika in \cite{jac-ps-shalika-rsc} and \cite{jac-ps-shalika-mathann} as well as Clozel \cite{clozel} Lemme 4.6 are valid for $\Pi_w$, $w\notin S(\psi)$, the proof of Mahnkopf \cite{mahnk-crelle} Prop.\ 2.3.(c) may be carried over to the situation considered here. In other words, $ \sigma(c_{\Pi_w})= c_{{}^\sigma\Pi_w}$ for all $w\notin S$, cf.\ \cite{mahnk} p.\ 621 or \cite{raghuram-imrn} Prop.\ 3.21. Finally, we obtain
$$\sigma({\rm Dia}(\xi_{\Pi_f}\times \xi_{\Pi'_f}))=\sigma\left(\frac{L(\tfrac12,\Pi_f\times \Pi'_f)}{p(\Pi)p(\Pi')p(\Pi_\infty,\Pi'_\infty)}\right)\cdot\frac{\prod_{w\notin S} c_{{}^\sigma\Pi_w}\prod_{w\in S_f} vol(K'(m'_w))}{\prod_{w\in S_f} L(\tfrac12,{}^\sigma\Pi_w\times {}^\sigma\Pi'_w)}.$$
On the other hand, as all maps in the definition of our diagram are $\sigma$-equivariant, we see $\sigma({\rm Dia}(\xi_{\Pi_f}\times \xi_{\Pi'_f}))={\rm Dia}({}^\sigma\!\xi_{\Pi_f}\times {}^\sigma\!\xi_{\Pi'_f})$. Therefore, Lem.\ \ref{lem:sigma-xi} implies
$$\sigma({\rm Dia}(\xi_{\Pi_f}\times \xi_{\Pi'_f}))=\left(\prod_{w\in S_f} \omega_{{}^\sigma\Pi'_w}(t_\sigma)\right)\cdot{\rm Dia}(\xi_{{}^\sigma\Pi_f}\times \xi_{{}^\sigma\Pi'_f}).$$
As $\omega_{{}^\sigma\Pi'_w}$ is unramified outside $S({}^\sigma\Pi')$ and since $t_\sigma\in\widehat{\mathcal{O}}^*$, we get $\omega_{{}^\sigma\Pi'_w}(t_\sigma)=1$ for $w\notin S({}^\sigma\Pi')$. Hence, observing that $S(\Pi')=S({}^\sigma\Pi')$, the finite product equals $\prod_{w\in S_f} \omega_{{}^\sigma\Pi'_w}(t_\sigma)=\omega_{{}^\sigma\Pi'_f}(t_\sigma)$. We claim that there is the identity
$$\omega_{{}^\sigma\Pi'_f}(t_\sigma)=\frac{\sigma(\mathcal{G}(\omega_{\Pi'_{f,0}}))}{\mathcal{G}(\omega_{{}^\sigma\Pi'_{f,0}})}.$$
Indeed, at a place $p \in S_f$,
\begin{eqnarray*}
\mathcal G(\omega_{{}^\sigma\!\Pi'_{p,0}}) & = & \int_{\mathcal \Z^\times_p}\omega_{{}^\sigma\!\Pi'_{p,0}}(u_p)^{-1}\psi_p(y_p u_p)du_p\\
 & = & \int_{\mathcal \Z^\times_p}\sigma(\omega_{\Pi'_{p,0}}(u_p))^{-1}\sigma(\psi_p(y_p u_p t_{\sigma}^{-1}))du_p\\
 & = & \sigma(\omega_{\Pi'_{p,0}}(t_{\sigma}))^{-1} \int_{\mathcal \Z^\times_p}\sigma\left(\omega_{\Pi'_{p,0}}(u_p)^{-1}\psi_p(y_p u_p)\right)du_p\\
& = & \omega_{{}^\sigma\!\Pi'_{p,0}}(t_{\sigma})^{-1} \cdot \sigma\left(\int_{\mathcal \Z^\times_p}\omega_{\Pi'_{p,0}}(u_p)^{-1}\psi_p(y_p u_p)du_p\right)\\
& = & \omega_{{}^\sigma\!\Pi'_{p,0}}(t_{\sigma})^{-1} \cdot \sigma(\mathcal G(\omega_{\Pi'_{p,0}})),
\end{eqnarray*}
where the second last equation follows form the fact that -- by the very choice of $\psi$ -- the above integral is the finite sum over the classes modulo the conductor ideal of $\omega_{\Pi'_p}$ (on which the integrand is constant). Therefore,
$$\sigma\left(\frac{L(\tfrac12,\Pi_f\times \Pi'_f)}{p(\Pi)p(\Pi')p(\Pi_\infty,\Pi'_\infty)}\right)\cdot\frac{\prod_{w\notin S} c_{{}^\sigma\Pi_w}\prod_{w\in S_f} vol(K'(m'_w))}{\prod_{w\in S_f} L(\tfrac12,{}^\sigma\Pi_w\times {}^\sigma\Pi'_w)}=\frac{\sigma(\mathcal{G}(\omega_{\Pi'_{f,0}}))}{\mathcal{G}(\omega_{{}^\sigma\Pi'_{f,0}})}\cdot {\rm Dia}(\xi_{{}^\sigma\Pi_f}\times \xi_{{}^\sigma\Pi'_f}).$$
Recalling that $S(\Pi')=S({}^\sigma\Pi')$ and $m'_w$ is the same for $\Pi'_w$ and ${}^\sigma\Pi'_w$, we obtain by what we have seen above, that
$${\rm Dia}(\xi_{{}^\sigma\Pi_f}\times \xi_{{}^\sigma\Pi'_f})=\frac{L(\tfrac12,{}^\sigma\Pi_f\times {}^\sigma\Pi'_f)}{p({}^\sigma\Pi)p({}^\sigma\Pi')p({}^\sigma\Pi_\infty,{}^\sigma\Pi'_\infty)}\cdot\frac{\prod_{w\notin S} c_{{}^\sigma\Pi_w}\prod_{w\in S_f} vol(K'(m'_w))}{\prod_{w\in S_f} L(\tfrac12,{}^\sigma\Pi_w\times {}^\sigma\Pi'_w)}.$$
Comparing the last two equations shows (1) if $\tfrac12$ is critical.\\\\
\emph{Step 2: The general case.}\\\\
We drop the assumption that $\tfrac12$ is critical now. Let $s=\tfrac12+m$ be an arbitrary critical value of $L(s,\Pi\times\Pi')$, $m\geq0$. Then $\tfrac12$ is critical for $L(s,(\Pi\|\cdot\|^m)\times\Pi')$ and we are in the situation considered above. Hence, by what we have just observed,
\begin{eqnarray*}
\sigma\left(\frac{L(\tfrac 12,(\Pi\|\cdot\|^m)_f\times\Pi_f')}{p(\Pi\|\cdot\|^m)p(\Pi')p(\Pi_\infty\|\cdot\|_\infty^m,\Pi'_\infty)\mathcal{G}(\omega_{\Pi'_{f,0}})}\right)  & = & \frac{L(\tfrac 12,{}^\sigma(\Pi\|\cdot\|^m)_f\times{}^\sigma\Pi'_f)}{p({}^\sigma(\Pi\|\cdot\|^m))p({}^\sigma\Pi')p({}^\sigma(\Pi_\infty\|\cdot\|_\infty^m),{}^\sigma\Pi'_\infty)\mathcal{G}(\omega_{{}^\sigma\Pi'_{f,0}})}\\
& = & \frac{L(\tfrac 12+m, {}^\sigma\Pi_f \times {}^\sigma\Pi'_f)} {p({}^\sigma\Pi\|\cdot\|^m)p({}^\sigma\Pi')p(m,{}^\sigma\Pi_\infty,{}^\sigma\Pi'_\infty)\mathcal{G}(\omega_{{}^\sigma\Pi'_{f,0}})}.
\end{eqnarray*}
By \cite{raghuram-shahidi-imrn}, Thm.\ 4.1, the period for $\Pi_f$ satisfies the relation
$$\sigma\left(\frac{p(\Pi\|\cdot\|^m)}{p(\Pi)}\right)=\frac{p({}^\sigma\Pi\|\cdot\|^m)}{p({}^\sigma\Pi)},$$
where we used that $\mathcal G(\|\cdot\|_f^m)=1$. This proves the theorem.

\end{proof}

%\begin{rem}
%One should observe that the central character $\omega_{\Pi'_f}$ in the Gau\ss {} sum can be replaced by its restriction $\omega_{\Pi'_f,0}$  to the ideles over $\Q$ (or rather their image inside the ideles of $\K$). This is important, in order to make our result compatible with the conjecture of Deligne on period relations and critical values (cf.\ \cite{deligne}), respectively with the conjecture of Ichino--Ikeda--Neil Harris (cf.\ \cite{nharris}).
%\end{rem}

\begin{rem}  At no time in the proof did we use the assumption that $\Pi'$ is of abelian type, other than to avoid introducing additional notation.  It is therefore clear that our result also holds for cuspidal automorphic representations $\Pi'$ of $G'(\A_\K)$; in other words, the proof of Thm.\ \ref{thm:whittaker-periods}
also yields Thm.\ \ref{thm:whittaker-periods-general}. As mentioned in the introduction, in this case, i.e., if $\Pi'$ is cuspidal, it has recently also been proved by Raghuram in \cite{ragh-gln}, Thm.\ 1.1 -- even over any number field.
Moreover, it is clear that we did not need that $\Pi$ and $\Pi'$ are constructed via base change, in order to prove Thm.\ \ref{thm:whittaker-periods}. However, for later use, and to simplify the determination of the set of critical points, we already assumed this to be the case.
\end{rem}

\begin{rem}
Let us finally also remark, that we did not need to fix some open compact subgroup $K_f\subset G(\A_f)$ within the entire process of showing Thm.\ \ref{thm:whittaker-periods}.

\end{rem}

\section{Motives -- An Interlude}\label{sect:motives}

This section reviews and extends the results of \cite{harcrelle} and \cite{haradj} that are needed to prove the
main results of this paper.  The crucial definitions and calculations are scattered in different parts of the two
papers in question, allowing space for several confusing sign changes, and their conventions are moreover
not quite compatible.  We hope that the present section will
provide a more convenient reference.  Moreover, the results of the earlier papers were proved under
unnecessarily restrictive hypotheses; here we have striven for maximum generality, always assuming
that the base field is an imaginary quadratic field.

The main purpose of this section is to interpret the results on special values proved by automorphic
methods to Deligne's conjecture, stated in setting of motives for absolute Hodge
cycles.  This setting is partially hypothetical.  It is not known, for example, that two geometric realizations in the
cohomology of Shimura varieties of  the Galois
representations attached to a (motivic) algebraic Hecke character are isomorphic as motives, and in particular
that they define the same periods.  We are therefore led to introduce automorphic analogues of the invariants
that arise in the (motivic) calculation of Deligne's periods over imaginary quadratic fields.  See also \cite{yoshida2} for calculations of
Deligne periods of motives over $\Q$.

From now on we let $\textbf{\sf m}=0$, i.e., the cuspidal automorphic representation $\Pi$ is simply the unitary base change from $\pi$.

\subsection{Tensor products of motives}
The $L$-functions of the automorphic representations $\Pi$ and $\Pi'$ of the previous section are conjecturally attached to motives over
$M(\Pi)$ and $M(\Pi')$ over $\K$ of ranks $n$ and $n-1$, respectively, with coefficients in finite, possibly non-trivial extensions
\begin{center}
\framebox{$E(\Pi)/\Q(\Pi_f) \quad\quad\textrm{and}\quad\quad E(\Pi')/\Q(\Pi_f).$}
\end{center}
The passage from $\Pi$ to $M(\Pi)$ involves a standard shift, so that
$$L(s,M(\Pi)) = L(s+ \tfrac{1-n}{2},\Pi)$$
whose center of symmetry is at $s = \frac{n}{2}$ rather than at $s = \frac{1}{2}$.  One can in any case attach Hodge structures
and compatible families of $\ell$-adic representations to $\Pi$ and $\Pi'$,  pure of
weight $n-1$ and $n-2$, respectively, so the tensor product $M(\Pi)\otimes_{\K}M(\Pi')$ is pure of weight
$2n-3$, and the center of symmetry of the functional equation of its $L$-function is at the point $s = n-1$.

More generally, we can take $\Pi'$ to be an
automorphic representation of $GL_{n'}(\A_\K)$. We will be particularly interested in the cases $n' = n-1$ and $n' = 1$,
but the methods of the present paper allow us to say quite a lot about the general case.  In particular, the center
of symmetry of the functional equation is the point $\frac{n+n'-1}{2}$, which is not necessarily an integer; however, we will also be interested in critical
values to the right of the center of symmetry.
The fact that $\Pi$ and $\Pi'$ arise by base change from unitary groups is reflected in the polarization property
\begin{equation}\label{pol} M(\Pi)^c \ira M(\Pi)^{\vee}(1-n) \quad\quad\textrm{and}\quad\quad M(\Pi')^c \ira M(\Pi')^{\vee}(1-n').  \end{equation}
A formalism relating the various realizations of motives satisfying \eqref{pol} and defining their period invariants is
developed in \S 1.1 of \cite{haradj}, and the Deligne periods of the associated adjoint and tensor product motives
are computed in \S 1.3-1.4 of \cite{haradj}.

Recall that the formula of Thm.\ \ref{thm:whittaker-periods} is valid provided the initial representations $\pi$ and $\pi'$ of
the definite unitary groups $H$ and $H'$ are cohomological with respect to finite-dimensional representations satisfying
the equivalent conditions of Lem.\ \ref{lem:coeff}.  Exactly the same conditions appear in (2.3.1) of \cite{haradj} where they are
used to evaluate the Deligne period of $M(\Pi)\otimes_{\K}M(\Pi')$ -- here $n' = n-1$ -- when the Ichino-Ikeda conjecture
computes the central value $L(\tfrac 12,\Pi_f \times \Pi'_f)$ in terms of periods of integrals of automorphic forms
on $H\times H'$.  This is the same central value that appears on the left-hand side of the formulas in \ref{thm:whittaker-periods}.
We exploit this identification in subsequent sections.  Here we complete \cite{haradj} by determining the
set of all critical values and their corresponding Deligne periods, when the inequalities of Lem.\ \ref{lem:coeff} are satisfied.

Write $M = M(\Pi)$, $M' = M(\Pi')$.  As in \cite{haradj}, the Hodge types for $M$ (resp. $M'$) are
denoted
\begin{equation}\label{polwt} (p_i,q_i) := (p_i, n-1 - p_i) \quad\textrm{and}\quad (p^c_i,q^c_i):= (p_i^c,n-1-p_i^c) = (q_{n+1-i},p_{n+1-i}) \end{equation}
(resp.\ $(r_j,n'-1-r_j)$ and  $(r_j^c,n'-1-r_j^c)$).
Bearing in mind that the cohomological parameters of $\Pi$ and $\Pi'$ are respectively
given by the tuples $\mu$ and $\lambda$ as in \ref{thm:base_change} (and not by $a_i$'s and $b_j$'s as in \cite{haradj}),
we have
\begin{equation}\label{param} p_i = n-i + \mu_i, ~~~r_j = n' - j + \lambda_j \end{equation}
For any integer $m$, the Hodge types for the Tate twist $M(m)$ are
\begin{equation}\label{pol1} (p_i-m, n-1 - p_i-m) \quad\textrm{and}\quad (p_i^c-m, n-1-p_i^c -m) =   (n-1- p_{n+1-i} - m, p_{n+1-i} - m), \end{equation}
where the equality on the right is a consequence of \ref{pol}.
We extend the definition of \cite{haradj} by letting
$T(M(m),M')$ be the set of pairs  $(t, u)$ of indices, $1 \leq t \leq n$, $1 \leq u \leq n'$,
 such that  $p^c_t  - m + r^c_u \geq \frac{w+1}{2}$, and write $T(M,M') = T(M(0),M')$.
 Using \ref{pol1}, we transform the left hand side of this equation to
$$  n-1 - p_{n+1-t} + n'-1 - r_{n'+1-u} - m \geq \frac{n+n'-1-2m}{2}$$$$
 \Leftrightarrow  \frac{n+n'-3}{2}  \geq  p_{n+1-t} +  r_{n'+1-u}   $$
and this is true if and only if
\begin{equation}\label{ineq1} \frac{n+n'+1}{2}  \geq t + u + \mu_{n+1-t} +  \lambda_{n'+1-u}  .\end{equation}

\subsection{Critical values when $n'=n-1$}\label{nnminus1}
For the moment we let $n' = n-1$.  The tensor product motive we consider is $R(M\otimes M') =  R_{\K/\Q}(M\otimes M')$, of weight $w = 2n-3$.  %In \cite{haradj}
Now in view of \ref{lem:coeff},
\begin{rem}\label{tmm} When $n' = n-1$, $T(M,M')$ consists exactly of the set of pairs
$\{(t,u)~|~ t+u \leq n \}$, of cardinality $\frac{n(n-1)}{2}$.
\end{rem}
To calculate the Gamma factor, we follow Serre's recipe (see \cite{deligne} Sect.\ 5.2--5.3) and find
$$\prod_{a,b}\Gamma_{\C}(s - p_a - r_b) \times \prod_{t,u}\Gamma_{\C}(s - p^c_t - r^c_u)$$
where the
product is taken over pairs such that $p_a + r_b \leq \frac{w-1}{2} = n- 2$ and $p^c_t + r^c_u \leq n- 2$.
For the first set we have
$$2n-1-a - b + \mu_a + \lambda_b \leq n-2 \Leftrightarrow n+1 + \mu_a + \lambda_b \leq a + b.$$
By  the branching rules this holds if and only if $a+b \geq n+1$.  %For the second set it's
 For the second set the polarization gives us
$$\begin{aligned}
&  p^c_t + r^c_u = 2n-3 - (p_{n+1-t} +  r_{n-u}) \leq n-2 \\ \Leftrightarrow & \,
n-1 \leq p_{n+1-t} +  r_{n-u} \\ \Leftrightarrow & \, n - 1 \leq t + u - 2 + \mu_{n+1-t}  + \lambda_{n-u}
\end{aligned}$$
%which comes down to
which is true if and only if $n + 1 \leq t+u$.

There is therefore a pole at the integer $m$ unless
$m \geq 1 + \sup(\sup_{a+b \geq n+1} p_a + r_b, \sup_{t+u \geq n + 1} p^c_t + r^c_u)$.
The inequalities \ref{lem:coeff} imply that it suffices to consider the pairs $(a,b)$ and $(t,u)$ with $a+b = n+1$, $t+u = n + 1$.
The minimum is thus
 $$m_{min} = 1 + \sup(\sup_{a+b = n+1} n-2 + \mu_a + \lambda_b, \sup_{t+u = n+1} n- 2 - \mu_{n+1-t}  - \lambda_{n-u}).$$
 This gives the lower bound of the critical set, and the functional equation that exchanges $s$ with $w+1 - s$ implies the upper bound is

$$\begin{aligned}
m_{max} = w + 1 - m_{min} &= 2n-2 - m_{min}  \\ &= \inf(n - 1 - \sup_{1 \leq a \leq n} (\mu_a + \lambda_{n+1-a}),  n - 1+ \inf_{1 \leq t \leq n} ( \mu_{u}  + \lambda_{n - u})) \\ &= n - 1 +  \sup(-\sup_{2 \leq a \leq n} (\mu_a + \lambda_{n+1-a}),   \inf_{1 \leq u \leq n-1} ( \mu_{u}  + \lambda_{n-u})).
 \end{aligned}$$

%** The indexing is wrong, we should have $a > 1$, $t < n$.
% We thus have $\mu_2 + \lambda_{n-1}, \mu_3 + \lambda_{n-2},  \dots, \mu_n + \lambda_1$
%for the first group, and $\mu_n + \lambda_1, \mu_{n-1} + \lambda_2, \dots, \mu_1 + \lambda_n$ for the second, which is wrong.

%The answer should be
% $$m_{max} = n - 1 +  \sup(-\sup_{2 \leq a \leq n} (\mu_a + \lambda_{n+1-a}),   \inf_{1 \leq t \leq n-1} ( \mu_{t}  + \lambda_{n-t})).$$

In the following lemma, $\mu(m)$ denotes the $n$-tuple
 $\mu_1 - m \geq \mu_2 - m \geq \dots \geq \mu_n - m$, the highest weight of the representation
 $E^{\sf unt}_{\mu}\otimes \det^{-m}$.

 \begin{lem}\label{critpoints}   The set of critical points of the $L$-function $L(s,R(M\otimes M'))$ is the set of integers $s_0 = n-1 \pm m$, where $m$ runs through the non-negative integers such that
\emph{Hom}$_{G'(\C)}((E^{\sf unt}_{\mu(m)}\otimes E^{\sf unt}_{{\mu^{\sf v}}(m)})\otimes (E^{\sf unt}_\lambda\otimes E^{\sf unt}_{\lambda^{\sf v}}),\C)\neq 0.$  The maximum value of $m$ is also the minimum of the distances between successive entries in the sequence of inequalities
$$\mu_1\geq -\lambda_{n-1}\geq \mu_2\geq -\lambda_{n-2}\geq...\geq -\lambda_1\geq \mu_n.$$
 \end{lem}

\begin{proof}  The branching law implies that the existence of $G'(\C)$-equivariant homomorphisms
$$E^{\sf unt}_{\mu(m)}\otimes E^{\sf unt}_\lambda\otimes \ra \C, ~~ E^{\sf unt}_{{\mu^{\sf v}}(m)}\otimes E^{\sf unt}_{\lambda^{\sf v}} \ra \C$$
is equivalent to the two series of inequalities
$$\mu_1 - m \geq -\lambda_{n-1}\geq \mu_2 - m \geq -\lambda_{n-2}\geq...\geq -\lambda_1\geq \mu_n - m,$$
$$-\mu_1 - m \leq \lambda_{n-1} \leq -\mu_2 - m \leq \lambda_{n-2} \leq...\leq \lambda_1\leq -\mu_n - m,$$
which is equivalent to the determination of $m_{max}$ above.  The final sentence is clear.
\end{proof}

We now derive an expression for $c^+(R(M\otimes M')(m))$, following \S 1.4 of \cite{haradj}, when $n' = n-1$.
Recall the formula in Lemma 1.4.1 of {\it loc.\ cit.}:

\begin{eqnarray}\label{cplus}  c^+(R(M\otimes M')^{\vee}) &=&  c^-(R(M\otimes M')^{\vee}) \nonumber \\
&\sim& \prod_{(t,u) \in T(M,M')} Q_{n+1-t}(M)^{-1}Q_{n-u}(M')^{-1} \cdot \delta(M\otimes M')^{-1} \nonumber \\
&\sim& \prod_{t+u \leq n} Q_{n+1-t}(M)^{-1}Q_{n-u}(M')^{-1} \cdot \delta(M)^{1-n}\cdot \delta(M')^{-n}
\end{eqnarray}

Here $\delta(M)$ and $\delta(M')$ are the full period determinants studied in \S 1.2 of \cite{haradj}, specifically in Lemma 1.2.7, which
we recall here. We let $E(M)$ and $E(M')$ be the respective coefficient fields of $M$ and $M'$.

\begin{lem}\label{delta}  There are elements $d(M) \in E(M)$ and $d(M') \in E(M')$ such that
$$\delta(M)^{-1} \sim d(M)^{\frac 12}(2\pi i)^{n(n-1)/2}Q_{\det M}^{1/2} \quad\textrm{and}\quad \delta(M')^{-1} \sim d(M')^{\frac 12}(2\pi i)^{(n-1)(n-2)/2}Q_{\det M'}^{1/2},$$
where $Q_{\det M} = \prod_{t = 1}^n Q_t(M)$ and  $Q_{\det M'} = \prod_{u= 1}^{n-1} Q_u(M')$.
\end{lem}
We introduce the new invariants
\begin{equation}\label{qm} q(M) = d(M)^{\frac 12}Q_{\det M}^{1/2} \quad\quad\textrm{and}\quad\quad Q_{\leq r}(M) = \prod_{t \leq r} Q_t(M), ~1 \leq r \leq n
 \end{equation}
and define $q(M')$ and $Q_{\leq r}(M')$ analogously. Then, combining \ref{cplus}, \ref{delta}, \ref{qm}, we obtain
\begin{eqnarray}\label{cplus1}
c^{\pm}(R(M\otimes M')^{\vee})
&\sim&  (2\pi i)^{n(n-1)^2/2}(2 \pi i)^{n(n-1)(n-2)/2}\prod_{t = 1}^{n-1}Q_t(M)^{n-t}\prod_{u = 1}^{n-1} Q_u(M)^{n-u}\cdot q(M)^{1-n}q(M')^{-n} \nonumber \\
&\sim& (2\pi i)^{n(n-1)(2n-3)/2}\prod_{r = 1}^{n-1} (Q_{\leq r}(M)\cdot Q_{\leq r}(M'))\cdot q(M)^{1-n}q(M')^{-n} \\
&\sim& (2\pi i)^{n(n-1)(2n-3)/2}\prod_{r = 1}^{n-1} [Q_{\leq r}(M)\cdot q(M)^{-1}] \prod_{r' = 1}^{n-2} [Q_{\leq r'}(M'))\cdot q(M')^{-1}] \nonumber
\end{eqnarray}
where we have used the relation $Q_{\leq n-1}(M') = Q_{\det M'} \sim q(M')^2$ to obtain the more symmetric expression in the last line. Finally, the polarization gives us an isomorphism for each integer $m$:
$$R(M\otimes M')(m) \cong R(M\otimes M')^{\vee}(m+3-2n).$$
We write $P_{\leq r}(M) = Q_{\leq r}(M)\cdot q(M)^{-1}$, and likewise for $M'$.
It thus follows from \cite{deligne}, Formula (5.1.8) that
\begin{eqnarray}\label{cplus2}
c^+(R(M\otimes M')(m)) &\sim& (2\pi i)^{n(n-1)(m+3-2n)} c^{\pm}(R(M\otimes M')^{\vee}) \nonumber \\
&\sim& (2\pi i)^{\frac{n(n-1)}{2}(2m-2n+3)}\prod_{r = 1}^{n-1} P_{\leq r}(M)\prod_{r' = 1}^{n-2} P_{\leq r'}(M')
\end{eqnarray}

%CHECK POWERS OF PI

\subsection{Critical values when $n' = 1$}\label{character}

We repeat the above calculation but now assume $n' = 1$ and weaken the polarization hypothesis
for $M' = M(\Pi')$.  We assume that $M'$ is of weight $-\kappa$  and that there is
 a Dirichlet character $\alpha_0$
and a non-degenerate pairing as in (1.6.2) of \cite{harcrelle}:
\begin{equation}\label{chipol} R(M')\otimes R(M') \ra \Q(\alpha_0\cdot\varepsilon_{\K})(\kappa). \end{equation}
We assume moreover that there are Hecke characters $\chi$ and $\alpha$, as in
\cite{harcrelle} (especially \S 2.9 and \S 3.5), so that $\Pi' = BC(\chi)\cdot\alpha$,
with $\chi_{\infty}(z) = z^{-k}$, $\alpha_{\infty}(z) = z^{\kappa}$, so that $\Pi'_{\infty}(z) = (z/\bar{z})^{-k}\cdot z^{\kappa}$.
The motive $R(M')$ then has Hodge types $(k-\kappa,-k); (-k,k-\kappa)$.  (The parameter $\kappa$ is needed for
parity considerations and will in practice either be $0$ or $1$.)
The restriction of $\Pi'$ to the id\`eles of $\Q$ equals $\alpha_0$ multiplied by a power of the norm; thus
the notation of  \ref{chipol} is consistent with that introduced on p.\ 92 of
\cite{harcrelle}.
%** What is the relation between $\alpha$ and $\eta_0$?

The indices $b$ and $u$ only take the value $1$; we have
$r_1 = k-\kappa$ and $r_1^c = -k$, and the weight of $R(M\otimes M')$ is $n-1-\kappa$.   By definition $T(M(m),M')$ is the set of
pairs $(t,u) = (t,1)$ such that $p^c_t - m - k \geq \frac{n-\kappa - 2m}{2}$.  In other words

\begin{rem}\label{tmm1} When $n' = 1$, $T(M(m),M')$ consists exactly of the set of pairs $(t,1)$
such that
$$n - 2p^c_t  \leq -2k + \kappa.$$
Since the $p^c_t$ are decreasing,  $T(M(m),M')$ is therefore the set of $t \leq r(M')$ where
$$r(M') = \sup\{t ~|~ n - 2p^c_t  \leq -2k + \kappa\} $$%= \inf\{t~|~ 2k-\kappa \geq ***** \}$$
and it is of cardinality $r(M')$.   Let $s(M') = n - r(M')$.  If we assume $\kappa \equiv n \pmod{2}$, then
\begin{equation}\label{interval} -2k + \kappa \in [n - 2p^c_{r(M')}, n - 2p^c_{r(M')+1} - 2] = [n - 2q_{s(M')+1}, 2p_{s(M')} - n],
\end{equation}
where we are using \ref{polwt}.
%$\{(t,1)~|~ t+ \leq n \}$, of cardinality $\frac{n(n-1)}{2}$.
\end{rem}

The set of critical values is determined in \cite{harcrelle}, (3.3.8.1).  When $n > 1$ we choose $\kappa$ such that the weight $w$
of $R(M)\otimes R(M')$ is odd; then $\frac{w+1}{2}$ is always critical.  When $n = n' = 1$ we take $\kappa = 0$ (resp.\ $\kappa = 1$)
and $k \neq 0$; then the near-central point $s = 1$ (resp.\ the central point $s = 0$)  is always critical.
As in the previous section, we obtain the formula

\begin{eqnarray}\label{cplus3}  c^{\pm}(R(M\otimes M')^{\vee}) &\sim& \prod_{(t \leq r(M'))} Q_{n+1-t}(M)^{-1}Q_{1}(M')^{-1} \cdot \delta(M\otimes M')^{-1} \nonumber \\
&\sim& \prod_{t' > s(M')} Q_{t}(M)^{-1}Q_{1}(M')^{-1} \cdot \delta(M)^{-1}\cdot \delta(M')^{-n} \\
&\sim & (2\pi i)^{n(n-1)/2}\cdot P_{\leq s(M')}(M) \cdot  Q_{1}(M')^{-s(M')}\delta(M')^{-n} \nonumber
\end{eqnarray}

The polarizations define an isomorphism
$$R(M\otimes M') \cong R(M\otimes M')^{\vee}(1-n+\kappa)\otimes \Q(\alpha_0\cdot\varepsilon_{\K})$$
and so as before
\begin{equation}\label{cplus4}
\begin{aligned}
c^+(R(M\otimes M')(m)) &\sim (2\pi i)^{n(1-n+\kappa) + n(n-1)/2 + nm}Q(\alpha_0\cdot \varepsilon_{\K})^n\cdot P_{\leq s(M')}(M) \cdot  Q_{1}(M')^{-s(M')}\delta(M')^{-n} \\
&\sim (2\pi i)^{mn - n(n-1)/2}Q(\alpha_0\cdot \varepsilon_{\K})^n\cdot P_{\leq s(M')}(M) \cdot  (2\pi i)^{n\kappa} Q_{1}(M')^{-s(M')}\delta(M')^{-n}
\end{aligned}
\end{equation}

\subsection{Holomorphic automorphic forms on unitary groups and automorphic critical intervals when $n' = 1$}\label{unitaryperiods}

In \cite{harcrelle} and its successor \cite{har2007}, expressions for the critical values
of $L$-functions of unitary groups are derived from the determination of fields
of rationality of Eisenstein series for the Siegel parabolic, on the one hand, and from
the construction of holomorphic differential operators.  The latter are used to specify the
signature of the unitary similitude group whose associated Shimura variety realizes
the critical values as the Petersson square norm of a (rational) holomorphic automorphic
form, up to elementary factors.   Unfortunately the definitions of the
parameters are scattered in several places with confusing sign changes between successive appearances.
For the reader's convenience, the definitions of all relevant parameters and the
main results on the existence of differential operators are reproduced here.

{\it Page numbers refer to \cite{harcrelle}.}  The parameter $(\mu_1,\dots, \mu_n)$ of $\Pi$ is denoted $(a_1,\dots,a_n)$, starting on p.\ 104;
we replace the $a$'s with $\mu$'s in what follows.   There is also a central parameter $\mathbf{c}$ (denoted $c$ in \cite{harcrelle}; we use boldface
to distinguish the parameter from complex conjugation) and assumed (on p.\ 103) to
be of the same parity as $\sum_{i = 1}^n \mu_i$, and such that the motives realized on the degree $i$ cohomology of unitary group
Shimura varieties are of weight $i - \mc$ (see p.\ 104).   The main results of \cite{harcrelle} apply to self-dual $\Pi$ and
with our conventions there it follows (though this is not adequately explained) that $\mc = 0$.   Here we assume for simplicity that $\mc = \sum_{i = 1}^n \mu_i = C(\mu)$
in the notation of p.\ 105.   This implies that the integers $\mathcal{P}(\mu)$ and $\mathcal{Q}(\mu)$ defined on p.\ 105 are given by
$$\mathcal{P}(\mu) = -\mc  \quad\textrm{and}\quad \mathcal{Q}(\mu) = 0.$$

We fix a signature $(r,s)$, with $n = r + s$.  For any integer $k$ one defines
$$\mu(k) := (\mu_1 - k,\mu_2 - k,\dots, \mu_n - k;-nk+\mc),$$
$$\Lambda(\mu(k);r,s) := (\mu_{s+1} - k - s, \dots, \mu_{n} - k - s; \mu_1 - k + r, \dots, \mu_s - k + r; -nk+\mc).$$
(buried on p.\ 136 between (2.9.8) and (2.9.9), with $r$ and $s$ switched).

The motive $M(\Pi)$ attached abstractly to $\Pi$ has the Hodge types given by \ref{polwt}.  These cannot always be
realized as cohomological motives, for several reasons.  If $\Pi$ descends to an ($L$-packet of) automorphic
representations $\pi$ of a unitary group $U(V)$, $V$ is a Hermitian space
over $\K$ of signature $(r,s)$, then we obtain a motive with coefficients in $\K$, denoted $M'(\pi,V;r)$ on p.\ 118, in the cohomology of the
corresponding Shimura variety,  denoted $Sh(V)$ in \cite{harcrelle}.  The $L$-packet contains the holomorphic
type representation denoted $\pi_{r,s}$ above.  When $\Pi$ is self-dual, as in \cite{harcrelle},
this is sufficient \footnote{More precisely, to get (the Betti realization of) a legitimate motive $M'(\pi,V;r)$ we need to take the sum $M'(\pi,V;r) \oplus M'(\pi,-V;s)$,
as on p.\ 118, and this is necessary to define the archimedean Frobenius as an operator on automorphic forms; but for the de Rham structure
it suffices to work with $M'(\pi,V;r)$.}, provided there is a unitary group $U(V)$ of signature $(n-1,1)$ to which $\Pi$ descends.  When $n$ is even there
is a local obstruction and one may only be able to realize {\it even} exterior powers of the desired $M(\Pi)$, even in the self-dual case.
We therefore make the following hypothesis:

\begin{hyp}\label{descent}  For every signature $(r,s)$ with $r+s = n$, there is a unitary group $U(V)$ of signature $(r,s)$ such
that the cuspidal automorphic representation $\Pi$ of $G(\A_\K)$ descends to an $L$-packet of automorphic representations
$\{\pi\}$ of $U(V)$.  Moreover, the $L$-packet $\{\pi\}$ contains a member whose archimedean component is in the holomorphic
discrete series.
\end{hyp}

This hypothesis will be assumed implicitly in the remainder of this section.
In the statements of the main theorems in Section 6 the hypothesis will be included explicitly;
there will also be a discussion of substitutes.
\\

When $\Pi$ is conjugate self-dual but not self-dual -- more precisely, when the central character $\omega_{\Pi}$ of $\Pi$ is
not trivial, in particular when the integer $C(\mu)$ defined above is not even --  there may be an additional obstruction.
We need to choose a Hecke character $\xi$ such that
\begin{equation}\label{xii}  \tilde{\xi} := \xi/\xi^{c} =  \omega_{\Pi}^{-1}.  \end{equation}
Such a choice is always possible, cf.\ Proposition 1.2.4 of  \cite{CHL} (which cites the book of Harris and Taylor).
%\footnote{The notation $\tilde{\xi}$ is as in \cite{harcrelle}, p. 82.}
This is because the Shimura
variety is defined in terms of $GU(V)$ and not $U(V)$, and an automorphic representation $\pi$ of $GU(V)$ gives rise by base change
to a pair of automorphic representations $(\Pi,\xi)$ of $GL_n(\A_{\K}) \times GL_1(\A_{\K})$.  The parameter $\mc$ is determined by the
archimedean component $\xi_{\infty}$ of $\xi$:
$$\xi_{\infty}(t) = t^{-\mc}.$$
The Hodge type of the motive realized in the cohomology of $Sh(V_1)$, with $(r,s) = (n-1,1)$ (see \cite{harcrelle}, Example 2.2.5) -- call it $M'(\pi,V_1;n-1)$ (and don't
confuse this $M'$ with the $M'$ for $GL_{n-1}$) --
is called the {\it principal Hodge type} in \S 2.2.9 of \cite{harcrelle}; it is given by
\begin{equation}\label{prinhodge} (p_i(\mu),q_i(\mu)) = (p_i + \mathcal{P}(\mu),q_i + \mathcal{Q}(\mu)) = (p_i - \mc,q_i); \end{equation}
where $(p_i,q_i)$ are attached to $\Pi$ as in \ref{polwt}.
More generally, if $V_s$ has signature $(r,s)$, the motive $M'(\pi_{r,s},V_s;r)$ attached to the (holomorphic)
representation $\pi_{r,s}$ of $GU(V)$ that is nearly equivalent to $\pi$, in the
sense that $\pi$ and $\pi_{r,s}$ have the same base change $(\Pi,\xi)$  to  $GL_n(\A_{\K}) \times GL_1(\A_{\K})$, is expected to bear the following relation
to the desired $M(\Pi)$:
\begin{equation}\label{wedge}  M'(\pi_s,V_s;r) \cong \wedge^r(M(\Pi)^{\vee}) \otimes M(\xi^c) \end{equation}
This is known to be true in most cases at the level of Galois representations (cf.\ \cite{CHL}, Proposition 4.3.8, as well as the more complete results
of Shin and Scholze--Shin), but in the absence of the Tate conjecture one can't even give a precise definition to the rank one motive $M(\xi^c)$, much
less prove a result such as \ref{wedge}.  The best we can do is to show how the expressions for special values obtained by analytic methods are
products of expressions that resemble those that would be expected if there were a complete theory of motives.   The choice of $\xi$ should have
no bearing on the expression of the special value of the $L$-function of $\Pi$; we return to this point
below.\footnote{Actually, the Tate conjecture is more than we need.
It would be enough to show that an isomorphism at the level of Galois representations implies an isomorphism at the level of motives for
absolute Hodge cycles.}

 We can now restate Corollary 3.3.8 of \cite{harcrelle}, which is the main result
 on the existence of holomorphic differential operators.  In what follows,
 notation is as on p.\ 145.   In particular, $\mathcal{E}_{m,\kappa}$ is a line bundle
 on the Shimura variety $Sh(n,n)$ (attached to a quasi-split unitary similitude
 group of size $2n$ and $E_{\Lambda,\Lambda^{\sharp}(\kappa)}$ is an automorphic
 vector bundle on the Shimura variety $Sh(V,-V)$ (essentially $Sh(V) \times Sh(-V)$).

\begin{lem}\label{diffop}   Let $n' = 1$ and $\Pi'$, $M' = M(\Pi')$, $k$, and $\kappa$ be as in \ref{character}.  Let $(r,s) = (r(M'),s(M'))$ and
let $V$ be a Hermitian space of signature $(r,s)$.  Let $m$ be a critical value of $R(M\otimes M')$ to the right
of (the central point) $\frac{1}{2}(n-\kappa)$. There is a differential operator
$$\Delta(m,\kappa,\Lambda): \mathcal{E}_{m,\kappa} \rightarrow E_{\Lambda,\Lambda^{\sharp}(\kappa)}$$
where $\Lambda = \Lambda(\mu(k);r,s)$ on the Shimura variety $Sh(V,-V)$.
\end{lem}
\begin{proof} Corollary 3.3.8 of \cite{harcrelle} asserts that such a differential operator exists for the pair $(r,s)$ provided
$$\frac{1}{2}(n-\kappa) \leq m \leq \min(q_{s+1}(\mu) + k-\kappa - \mathcal{Q}(\mu)$$
and
$$p_s(\mu) - k - \mathcal{P}(\mu)) = \min(q_{s+1} + k - \kappa, p_s - k),$$
where the second equality is \ref{prinhodge}.
In particular, we have $2q_{s+1} - 2k \geq n - \kappa; 2p_s - 2k \geq n - \kappa$; in other words
$$-2k + \kappa \in [n - 2q_{s+1}, 2p_{s} - n].$$
By \ref{interval} this is true exactly when $s = s(M')$.
\end{proof}

We now summarize informally the constructions of \S 3 of \cite{harcrelle} and the corresponding
sections of \cite{har2007}.  We let $\alpha$ and $\chi$ be Hecke characters of $\A_\K$ and $U(1)(\A)$,
respectively, with parameters $\kappa$ and $k$ as above.  Let $\pi$ be
an anti-holomorphic representation of  $GU(V)(\A)$ and identify its contragredient $\pi^{\vee}$ with
an anti-holomorphic representation of $GU(-V)(\A)$ as in \cite{harcrelle} (this identification is
spread out over several sections; the most important statements are Corollary 2.5.11.2, Lemma 2.8.8,  and (2.8.9.5)).
If $Eis \in \Gamma(Sh(n,n),\mathcal{E}_{m,\kappa})$ -- in practice a holomorphic Eisenstein series --
and  $f \in \pi$, $f' \in \pi^{\vee}$, and $f_{\chi} = f\otimes \chi\circ \det$, $f'_{\chi} = f'\otimes \chi^{-1}\circ \det$, then
the value of the standard $L$-function at $s = m$ can be identified, up to elementary and local terms, with
the cup product (Serre duality) pairing
\begin{equation}\label{cup} \Delta(m,\kappa,\Lambda) \cup [f_{\chi}\otimes (f'_{\chi}\otimes \alpha^{-1}\circ \det)]  \ra \C \end{equation}
This pairing is well defined provided (in the notation of p.\ 147)
\begin{equation}\label{vee} f_{\chi}\otimes (f'_{\chi}\otimes \alpha^{-1}\circ \det) \in \overline{H}^{rs}(Sh(V,-V), \Omega^{top}_{Sh}\otimes E_{\Lambda,\Lambda^{\sharp}(\kappa)}^{\vee}).
\end{equation}
where $\Omega^{top}_{Sh}$ is the canonical bundle on $Sh(V,-V)$ and $\overline{H}$ is the version of coherent cohomology
used in \cite{harcrelle}. Now there is an isomorphism of automorphic vector bundles:
\begin{equation}\label{serre}
\Omega^{top}_{Sh}\otimes E_{\Lambda,\Lambda^{\sharp}(\kappa)}^{\vee} \cong E_{\mu^{\vee}(-k),\mu^{\vee}(k-\kappa)}
\end{equation}
Here $\mu^{\vee} = (-\mu_n,-\mu_{n-1},\dots,-\mu_1;-\mc)$.  This isomorphism is stated in the third displayed formula on p.\ 150
of \cite{harcrelle} with $\mu$ in place of $\mu^{\vee}$, which is legitimate in that setting because of the self-duality hypothesis.
Now \ref{serre} implies that \ref{vee} holds provided $\Pi$ has cohomology with respect to $(E^{\sf unt}_{\mu^{\sf v}}\otimes E^{\sf unt}_{\mu})$, which is dual to the hypotheses we have been using thus far.  However,
$$R(M(\Pi)\otimes M(\Pi')) \equiv R(M(\Pi^c)\otimes M(\Pi^{\prime,c})) = R(M(\Pi^{\vee})\otimes M(\Pi^{\prime,c})).$$
Rather than change notation for $\Pi$, we replace $\Pi'$ by $\Pi^{\prime,c}$ and apply the results of \cite{harcrelle} to $\Pi^{\vee}$. Note in particular that
$$-2k + \kappa \in [n- 2q_{s+1}(\Pi^{\vee}), 2p_s(\Pi^{\vee}) - n] = [n-2q_{s+1}^c(\Pi),2p_s^c(\Pi) - n]$$
in the obvious notation.

Thus we now reverse the notation of \ref{character} and write
$$\Pi^{\prime,c} = BC(\chi)\cdot\alpha, ~~M^{\prime,c}= M(\Pi^{\prime,c}). $$
where the parameters $k$ and $\kappa$ are attached to $\chi$ and $\alpha$ rather than to $\Pi'$. Of course $M(\Pi^{\vee})$ has cohomology with respect to $E^{\sf unt}_{\mu^{\sf v}}\otimes E^{\sf unt}_{\mu}$. The signature $(r(M^{\prime,c}),s(M^{\prime,c}))$ is attached to the pair $(M(\Pi^{\vee}) = M(\Pi)^c, M^{\prime,c})$,
so we will write
$$(r(M^{\prime,c}),s(M^{\prime,c})) = (r(M^{\prime,c};\Pi^{\vee}),s(M^{\prime,c};\Pi^{\vee})).$$
Replacing $M^{\prime,c}$ by $M'$ replaces $-2k+ \kappa \in [n-2q_{s+1}^c(\Pi),2p_s^c(\Pi) - n]$ by
$2k-\kappa \in [n-2p^c_s(\Pi),2q^c_{s+1}(\Pi)-n] = [n- 2q_{r+1}(\Pi), 2p_r(\Pi) - n]$.   Thus
$$(r(M^{\prime,c};\Pi^{\vee}),s(M^{\prime,c};\Pi^{\vee})) = (s(M';\Pi), r(M';\Pi)).$$

The above calculations apply and give us the following version of Theorem 3.5.13 of \cite{harcrelle} and Theorem 4.3. Recall that when $\beta$ is a Hecke character, we have defined $\check{\beta} = (\beta^c)^{-1}$ on p.\ 82 of \cite{harcrelle}. This should not be confused with $\beta^{\vee} = \beta^{-1}$. In what follows, we let $\pi = \pi(\Pi^{\vee},\xi)$  be an anti-holomorphic automorphic representation of $GU(V)(\A)$, for some Hermitian space $V$ to be specified, whose base change equals $(\Pi^{\vee},\xi)$. For any integer $j \in [0,n]$ define
$P^{(j)}(\Pi^{\vee},\xi)$ to be the period denoted $P^{(j)}(\Pi,\star,\beta)$ in \cite{harcrelle}, (2.8.2): It is the Petersson inner product with itself of a rationally normalized holomorphic form on $GU(V)$ whose restriction to $U(V)$ is weakly equivalent to $\pi$ on a Hermitian space of signature $(n-j,j)$. We set
\begin{equation}\label{petersson}
\textrm{\framebox{$P^{(j)}(\Pi) :=
(2\pi)^{-\mc}P^{(j)}(\Pi^{\vee},\xi).$}}
\end{equation}
This can be shown to be independent of the choice of $\xi$ when appropriate $L$-functions have non-vanishing critical values (see \ref{similitudechar} for an explanation).  In what follows, we will apply this when $j = s(M^{\prime,c};\Pi^{\vee}) = r(M';\Pi)$. We write
$$G(\Pi^{\prime,c}) =  (2\pi i)^k p(\check{\chi}^{(2)}\cdot \check{\alpha},1).$$
The notation $\alpha_{f,0}$ for the restriction of (the finite part of) a Hecke character to the id\`eles of $\Q$, and $\mathcal{G}$ for the Gau\ss {} sum, are as in \S 3.

\begin{thm}\label{crelle}  Let $m \geq \frac{1}{2}(n-\kappa)$ be a critical value of $L(s,R(M(\Pi)\otimes M(\Pi'))) = L(s,R(M(\Pi^{\vee})\otimes M(\Pi^{\prime,c})))$. Let $(r,s) = (r(M^{\prime,c}),s(M^{\prime,c}))$ and let $V$ be a Hermitian space of dimension $n$ over $\K$ with signature $(r,s)$. Let $\xi$ be a Hecke character of $\K$ satisfying \ref{xii}. Let $G(\Pi^{\prime,c}) =  (2\pi i)^k p((\check{\chi}^{(2)}\cdot \check{\alpha}),1)$.  Then
  \begin{equation*}\begin{aligned}
 L(m,R(M(\Pi)\otimes M(\Pi'))) &\sim_{\K E(\Pi) E(\Pi')}
  (2 \pi i)^{mn - \frac{n(n-1)}{2}} \mathcal{G}(\varepsilon_{\K,f})^{\lceil\frac{n}{2}\rceil}P^{(s)}(\Pi) [(2\pi i)^{\kappa}\mathcal{G}(\alpha_{f,0})]^s G(\Pi^{\prime,c})^{n-2s} \\
 \sim_{\K E(\Pi) E(\Pi')}
  &(2 \pi i)^{mn - \frac{n(n-1)}{2}} \mathcal{G}(\varepsilon_{\K,f})^{\lceil\frac{n}{2}\rceil}P^{(r(M';\Pi))}(\Pi) [(2\pi i)^{\kappa}\mathcal{G}(\alpha_{f,0})]^{r(M';\Pi)} G(\Pi^{\prime,c})^{n-2r(M';\Pi)}
 \end{aligned}
 \end{equation*}
 \end{thm}

%** (1) ~~ The term $[(2\pi i)^{\kappa}\mathcal{G}(\alpha_{0,f})]^s$ must come from the polarization.  In the applications, $\alpha_0$ is the $n$-th power of the quadratic
%character and $\kappa = 1$ for $n$ odd, $0$ for $n$ even, and the resulting power of $2 \pi i$ must be absorbed into the shift by $1/2$.
%	(2) ~~ The power of $\pi$ is adapted to $\chi$ of type $z^{-k}$.

\begin{op}\label{comparison} Now let us compare this to the critical value predicted by Deligne's conjecture,
namely \ref{cplus4}:
$$(2\pi i)^{mn - \frac{n(n-1)}{2}}Q(\alpha_0\cdot \varepsilon_{\K})^n\cdot P_{\leq s(M')}(M) \cdot  (2\pi i)^{n\kappa} Q_{1}(M')^{-s(M')}\delta(M')^{-n} .$$
Or rather compare
\begin{equation}\mathcal{G}(\varepsilon_{\K,f})^{\lceil\frac{n}{2}\rceil}(2\pi)^{-\mc}P^{(j)}(\Pi^{\vee},\xi) [(2\pi i)^{\kappa}\mathcal{G}(\alpha_{0,f})]^{r(M';\Pi)} G(\Pi^{\prime,c})^{n-2r(M';\Pi)} \tag{A}
\end{equation}
to
\begin{equation} Q(\alpha_0\cdot \varepsilon_{\K})^n\cdot P_{\leq s(M')}(M) \cdot  (2\pi i)^{n\kappa} Q_{1}(M')^{-s(M')}\delta(M')^{-n}. \tag{M} \end{equation}
Applying \ref{wedge}, and bearing in mind that the base change of $\pi$ is $(\Pi^{\vee},\xi^{-1})$, we see that $i^{\mc}(2\pi)^{-\mc}P^{(j)}(\Pi^{\vee},\xi) $ is the correct normalization for the Petersson inner product of a rational class in the bottom stage of the Hodge filtration of
$\wedge^{s(M',\Pi)}M(\Pi)\otimes M(\xi^c)$.  (The presence of the power $i^{\mc}$, which will disappear in our applications in any case,
is  explained in \cite{harcrelle}, Lemma 2.8.8.)  On the other hand, $P_{\leq s(M')}(M) = Q_{\leq s(M')}(M)\cdot q(M)^{-1}$ is the product of the Petersson norm
of a rational class in the bottom stage of the Hodge filtration of $\wedge^{s(M',\Pi)}M(\Pi)$, and plausible arguments can be made to justify
identifying $(2\pi i)^{-\mc}q(M)^{-1}$ with the Petersson norm of $M(\xi^c)$ (cf. \cite{harcrelle}, (3.7.9.3)-(3.7.9.5), where however it is assumed
that $\mc = 0$).

\begin{rem}\label{similitudechar}  Note that the automorphic side (A) of the comparison involves a Petersson norm on the unitary similitude group, and therefore
depends in principle on the auxiliary Hecke character $\xi$ defined in \ref{xii}; we should really write out the expression
$(2\pi)^{-\mc}P^{(r(M';\Pi))}(\Pi,\xi)$ rather than the abbreviation $P^{(r(M';\Pi))}(\Pi)$.  The motivic side (M), on the other hand, is independent of this choice.
Now Theorem \ref{crelle} shows that replacing $\xi$ with $\xi'$, also assumed to satisfy \ref{xii}, gives another expression for the same critical value;
in particular, if the critical value does not vanish, then we see that $(2\pi)^{-\mc}P^{(r(M';\Pi))}(\Pi,\xi) \sim (2\pi)^{-\mc'}P^{(r(M';\Pi))}(\Pi,\xi')$ (where $\sim$
is taken over a simultaneous field of definition for $\Pi, \xi, \xi'$ and $c'$ bears the same relation to $\xi'$ that $c$ bears to $\xi$.  The ``plausible
arguments" mentioned above give a motivic interpretation for this identity.
\end{rem}

It thus remains to justify an identification
\begin{equation}\label{n1}  Q(\alpha_0\cdot \varepsilon_{\K})^n\cdot (2\pi i)^{n\kappa} Q_{1}(M')^{-s(M')}\delta(M')^{-n} \sim
\mathcal{G}(\varepsilon_{\K,f})^{\lceil\frac{n}{2}\rceil}[(2\pi i)^{\kappa}\mathcal{G}(\alpha_{0,f})]^{r(M';\Pi)} G(\Pi^{\prime,c})^{n-2r(M';\Pi)}.\end{equation}
We provide such a justification, at least up to Galois conjugation over $\K$, in the next section.

%*****************
%It looks like  $Q(\alpha_0\cdot \varepsilon_{\K})$ is related to $\mathcal{G}(\alpha_{0,f})^s$,
%$G(\Pi^{\prime,c})^{n-2r(M';\Pi)}$ to $Q_{1}(M')^{-s(M')}\delta(M')^{-n}$, and $P_{\leq s(M')}$ to $(2\pi)^{-\mc}P^{(r(M';\Pi))}(\Pi)$.  But now this is $\Pi^{\vee}$...
%What is the motive for $\Pi^{\vee}$ and $\xi$ in $Sh(s(M';\Pi),r(M';\Pi))$?  By Theorem 2.7.6.1 it is
%$\wedge^{s(M';\Pi)}(M(\Pi))$ up to a twist.     So it looks like it works out after all!!!!

\end{op}

%Now to end this for the day, assume $\kappa = 0$, insert an index $i$ on $k$, and return to the formula
%$-2k_i \in [n- 2q_{s+1}(\Pi^{\vee}), 2p_s(\Pi^{\vee}) - n] = [n-2q_{s+1}^c(\Pi), 2p_s^c(\Pi) - n]$.
%In terms of Lemma 4.2.4 of \cite{haradj} we have (for $n$ even)
%$-2k_i = -2b_i - n + 2i \in [n-2p_i,n-2p_{i+1} - 2]$.  And $p_i = q_{n+1 - i}^c$ so it looks like $i = r$.

\subsection{The case of a Hecke character}

We now assume $n = 1$ and $\Pi$ the trivial character $\triv$, so $M = \Q(0)$.
We have $r(M') = r(M';\triv) = 1$ if $-2k+\kappa  \geq 1$, $r(M') = 0$ if $-2k+\kappa \leq 0$.  Then if $R(M')(m)$ is critical, \ref{cplus4} becomes
\begin{equation}\label{cplus11}
c^+(R(M')(m)) \sim \begin{cases} (2\pi i)^{m+\kappa} Q(\alpha_0\cdot \varepsilon_{\K}) \cdot \delta(M')^{-1}, & \text{if } 2k - \kappa \leq -1 \\
 (2\pi i)^{m + \kappa} Q(\alpha_0\cdot \varepsilon_{\K}) \cdot Q_1(M')^{-1}\delta(M')^{-1} &\text{if } 2k - \kappa  \geq 0
\end{cases}
\end{equation}

Since Deligne's conjecture is known for critical values of $L(s,R(M'))$, we thus have
\begin{equation}\label{damarell}  L(m,R(M')) \sim c^+(R(M')(m)) \sim \begin{cases} (2\pi i)^{m+\kappa} Q(\alpha_0\cdot \varepsilon_{\K}) \cdot \delta(M')^{-1}, & \text{if }
2k - \kappa \leq -1 \\
 (2\pi i)^{m+\kappa} Q(\alpha_0\cdot \varepsilon_{\K}) \cdot Q_1(M')^{-1}\delta(M')^{-1} &\text{if } 2k - \kappa  \geq 0
\end{cases}
\end{equation}
Here $\sim$ means $\sim_{\Q(M');\K}$ as in \cite{harcrelle}, because we are going to compare with the result of \ref{crelle}:
\begin{equation}\label{hecke}
L(m,R(M')) \sim (2\pi i)^m [(2\pi i)^{\kappa}\mathcal{G}(\alpha_{0,f})]^{r(M')}G(\Pi^{\prime,c})^{1-2r(M')}
\end{equation}
Assume $r(M') = 0$, so $2k - \kappa \geq 0$.  Then comparing \ref{hecke} and \ref{damarell} we find
$$Q(\alpha_0\cdot \varepsilon_{\K})\cdot Q_1(M')^{-1}\delta(M')^{-1} \sim G(\Pi^{\prime,c}).$$

Suppose for the moment $\alpha = 1$, $\kappa = 0$.    So we
compare $$Q(\varepsilon_{\K})^nQ_{1}(M')^{-s(M')}\delta(M')^{-n}$$ on the Deligne side with
$$\mathcal{G}(\varepsilon_{\K,f})^{\lceil\frac{n}{2}\rceil}G(\Pi^{\prime,c})^{n-2r(M';\Pi)} \sim \mathcal{G}(\varepsilon_{\K,f})^{\lceil\frac{n}{2}\rceil}
\cdot [Q(\varepsilon_{\K})\cdot Q_1(M')^{-1}\delta(M')^{-1}]^{n-2r(M')}$$
on the automorphic side.   Substituting $Q_1(M') \sim \delta(M')^{-2}$, we find
$Q(\varepsilon_{\K})^n \delta(M')^{2s-n}$ on the Deligne side and
$\mathcal{G}(\varepsilon_{\K,f})^{\lceil\frac{n}{2}\rceil}Q(\varepsilon_{\K})^{n-2r} \cdot \delta(M')^{n-2r}$, which coincide up to a power
of the Gau\ss {} sum.

The case of non-trivial $\alpha$ is only slightly more complicated, and will not be used in the sequel.

%** NOW if $\kappa = 0$ we get

%$$c^+(R(M\otimes M')) \sim (2\pi i)^{-n(n-1)/2}Q(\alpha_0\cdot \varepsilon_{\K})\cdot P_{\leq s(M')}( \cdot  Q_{1}(M')^{-s(M')}\delta(M')^{-n},$$
%so
%$$c^+(R(M\otimes M')(\frac{n}{2})) \sim (2\pi i)^{\frac{n}{2}}P_{\leq s(M')} \cdot  Q_{1}(M')^{-s(M')}\delta(M')^{-n}Q(\alpha_0\cdot \varepsilon_{\K}).$$
%Now $\delta(M')^{-1} \sim d(M')^{\frac 12}Q_{1}(M')^{1/2}$ so this is rewritten
%$$c^+(R(M\otimes M')(\frac{n}{2})) \sim (2\pi i)^{\frac{n}{2}}P_{\leq s(M')} \cdot  Q_{1}(M')^{\frac{n}{2}-s(M')}[d(M')^{\frac{n}{2}}Q(\alpha_0\cdot \varepsilon_{\K})].$$
%Comparing this to
%$$L(\frac{n}{2},R(M(\Pi)\otimes M(\chi_i))) \sim (2\pi i)^{\frac{n}{2}} G(i,\mathbf{\chi})P^{(n-i)}(\Pi) = (2\pi i)^{\frac{n}{2}} G(i,\mathbf{\chi})P^{(n-i)}(\Pi),$$
%it seems that $r(M') = i$ when $r = r(M')$ and
%$-2k+\kappa \in [n-2p^c_{r-1}+2,n-2p^c_{r}]$, and that we need to identify
%$$G(r,\mathbf{\chi}) \sim Q_{1}(M')^{\frac{n}{2}-s(M')}[d(M')^{\frac{n}{2}}Q(\alpha_0\cdot \varepsilon_{\K})].$$
%But
%$$G(r,\mathbf{\chi}) = [(2\pi i)^{k_r}\cdot p((\chi^{(2)})^{\vee},1)]^{2r-n} = [(2\pi i)^{k_r}\cdot p((\chi^{(2)})^{\vee},1)]^{n-2s}$$
%so this comes down to
%$$Q_1(M') \sim [(2\pi i)^{k_r}\cdot p((\chi^{(2)})^{\vee},1)]^2; ~~d(M')^{\frac{n}{2}} \sim Q(\alpha_0\cdot \varepsilon_{\K}).$$

\section{Rational Eisenstein classes of abelian type}
\subsection{Boundary cohomology}
Recall the abelian representation $\Pi'$ from Sect.\ \ref{sect:Pi'} and the space $S_{n-1}$ from Sect.\ \ref{sect:autcoh}. We denote by $\overline S_{n-1}$ the Borel-Serre compactification of $S_{n-1}$ and by $\partial\overline S_{n-1}$ its boundary, cf.\ \cite{boserre}, \cite{roh}. The face coming from the Borel subgroup $B'$ of $G'$ is denoted $\partial_{B'}\overline S_{n-1}$. It is given by
$$\partial_{B'}\overline S_{n-1}=B'(\K)\backslash G'(\A_\K)/K'.$$
Since the only automorphic representations of the torus $T'\subset B'$ are its Hecke characters $\chi$ it follows that
$$H^q(\partial_{B'}\overline S_{n-1},\E_\lambda)=\bigoplus_{\chi }H^q(\g',K',\textrm{Ind}^{G'(\A_\K)}_{B'(\A_\K)}[\chi]\otimes E_\lambda),$$
see e.g. Harder \cite{harder1447}. Observe that induction is unitary here and that $\pi_0(G'(\C))=\pi_0(B'(\C))=0$, which simplifies things. Recall the character $\tilde\tau=BC(\chi_1)\gamma\otimes...\otimes BC(\chi_{n-1})\gamma$ from Sect.\ \ref{sect:cuspEisinj} and let $W(G')$ be the Weyl group of $G'$. For $w\in W(G')$, we let $\tilde\tau^w$ be the representation
$$\tilde\tau^w:=BC(\chi_{w^{-1}(1)})\gamma\otimes...\otimes BC(\chi_{w^{-1}(n-1)})\gamma$$
of $T'(\A_\K)$, where $w$ is viewed as a permutation. In particular, we may write
$$H^q(\partial_{B'}\overline S_{n-1},\E_\lambda)=\bigoplus_{w\in W(G')}H^q(\g',K',\textrm{Ind}^{G'(\A_\K)}_{B'(\A_\K)}[\tilde\tau^w]\otimes E_\lambda)\oplus\bigoplus_{\substack{\chi\neq\tilde\tau^w\\ \forall w\in W(G')} }H^q(\g',K',\textrm{Ind}^{G'(\A_\K)}_{B'(\A_\K)}[\chi]\otimes E_\lambda),$$
and there is no weak $G'(\A_f)$-intertwining between the two sums. Hence, for degree $q=b_{n-1}$ we obtain
\begin{equation}\label{eq:boundary_coh}
H^{b_{n-1}}(\partial_{B'}\overline S_{n-1},\E_\lambda)=\bigoplus_{w\in W(G')}\textrm{Ind}^{G'(\A_f)}_{B'(\A_f)}[\tilde\tau_f^w]\oplus\bigoplus_{\substack{\chi\neq\tilde\tau^w\\ \forall w\in W(G')} }H^{b_{n-1}}(\g',K',\textrm{Ind}^{G'(\A_\K)}_{B'(\A_\K)}[\chi]\otimes E_\lambda),
\end{equation}
and the summands in the first sum are all irreducible.

\begin{prop}\label{prop:rational_Eisenstein}
For all $\sigma\in$ \textrm{\emph{Aut}}$(\C)$, the following diagram of $G'(\A_f)$-equivariant homomorphisms commutes:

\begin{displaymath}
\xymatrix{
H^q(\g',K',\Pi'\otimes E_\lambda) \ar@{^{(}->}[r]\ar[d]^{\sigma^*} & H^q(\overline S_{n-1},\mathcal E_\lambda) \ar[d]^{\sigma^*} \ar[r]^{res^q} &
H^q(\partial\overline S_{n-1},\mathcal E_\lambda) \ar[d]^{\sigma^*} \ar[r]^{res_{B'}^q} & H^q(\partial_{B'}\overline S_{n-1},\mathcal E_\lambda) \ar[d]^{\sigma^*}\\
H^q(\g',K',{}^\sigma\Pi'\otimes {}^\sigma\! E_\lambda) \ar@{^{(}->}[r] & H^q(\overline S_{n-1},{}^\sigma\!\mathcal E_\lambda) \ar[r]^{res^q} &
H^q(\partial\overline S_{n-1},{}^\sigma\!\mathcal E_\lambda) \ar[r]^{res_{B'}^q} & H^q(\partial_{B'}\overline S_{n-1},{}^\sigma\!\mathcal E_\lambda)}
\end{displaymath}
Furthermore, there are sections $s_{B'}$ to the composition $r_{B'}$ of the morphisms in the rows of the diagram,
$$s_{B'}:H^{b_{n-1}}(\partial_{B'}\overline S_{n-1},{}^\sigma\!\mathcal E_\lambda) \longrightarrow H^{b_{n-1}}(\g',K',{}^\sigma\Pi'\otimes {}^\sigma\! E_\lambda).$$
\end{prop}
\begin{proof}
We only prove the last assertion. Recall the archimedean vectors $\xi_{\Pi'_\infty,\underline{j},\beta}\in W(\Pi'_\infty)$ from Sect.\ \ref{sect:Theta}. Then, for every non-zero $\xi_f\in W(\Pi'_f)$, we let
$$E_{\Pi',\underline{j},\beta}:=W^{-1}(\xi_{\Pi'_\infty,\underline{j},\beta}\otimes\xi_f)\in\Pi'$$
be the corresponding Eisenstein series. Hence,
$$\sum_{\underline j,\beta}\left(X'^*_{\underline j}\otimes E_{\Pi',\underline{j},\beta}\otimes e'_\beta\right)\in H^{b_{n-1}}(\g',K',\Pi'\otimes E_\lambda)$$
is a non-zero cohomology class by the choice of $\xi_{\Pi'_\infty,\underline{j},\beta}$. Via the inclusion
$$H^{b_{n-1}}(\g',K',\Pi'\otimes E_\lambda) \hookrightarrow H^{b_{n-1}}(\overline S_{n-1},\mathcal E_\lambda)$$
it is mapped onto
$$\sum_{\underline j,\beta}\left(E_{\Pi',\underline{j},\beta}\otimes e'_\beta\right) dx_{\underline j}\in H^{b_{n-1}}(\overline S_{n-1},\mathcal E_\lambda) $$
and finally, when restricting the the Borel-stratum, we obtain
$$\sum_{\underline j,\beta}\left(E_{\Pi',\underline{j},\beta}|_{B'}\otimes e'_\beta\right) dx_{\underline j}\in H^{b_{n-1}}(\partial_{B'}\overline S_{n-1},\mathcal E_\lambda) .$$
Here, $E_{\Pi',\underline{j},\beta}|_{B'}$ is the constant term of the Eisenstein series $E_{\Pi',\underline{j},\beta}$ with respect to the Borel subgroup $B'$, see \cite{schwLNM} Satz 1.10. According to \cite{moewal} Prop.\ II.7 this constant term can be written as
$$E_{\Pi',\underline{j},\beta}|_{B'}=\sum_{w\in W(G')} M(w,\tilde\tau) f_{\Pi',\underline{j},\beta},$$
where $f_{\Pi',\underline{j},\beta}$ is a $K'$-finite section in $\textrm{Ind}^{G'(\A_\K)}_{B'(\A_\K)}[\tilde\tau]$, used to define the Eisenstein series $E_{\Pi',\underline{j},\beta}$ and $M(w,\tilde\tau)$ is the corresponding intertwining operator,
$$M(w,\tilde\tau): \textrm{Ind}^{G'(\A_\K)}_{B'(\A_\K)}[\tilde\tau]\ra \textrm{Ind}^{G'(\A_\K)}_{B'(\A_\K)}[\tilde\tau^w]$$
cf.\ \cite{moewal} II.6. As a consequence, the Eisenstein cohomology class, when restricted to the Borel-stratum, yields the class
$$\sum_{w\in W(G')}\sum_{\underline j,\beta}\left(M(w,\tilde\tau)f_{\Pi',\underline{j},\beta}\otimes e'_\beta\right) dx_{\underline j},$$
which, via the isomorphism \eqref{eq:boundary_coh}, lies in $\bigoplus_{w\in W(G')}\textrm{Ind}^{G'(\A_f)}_{B'(\A_f)}[\tilde\tau_f^w]$. Here we observe that there is no weak intertwining of the two sums in \eqref{eq:boundary_coh}. This class is non-zero, since all Eisenstein series $E_{\Pi',\underline{j},\beta}$ are holomorphic at their point of evaluation $s=0$, cf.\ the proof of Thm.\ 4.11 in \cite{schwLNM}. Hence, since all representations $\textrm{Ind}^{G'(\A_f)}_{B'(\A_f)}[\tilde\tau_f^w]$ are irreducible, the various summands of the image of $r_{B'}$ are either trivial or the whole target space. The assertion follows now from the irreducibility of the source $\Pi'_f=H^{b_{n-1}}(\g',K',\Pi'\otimes E_\lambda)$.
\end{proof}

\subsection{A rationality result}
Recall Prop.\ \ref{prop:regalg}. It is clear that $H^{b_{n-1}}(\partial_{B'}\overline S_{n-1},\mathcal E_\lambda)$ inherits from $E_\lambda$ a natural $\Q(E_\lambda)$-structure. By  Prop.\ \ref{prop:sigmatwistabelian} and the temperedness of $\Pi'_\infty$, there is the equality of number fields $\Q(E_\lambda)=\Q(\Pi'_\infty)$. Hence, by strong multiplicity one for isobaric automorphic representations, $\Q(E_\lambda)\subseteq \Q(\Pi'_f)$. Next, observe that by Prop.\ \ref{prop:regalg}, $\Pi'_f=H^{b_{n-1}}(\g',K',\Pi'\otimes E_\lambda)$ carries a natural $\Q(\Pi'_f)$-structure. We deduce the following corollary.

\begin{cor}\label{cor:rational_Eisenstein}
The morphism $r_{B'}$ is $\Q(\Pi'_f)$-rational, i.e., it maps the natural $\Q(\Pi'_f)$-structure of the cohomology $H^{b_{n-1}}(\g',K',\Pi'\otimes E_\lambda)$ to the natural $\Q(\Pi'_f)$-structure of $H^{b_{n-1}}(\partial_{B'}\overline S_{n-1},\mathcal E_\lambda)$. The same holds true for the section $s_{B'}$.
\end{cor}
\begin{proof}
This is clear by the commutativity of the diagram in Prop.\ \ref{prop:rational_Eisenstein}.
\end{proof}

\begin{rem}
We would like to point out that for regular coefficients $E_\lambda$, Thm.\ 7.23 in Grobner--Raghuram \cite{grob-rag-I} should provide an alternative approach to Cor.\ \ref{cor:rational_Eisenstein}.
\end{rem}

Recall from Section \ref{sect:cuspEisinj} that the rational structure on $\Pi'_f$ is given by the natural
rational structure on the induced representation.  Let $F \mapsto F|_{B'}$
denote the constant term on functions (Eisenstein series) $F \in \Pi'$, and let
$F|_{B'}(\tilde{\tau})$ denote its $\tilde{\tau}$-component with respect to the action of the maximal
torus in $B'$.  It follows from the proof of Proposition \ref{prop:rational_Eisenstein} that

\begin{lem}\label{constterm}
Let $F = F_{\infty}\otimes F_f  \in \Pi'$ be a function with $F_{\infty}$ in the $K'$-type contributing to the cohomology space
$H^{b_{n-1}}(\g',K',\Pi'_{\infty}\otimes E_\lambda)$. Let $L \supseteq \Q(\Pi'_f)$ be a field extension. Then $F$ defines an $L$-rational
cohomology class $[F]$ if and only if $F_f$, viewed as an element of
${\rm Ind}_{B'(\A_f)}^{G'(\A_f)} [\tilde{\tau}_f] $, takes values in $L$, which is furthermore equivalent to  $F|_{B'}(\tilde{\tau})$,
viewed as an element of
${\rm Ind}_{B'(\A_f)}^{G'(\A_f)}[\tilde{\tau}_f] $, taking values in $L$.  Moreover, for all $\sigma \in {\rm Aut}(\C)$,
$\sigma[F] = {}^{\sigma}F_{\infty}\otimes \sigma(F_f)$ (the notation having the obvious meaning).
\end{lem}

\subsection{Whittaker coefficients}
Let $f=\otimes'_w f_w$ be a $K'$-finite, decomposable section in $\textrm{Ind}^{G'(\A_\K)}_{B'(\A_\K)}[\tilde\tau]$. We let $E(f)=E(f,0)$ be the attached Eisenstein series, evaluated at $0$, where it is always holomorphic. We will consider its global $\psi$-Whittaker functional ($\psi$-Fourier coefficient):
$$E_\psi(f)(g):=\int_{U'(\K)\backslash U'(\A_\K)} E(f)(u'g)\psi^{-1}(u')du'\in W(\Pi').$$
To each ``function'' $\varphi\in\tilde\tau$ (which is of course simply a constant), we can associate a $K'$-finite, decomposable section $f_\varphi=\otimes'_w f_{\varphi,w} \in\textrm{Ind}^{G'(\A_\K)}_{B'(\A_\K)}[\tilde\tau]$ as in \cite{shahidi-book} 6.3. Recall the choice of an element $w_0\in G'(\K)$, representing the longest element in the Weyl group of $G'/\K$. Then we obtain the following result

\begin{prop}\label{WhitEis1}
The value at the identity $id\in G'(\A_\K)$ of the $\psi$-Whittaker functional $E_\psi(f_\varphi)$ equals
$$E_\psi(f_\varphi)(id)=\prod_{w\in S(\Pi')} W_w(id_w) \cdot \prod_{\substack{w\notin S(\Pi')\\ 1\leq i < j\leq n-1}}L(1,BC(\chi_i)_wBC(\chi_j^{-1})_w)^{-1},$$
where
$$W_w(id_w)=\int_{U'(\K_w)}f_{\varphi,w}(w_0^{-1}n'_w)\psi_w^{-1}(n'_w)dn'_w.$$
In particular, $E_\psi(f_\varphi)(id)$ has an Euler factorization.
\end{prop}
\begin{proof}
This is \cite{shahidi-book} Thm.\ 7.1.2. One simply observes that Shahidi's functional $\lambda^0_w$
(cf.\ \cite{shahidi-book}, p.\ 122) gives the complex number $f_{\varphi,w}(w_0^{-1}n'_w)$, viewed
as an element in $\tilde\tau$.
\end{proof}

We recall the $\Q$-rational (and hence by our assumption also $\Q(E_\lambda)$-rational) basis vectors $X'^*_{\underline j}$, resp.\ the $\Q(E_\lambda)$-basis $\{ e'_\beta\}$ of $E_\lambda$ from Sect.\ \ref{sect:Theta}. Moreover, we assume from now on that $\xi_{\Pi'_\infty,\underline j, \beta}$ was chosen in a way such that the generator $[\Pi'_\infty]$ is $\Q(E_\lambda)$-rational in the sense of Sect.\ \ref{rationalgk}.

We now choose a section $f_\varphi$ such that $f_{\varphi,\infty}$ is one of the vectors $W^{-1}(\xi_{\Pi'_\infty,\underline j, \beta})$ appearing in the definition of $[\Pi'_\infty]$ (or rather its transfer by $W^{-1}$) and that the non-archimedean part of $f_{\varphi}$ lies in the given $\Q(\Pi'_f)$-structure of $\Pi'_f$. Then, the evaluated Eisenstein series $E(f_\varphi)$ appears as a factor in the tensor product $\sum_{\underline j,\beta}\left(X'^*_{\underline j}\otimes E_{\Pi',\underline{j},\beta}\otimes e'_\beta\right)$, which defines a $\Q(\Pi'_f)$-rational Eisenstein cohomology class.

\begin{cor}\label{WhitEis2}  There is a constant $\Omega(\Pi'_{\infty}) \in \C^{\times}$, depending only on $\Pi'_{\infty}$ and
on the choice of basis vectors $\xi_{\Pi'_\infty,\underline{j},\beta}$, such that
$$p(\Pi') \sim_{\Q(\Pi'_f)} \Omega(\Pi'_{\infty})\cdot \prod_{1\leq i < j\leq n-1}L(1,BC(\chi_i\cdot\chi_j^{-1})_f).$$
\end{cor}
\begin{proof}  It follows from Corollary \ref{cor:rational_Eisenstein}, \ref{constterm}, and the definitions that the period $p(\Pi')$ can be represented
by the ratio between the value at the identity of the $\tilde{\tau}$-component of the constant term and the value
at the identity of the $\psi$-Whittaker functional $E_\psi(f_\varphi)$, if the latter doesn't vanish.
We have chosen $f_\varphi$ so that $E(f_\varphi)$ defines a $\Q(\Pi'_f)$-rational class, which means that $p(\Pi')$ is the inverse of
of $E_\psi(f_\varphi)(id)$, which by Proposition \ref{WhitEis1} equals
$$W_{\infty}(id_\infty)\cdot \prod_{\substack{w \in S(\Pi') \setminus \{\infty\}\\1\leq i < j\leq n-1}} W_w(id_w)\cdot L(1,BC(\chi_i)_wBC(\chi_j^{-1})_w) \cdot
\prod_{1\leq i < j\leq n-1}L(1,BC(\chi_i)_fBC(\chi_j^{-1})_f)^{-1}.$$

We take $\Omega(\Pi'_{\infty}) = W_{\infty}(id_\infty)^{-1}$.  Since the local $L$-factors belong to $\Q(\Pi'_f)$ and transform under the Galois group
along with $\Pi'$, we will finish the proof if we can show that $f_{\varphi,w}$ can be chosen so that the integral
defining $W_w(id_w)$ is a rational constant.  Now, as in Lemma \ref{constterm}, the $\Q(\Pi'_w)$-rational structure on the principal series representation $\Pi'_w$ is given by
that on the boundary cohomology, and coincides with the one defined (adelically) in \ref{rem:induced}, namely the
space of functions in the principal series whose restrictions to a
fixed maximal compact subgroup $K'_w$ take values in $\Q(\Pi'_w)$.  We take $U_1 \subset U'(\K_w)$ to be a subgroup of the
kernel of $\psi_w$.  Possibly shrinking $U_1$ further, we can assume that $f_{\varphi,w}(w_0n'_w)$ is supported in $U_1$
as a function of $n'_w$ and takes value $1$ there.  Then the integral is just a volume factor and belongs to $\Q$.
\end{proof}

%\vfill
%\pagebreak

\section{Period relations - The main results}

%THIS SECTION IS UNDER CONSTRUCTION

In this section we always assume $\Pi$ and $\Pi'$ to be obtained by base change
from unitary groups of all signatures; in other words, that they both satisfy Hypothesis \ref{descent}.

\subsection{Critical values of tensor products when one of the representations is of abelian type}
Let $\Pi'$ be the abelian automorphic representation of \ref{sect:Pi'}, attached to
the $(n-1)$-tuple $\mathbf{\chi} = (\chi_1,\dots,\chi_{n-1})$ and the auxiliary character
$\gamma$;  we write $\Pi' = \Pi'(\mathbf{\chi},\gamma)$.
Recall that  $\chi_{j,\infty}(e^{i\theta}) = e^{i k_j\theta}$ and $\gamma_{\infty}(z) = (z/\bar{z})^t$ for $k_j \in \ZZ$
and $t \in \tfrac{n-1}{2} + \ZZ$.   We henceforth
assume $t = \frac{1}{2}$, to apply the formulas of \cite{har2007}.  Define $\gamma^+(x) = \gamma(x)\cdot \|x\|^{\frac{1}{2}}$ as in Rem.\ \ref{rem:induced}.
Then $\gamma^+_{\infty}(z) = z$; thus $\gamma^+$ is of type $\eta_{\kappa}$ with $\kappa  = 1$, in the notation
of \cite{harcrelle} and (especially) Theorem 4.3 of \cite{har2007}.
As in \cite{haradj}, we introduce period invariants
$$G(j,\mathbf{\chi}) = [(2\pi i)^{k_j}\cdot p(\check{\chi}_j^{(2)},1)]^{2j-n} \quad\textrm{and}\quad G(j,\mathbf{\chi},\gamma) = [(2\pi i)^{k_j}\cdot p(\check{\chi}_j^{(2)}\check{\gamma}^+,1)]^{2j-n}.$$

Then, if $n$ is odd:
\begin{equation}\label{fact}  \begin{aligned} L(s,R(M(\Pi)\otimes M(\Pi'))) &= \prod_{j = 1}^{n-1} L(s+\tfrac{3-2n}{2},\Pi\otimes BC(\chi_j)\cdot \gamma) \\
&=  \prod_{j = 1}^{n-1} L(s - \tfrac{n-1}{2},M(\Pi)\otimes BC(\chi_j)\cdot \gamma^+).
\end{aligned}
\end{equation}

In what follows, $\alpha = \gamma^+$ and so $\alpha_0 = \varepsilon_{\K}^{n}$ in all cases.  Thus the term
$[(2\pi i)^{\kappa}\mathcal{G}(\alpha_{0,f})]^{r(M';\Pi)}$ in expression \ref{crelle} is just $ \sim (2 \pi i)^{\kappa r(M';\Pi)}\cdot \mathcal{G}(\varepsilon_{\K,f})^{nr(M';\Pi)}$.
Let $s_0 = n-1 + m$, where $m$ is as in \ref{critpoints} above.  We apply \ref{crelle} to the critical values of $L(s,R(M(\Pi)\otimes M(\Pi'))$.  The results are as
follows.  If $n$ is even, $\alpha = \gamma^+$ is the trivial character and we have

\begin{equation}\label{abelianeven} \begin{aligned}
&L(n-1+m,R(M(\Pi)\otimes M(\Pi'))) =  \prod_{j = 1}^{n-1} L( \tfrac{n}{2} + m,M(\Pi)\otimes BC(\chi_j)) \\
&\sim_{\K E(\Pi) E(\Pi')} \prod_{j = 1}^{n-1} [(2 \pi i)^{( \tfrac{n}{2} + m)n - \frac{n(n-1)}{2}}\cdot \mathcal{G}(\varepsilon_{\K,f})^{\lceil\frac{n}{2}\rceil}\cdot P^{(r(\chi_j;\Pi))}(\Pi) G(\Pi_j^{\prime,c})^{n-2r(\chi_j;\Pi)} \\
%& \sim_{\K E(\Pi) E(\Pi')} \prod_{j = 1}^{n-1} (2\pi i)^{(\frac{n}{2} + m)n - \frac{n(n-1)}{2}} G(j,\mathbf{\chi}) P^{(n-j)}(\Pi) \\
% & \sim_{\K E(\Pi) E(\Pi')} (2\pi i)^{\frac{n(n-1)}{2} + nm(n-1)}\prod_{j = 1}^{n-1}  G(j,\mathbf{\chi}) P^{(n-j)}(\Pi)
&\sim_{\K E(\Pi) E(\Pi')}  (2 \pi i)^{(n-1)(mn) + \frac{n(n-1)}{2}}
\mathcal{G}(\varepsilon_{\K,f})^{\lceil\frac{n}{2}\rceil(n-1)}\prod_{j = 1}^{n-1} P^{(r(\chi_j;\Pi))}(\Pi) G(\Pi_j^{\prime,c})^{n-2r(\chi_j;\Pi)}
\end{aligned}
\end{equation}
%$$
with $\Pi_j^{\prime,c} = \chi_j^{-1}$.

%**FROM HERE NOT CORRECTED**
If $n$ is odd, then $\alpha = \gamma^+$ and we have %in the formula of \cite{har2007})
\begin{equation}\label{abelianodd} \begin{aligned}
L(n-1+m,R(M(\Pi)\otimes M(\Pi'))) &= \prod_{j = 1}^{n-1} L( \tfrac{n-1}{2} + m,M(\Pi)\otimes BC(\chi_j)\cdot \gamma^+)  \\
&\sim_{\K E(\Pi) E(\Pi')} \prod_{j=1}^{n-1} (2\pi i)^{(\frac{n-1}{2} + m)n - \frac{n(n-1)}{2}} (2\pi i)^{n-j}G(j,\mathbf{\chi},\gamma) P^{(n-j)}(\Pi) \\
&\sim_{\K E(\Pi) E(\Pi')}  (2\pi i)^{(n-1)(mn)+\frac{n(n-1)}{2}} \prod_{j = 1}^{n-1}  G(j,\mathbf{\chi}) P^{(n-j)}(\Pi)
\end{aligned}
\end{equation}
Here $\Pi_j^{\prime,c} = \chi_j^{-1}\cdot(\gamma^+)^c$.  The contribution of $\gamma^+$ disappears
because the product of the periods of $\gamma^+$ to the $2j-n$ is just $1$.  In particular the
expressions \ref{abelianeven} and \ref{abelianodd} are identical.

\subsection{Whittaker periods and Petersson norms}

In this section we combine the results of the previous sections  to express
the invariants $p(\Pi)$ in terms of the $P^{(s)}(\Pi)$.  Notation is as above; in particular, $\Pi$ is
cohomological and obtained by base change from unitary groups and
$\Pi' = \Pi'(\mathbf{\chi},\gamma)$.

\begin{prop}  Suppose $\Pi$ satisfies Hypothesis \ref{descent}. Let $s_0=n-1+m$, $m\geq 0$, be a critical value of $L(s,R(M(\Pi)\otimes M(\Pi')))$ satisfying $L(n-1+m,R(M(\Pi)\otimes M(\Pi'))) \neq 0$. Then there are constants $p(m,\Pi_{\infty},\Pi'_{\infty}) \in \C^\times$ and $\Omega(\Pi'_{\infty}) \in \C^{\times}$ such that
\begin{equation}\label{comparison2}
\begin{aligned}
 p(m,\Pi_{\infty},\Pi'_{\infty})\, \Omega(\Pi'_{\infty}) \, p(\Pi) & \prod_{1 \leq i < j \leq n-1} L(1,BC(\chi_i\cdot \chi_j^{-1})_f)   \sim_{\K E(\Pi) E(\Pi')}  \\
& (2 \pi i)^{(n-1)(mn)+\frac{n(n-1)}{2}}
\mathcal{G}(\varepsilon_{\K,f})^{\lceil\frac{n}{2}\rceil(n-1)}\prod_{j = 1}^{n-1} P^{(r(\chi_j;\Pi))}(\Pi) G(\Pi^{\prime,c})^{n-2r(\chi_j;\Pi)}.
\end{aligned}
\end{equation}
A critical point $s_0=n-1+m$, $m\geq 0$, giving $L(n-1+m,R(M(\Pi)\otimes M(\Pi'))) \neq 0$ exists in particular, if the inequalities separating the archimedean
parameters of $\Pi$ and $\Pi'$ are strict:
$$\mu_1> -\lambda_{n-1} > \mu_2 > -\lambda_{n-2} >... >  -\lambda_1 > \mu_n.$$
\end{prop}

\begin{proof}  This is obtained by comparing  the expressions of \ref{thm:whittaker-periods}, \ref{WhitEis2}, \ref{abelianeven}, and
\ref{abelianodd} for the non-vanishing value of the $L$-function, and bearing in mind
that, by \ref{nogauss}, the term $\mathcal{G}(\omega_{\Pi'_{f,0}})$ in \ref{thm:whittaker-periods}
is trivial.  The last assertion follows from the description of \ref{critpoints}:  It implies there
is a critical value $s_0$ strictly to the right of the center of symmetry. Since $\Pi$ is cuspidal, it is
a theorem of Shahidi and Jacquet-Shalika that the value $L(s_0,R(M(\Pi)\otimes M(\Pi'))) \neq 0$.
\end{proof}

Now to continue, note that the character $\chi_j$ has infinity type $(z/\bar{z})^{k_j}$ which in the conventions of 4.2 means
that the parameter is $-k_j$.   Then the set $T(M,M'(\chi_j))$ is the set of $t$ such that
$2k_j \geq n- 2p_t^c$.  On the other hand, we are given that $2k_j \in [n-2q_{j+1},2p_j - n]$, i.e.
$j$ is the largest integer such that $2k_j \geq n-2q_{j+1} = n - 2p^c_{n-j}$.
 Thus $r(M'(\chi_j);M) = n-j$.  The factor multiplying $P^{(r)}(\Pi)$ will then be $G((\chi_j)^c)^{2j-n}$.

On the other hand, if $M'_{ij} = M(\chi_{ij}) = M(\chi_i\chi_j^{-1})$ with $i < j$,
the parameter of $M'_{ij}$ is $2(k_j - k_i) < 0$.   So $r(M'_{ij}) = 0$ and we get
$$L(1,M'_{ij}) \sim (2\pi i)G(\chi^c_i(\chi^c_j)^{-1}) = (2\pi i) G(\chi^c_i)G(\chi^c_j).$$
Then
$$\prod_{i < j} L(1,M'_{ij}) \sim (2\pi i)^{\frac{n(n-1)}{2}} G(\chi^c_j)^{(j-1) - (n-j)} =  (2\pi i)^{\frac{n(n-1)}{2}} G(\chi^c_j)^{2j-n}.$$

%What a relief!
Putting this together we find that, assuming the hypotheses of
the preceding proposition are verified, we have

%**CHECK POWERS OF PI
%** Also $P^{j}(\Pi)$ vs. $P^{j}(\Pi)$.

\begin{equation*}
\begin{aligned}
 p(m,\Pi_{\infty},\Pi'_{\infty})\, \Omega(\Pi'_{\infty}) \, p(\Pi) & (2\pi i)^{\frac{n(n-1)^2}{2}} \prod_{j = 1}^{n-1} G(\chi^c_j)^{2j-n}   \sim_{\K E(\Pi) E(\Pi')}  \\
& (2 \pi i)^{(n-1)(mn)+\frac{n(n-1)}{2}}
\mathcal{G}(\varepsilon_{\K,f})^{\lceil\frac{n}{2}\rceil(n-1)}\prod_{j = 1}^{n-1} P^{(n-j)}(\Pi)  \prod_{j = 1}^{n-1} G(\chi^c_j)^{2j-n}
\end{aligned}
\end{equation*}
or more simply
\begin{equation}\label{pW}
p(m,\Pi_{\infty},\Pi'_{\infty})\, \Omega(\Pi'_{\infty}) \, p(\Pi)  \sim_{\K E(\Pi) E(\Pi')}(2 \pi i)^{n(n-1)(m-n+2)} \prod_{j = 1}^{n-1} P^{(j)}(\Pi).
\end{equation}

Now, we can state our first main theorem:

\begin{thm}\label{heckethm}
Let $\mu$ be the parameter of $\Pi$, and suppose $\mu_i - \mu_{i+1} \geq 2$ for all $i$, so that there is a parameter $\lambda$ for $G'$ and an integer $m > 0$ such that $s_0 = n-1 +m$ is a critical value of the tensor product $L$-function $L(s,R(M(\Pi) \otimes M(\Pi')))$  (in the motivic normalization) for any $\Pi'$ with parameter $\lambda$.  Suppose $\Pi$ satisfies Hypothesis \ref{descent}. Then there is a non-zero constant $Z(\Pi_{\infty})$ which depends only on the local representation $\Pi_{\infty}$ and such that
$$p(\Pi)  \sim_{\K E(\Pi)}
Z(\Pi_{\infty})\prod_{j = 1}^{n-1} P^{(j)}(\Pi).$$
\end{thm}
\begin{proof}
The hypothesis on $\mu$ implies that there exist $\Pi'$ of abelian type such that the comparison of \ref{pW} is valid.  By letting $\Pi'$ vary among abelian type representations with parameter $\lambda$, we can remove the $E(\Pi')$ from the equivalence relation.   Initially one obtains the relation
\begin{equation}\label{withm} p(\Pi)  \sim_{\K E(\Pi)}
Z(m,\Pi_{\infty},\Pi'_{\infty})\prod_{j = 1}^{n-1} P^{(j)}(\Pi), \end{equation}
where $Z(m,\Pi_{\infty},\Pi'_{\infty}) = [\Omega(\Pi'_{\infty})p(m,\Pi_{\infty},\Pi'_{\infty})]^{-1}\cdot (2 \pi i)^{n(n-1)(m-n+2)}$.  But
since the left-hand side  of \ref{withm} is independent of $m$ and $\Pi'_{\infty}$, so is the right-hand side.
\end{proof}

\begin{rem}\label{substitutes}  Two substitutes are possible for Hypothesis \ref{descent} when it is not satisfied.  As noted above,
there is only an obstruction if $n$ is even.   Under the regularity hypothesis on the parameter
$\mu$, it can be shown that the standard $L$-function of $\Pi$ can be realized, up to an abelian
twist, as the $L$-function of a holomorphic automorphic representation of $U(n+1-j,j)$ for any $j$; this was done in most cases using the
theta correspondence in \cite{harHowe}.   The Petersson norm of an arithmetically normalized form in this representation
can then be used in place of the missing $P^{(j)}(\Pi)$.

Even in the absence of the regularity hypothesis, one can use quadratic base change, as in \cite{yoshida1}, to define
versions of the missing $P^{(j)}(\Pi)$.  These are only well-defined up to square roots of elements in the coefficient
field, so are less precise than the ones defined motivically.
\end{rem}

\subsection{General tensor products}

Combining the comparison in Thm.\ \ref{heckethm} with Thm.\ \ref{thm:whittaker-periods-general}, we obtain the following general
result, when $\Pi_{\infty}$ and $\Pi'_{\infty}$ are both sufficiently regular.

\begin{thm}\label{tensorproduct}  Let $\Pi=BC(\pi)$ and $\Pi'=BC(\pi')$ be cuspidal automorphic representations of $G(\A_{\K})$ and
$G'(\A_{\K})$ which are cohomological with respect to $E_{\mu}$ and $E_{\lambda}$, respectively.  Suppose
\begin{enumerate}
\item  $\mu_i - \mu_{i+1} \geq 2$ for all $i$ and $\lambda_j - \lambda_{j+1} \geq 2$ for all $j$.
\item Both $\Pi$ and $\Pi'$ satisfy Hypothesis \ref{descent}.
\end{enumerate}
Then for every critical point $s_0 = \tfrac12 + m$ of $L(s,\Pi\times\Pi')$ with $m \geq 0$,
$$L(\tfrac12+m,\Pi_f \times \Pi'_f)  \sim_{\K E(\Pi) E(\Pi')} p(m,\Pi_\infty,\Pi'_\infty) Z(\Pi_{\infty})Z(\Pi'_{\infty})\prod_{j = 1}^{n-1} P^{(j)}(\Pi)
\prod_{k = 1}^{n-2} P^{(k)}(\Pi').$$
Equivalently, for every critical point $s_0 = n-1 + m$ of $L(s,R(M(\Pi)\otimes M(\Pi')))$ with $m \geq 0$,
$$
L(n-1+m,R(M(\Pi) \otimes M(\Pi'))) \sim_{\K E(\Pi) E(\Pi')} p(m,\Pi_\infty,\Pi'_\infty) Z(\Pi_{\infty})Z(\Pi'_{\infty})\prod_{j = 1}^{n-1} P^{(j)}(\Pi)\prod_{k = 1}^{n-2}  P^{(k)}(\Pi'). $$
\end{thm}

\subsection{The archimedean constant}  The constants $\Omega(\Pi'_{\infty})$,  $p(m,\Pi_{\infty},\Pi'_{\infty})$, and $p(\Pi)$
all depend on choices of test vectors in the induced and Whittaker models of $\Pi_{\infty}$ and $\Pi'_{\infty}$.
We have decided to choose vectors rational over $\K$, but we could also choose test vectors so that
$W_{\infty}(\Pi') = 1$.  Since this choice depends only on $\Pi'_{\infty}$, it makes sense for cuspidal as
well as Eisenstein representations.  This provides a unique normalization of $p(\Pi)$ and thus a
unique value for $p(m,\Pi_{\infty},\Pi'_{\infty})$ (as does the choice of rational vectors).

In \cite{linjie}, Lin Jie proves a period relation for $\Pi$ and $\Pi'$ obtained by automorphic induction from Hecke
characters.  Using this relation, she deduces the following refinement of Theorem \ref{tensorproduct}.

\begin{thm}[Lin Jie]\label{lin}  Under the hypotheses of Theorem \ref{tensorproduct}, we have the following relations:
$$L(\tfrac12+m,\Pi_f \times \Pi'_f)  \sim_{\K E(\Pi) E(\Pi')} (2\pi i)^{(m + \frac{1}{2})n(n-1)}\prod_{j = 1}^{n-1} P^{(j)}(\Pi)
\prod_{k = 1}^{n-2} P^{(k)}(\Pi').$$
or equivalently,
$$
L(n-1+m,R(M(\Pi) \otimes M(\Pi'))) \sim_{\K E(\Pi) E(\Pi')} (2\pi i)^{(m + \frac{1}{2})n(n-1)}\prod_{j = 1}^{n-1} P^{(j)}(\Pi)\prod_{k = 1}^{n-2}  P^{(k)}(\Pi'), $$
provided $m > 0$.
\end{thm}
If $m = 0$ her result remains true under a certain global non-vanishing hypothesis.
In view of formula \ref{cplus2}, this is exactly what is predicted by Deligne's conjecture,
assuming the validity of Optimistic Comparison \ref{comparison}.

%\begin{thm}\label{heckethm}
%\end{thm}

\subsection{Whittaker periods and adjoint $L$-values}\label{wei}

The results of this section, in contrast to Thm.\ \ref{heckethm}, are unconditional, but for the moment they
only apply under slightly restrictive hypotheses.  This is because they are based on Wei Zhang's version
\cite{weizhang} of the Ichino-Ikeda-Neal Harris conjecture for automorphic forms on unitary groups, and for the moment neither the
local nor global properties of the relative trace formula are known in sufficient generality to allow a complete
comparison.  The simplifying hypotheses in \cite{weizhang} are analogous to those used in earlier
applications of the Arthur-Selberg trace formula, and the history of the latter gives us reason to be optimistic
that the restrictions will soon be unnecessary.

For ease of reference, the simplifying hypotheses are listed separately.

\subsubsection{Simplifying hypotheses}
The unitary groups $H$ and $H'$ are as in the earlier sections. We let $\pi$ and $\pi'$ be unitary cuspidal automorphic representations of $H(\A)$, resp.\ $H'(\A)$, which occur with multiplicity one in the discrete spectrum of $H$, resp.\ $H'$. Say a prime $p$ is {\it split} if it splits in the quadratic extension $\mathcal{K}$. Since we want to use Wei Zhang's results, we have to assume

\begin{hyp}\label{hypotheses}
\begin{itemize}
\item{(a)} There exists some split prime $p$ such that $\pi_p$ and $\pi'_p$ are both supercuspidal.
\item{(b)}  Every prime $p < M$ is split, where $M$ is the ``algorithmically computable'' constant that
arises in \cite{jgordon} in the transfer to characteristic zero of the Jacquet-Rallis fundamental lemma
(see below).
\item{(c)} If $p$ is not split, then $\pi_p$ and $\pi'_p$ are both unramified.
%\item{(d)}
\end{itemize}
\end{hyp}

\begin{rem}
(a) Let $\Pi=BC(\pi)$ and $\Pi'=BC(\pi')$ denote the base change representations of $\pi$ and $\pi'$, to $G(\A_\mathcal{K})$
and $G'(\A_{\mathcal{K}})$, respectively.  Hypothesis (a) above implies that $\Pi$ and $\Pi'$ are both
cuspidal automorphic representations. Moreover, by Thm.\ \ref{thm:base_change}, $\Pi$ and $\Pi'$ are both of cohomological type. Hence, the Asai $L$-functions $L(s,\Pi,As^{(-1)^{n-1}})$ and $L(s,\Pi',As^{(-1)^{n-2}})$ have a pole at $s=1$, by \cite{harlab} or \cite{mok} Cor.\ 2.5.9 and (2.5.12). As a consequence, Wei Zhang's additional set of assumptions {\bf RH(I)} (cf.\ \cite{weizhang}, p.\ 544) is satisfied in our case. Moreover, Zhang's second set of assumptions {\bf RH(II)} (cf.\ \cite{weizhang}, pp.\ 544--545) is fulfilled by the work of Beuzart-Plessis, \cite{plessis}. (We should, however, mention that the latter paper depends on the (expected) properties of local $L$-packets for unitary groups. See \cite{plessis} Sect.\ 18.1 for the precise expectations.)

(b) In \cite{weizhang} it is also assumed that $\pi$ and $\pi'$ are locally tempered at all places. Due to Caraiani \cite{car}, Thm.\ 1.2, $\Pi=BC(\pi)$ and $\Pi'=BC(\pi')$ are tempered at all places (under hypothesis (a) it was also proved in the book of Harris and Taylor). We expect that by the forthcoming work of Kaletha, Minguez, Shin and White, it should automatically follow that hence also $\pi$ and $\pi'$ are tempered at all places, whence we did not assume this explicitly. However, the extremely careful reader may restrict himself/herself in what follows to the case where $\pi$ and $\pi'$ are locally tempered at all primes $p$.

(c) Zhiwei Yun proves the Jacquet-Rallis fundamental lemma in \cite{yun} for local fields
of characteristic $> n$.  In her appendix to Yun's paper, Julia Gordon derives the Jacquet-Rallis fundamental
lemma for a local field $K$ of characteristic $0$, but at present her methods (from motivic integration) require her
to assume that the residue characteristic of $K$ is greater than an unspecified positive constant $M$.  This
is the $M$ that appears in (b) above; it is effectively computable but no one has carried out the computation.
For split primes the Jacquet-Rallis fundamental lemma is vacuous.  The requirement that a prime $p$ be split in the results
of \cite{weizhang} can be removed whenever the Jacquet-Rallis fundamental lemma is known in residue characteristic $p$.
\end{rem}

Since $H(\R)$ and $H'(\R)$ are both assumed compact we are in case (2) of \cite{weizhang}, Theorem 1.2.  We define
$$\mathcal{L}(\pi,\pi'): = \Delta_n \cdot \frac{L(\frac{1}{2},\Pi\times \Pi')}{L(1,\pi,Ad)L(1,\pi',Ad)},$$
where $\Delta_n$ is the $L$-function of the {\it Gross motive}:
$$\Delta_n = \prod_{i = 1}^n L(i,\varepsilon_{\K,f}^i).$$

For $\phi \in \pi$ and $\phi' \in \pi'$, we let
$$P(\phi,\phi') := \int_{H'(\Q)\backslash H'(\A)} \phi(h')\phi'(h') dh';$$
$$\<\phi,\phi\> :=  \int_{H(\Q)\backslash H(\A)} \phi(h)\bar{\phi}(h) dh.$$
We use Tamagawa measures for $dh$ and $dh'$. Here is Wei Zhang's theorem.

\begin{thm}(\cite{weizhang})  Assume $\pi$, $\pi'$, and $\K$ satisfy hypothesis \ref{hypotheses}.   Then there is a
non-zero constant $c(\pi_{\infty},\pi'_{\infty})$, depending only on the archimedean components
of $\pi$ and $\pi'$, such that, for every factorizable $\phi \in \pi$ and $\phi' \in \pi'$,
$$2^{-2}c(\pi_{\infty},\pi'_{\infty})Z_{loc}(\phi,\phi') \mathcal{L}(\pi,\pi') = \frac{|P(\phi,\phi')|^2}{\<\phi,\phi\>\<\phi',\phi'\>}.$$
Here $Z_{loc} = \prod_{v \in S} Z_v$ is a product of normalized local integrals of matrix coefficients;
$Z_v$ is denoted
$$\frac{\alpha_v^{\prime\natural}(\Phi_v,\Phi_v)}{(\Phi_v,\Phi_v)_v}$$
in \cite{weizhang}, with $\Phi_v = \phi_v\otimes \phi'_v$ (Zhang uses $\phi_v$ for this tensor product).
\end{thm}

\begin{rem}  The theorem in \cite{weizhang} is stated under a more general version of hypothesis (c).  With a bit more work we
could draw conclusions below in this more general situation; however, hypothesis (c) is destined to disappear in the short term,
so this seems unnecessary.
\end{rem}
%We rewrite the main formula
%\begin{equation}
%c(\pi_{\infty},\pi'_{\infty})Z_{loc}(\phi,\phi') L(\frac{1}{2},\Pi\times \Pi') = 4\frac{|P(\phi,\phi')|^2}{<\phi,\phi><\phi',\phi'>}(\Delta_n)^{-1} \cdot L(1,\pi,Ad)L(1,\pi',Ad).
%\end{equation}

As in (4.1.2)  of \cite{haradj}, we have

\begin{equation}
\Delta_n \sim_{\Q(\varepsilon_{\K})} (2 \pi i)^{\frac{n(n+1)}{2}}\cdot \mathcal{G}(\varepsilon_{\K,f})^{\lceil\frac{n+1}{2}\rceil}.
\end{equation}

In {\it loc.\ cit.\ } $n$ is assumed even, but the same argument gives the above result.  Since
$H$ and $H'$ are both definite unitary groups, the archimedean local integral $Z_{\infty}$ is an algebraic number, rational
over the field of definition of (the finite-dimensional representation) $\pi_{\infty} \otimes \pi'_{\infty}$.  The
discussion leading to \S 4.1.6 of \cite{haradj} thus yields:

\begin{cor}  Assume $\pi$, $\pi'$, and $\K$  satisfy hypothesis \ref{hypotheses}.  Then there exists a non-zero constant $c_{\pi_\infty,\pi'_\infty}$ such that
\begin{itemize}
\item
\begin{equation}\label{IINH} \Lambda(\pi,\pi') := c_{\pi_{\infty},\pi'_{\infty}}\mathcal{G}(\varepsilon_{\K,f})^{\lceil\frac{n+1}{2}\rceil}(2\pi i)^{\frac{n(n+1)}{2}}\cdot \frac{L(\frac{1}{2},\Pi_f\times \Pi'_f)}{L(1,\pi_f,Ad)L(1,\pi'_f,Ad)} \in \bar{\Q}
\end{equation}
\item For all $\sigma \in {\rm Aut}(\C)$,
\begin{equation} \label{IINH2}
\sigma(\Lambda(\pi,\pi')) = \Lambda({}^\sigma\!\pi,{}^\sigma\!\pi') \end{equation}
%$$ \begin{matrix} \sigma[c(\pi_{\infty},\pi'_{\infty})\mathcal{G}(\varepsilon_{\K,f})^{\frac{n+1}{2}}(2\pi i)^{\frac{n(n+1)}{2}}\frac{L(\frac{1}{2},BC(\pi)\times BC(\pi'))}{L(1,\pi,Ad)L(1,\pi',Ad)}] \\
%= c(\sigma(\pi_{\infty}),\sigma(\pi'_{\infty}))\mathcal{G}(\varepsilon_{\K,f})^{\frac{n+1}{2}}(2\pi i)^{\frac{n(n+1)}{2}}\frac{L(\frac{1}{2},BC(\sigma(\pi))\times BC(\sigma(\pi')))}{L(1,\sigma(\pi),Ad)L(1,\sigma(\pi'),Ad)}.
%\end{matrix}$$
\end{itemize}

\end{cor}

The only difference with (4.1.6.2) of \cite{haradj} is the inclusion of the factor $c_{\pi_{\infty},\pi'_{\infty}}$ -- since the Ichino-Ikeda-Neal Harris conjecture is
only known up to a factor $c(\pi_{\infty},\pi'_{\infty})$ -- and of the power of the Gau\ss {} sum is not omitted.

Theorem \ref{thm:whittaker-periods-general} gives a different expression for the numerator of (\ref{IINH}).  Before we compare the two expressions, we need the
following non-vanishing result, which is essentially due to Wei Zhang.

\begin{prop}\label{nozero}  Let $\pi$ be an automorphic representation of $H(\A)$ satisfying hypothesis \ref{hypotheses}. Then there exists an automorphic representation $\pi'$ of $H'(\A)$ satisfying hypothesis \ref{hypotheses} and $\phi \in \pi$, $\phi' \in \pi'$, such that $P(\phi,\phi') \neq 0$ and $L(\frac{1}{2},BC(\pi)\times BC(\pi')) \neq 0$.
\end{prop}

\begin{proof}  This is essentially Lemma 2.15 of \cite{weizhang1}.   Let $p$ be a split place such that $\pi_p$ is supercuspidal.  Since the center of
$H$ is anisotropic, the central character of $\pi_p$ is unitary.  Let
$\mu_p$ be a supercuspidal representation of $H'(\Q_p)$ with unitary central character.  Since $H(\Q_p) = GL_n(\Q_p)$ and $H'(\Q_p)$ is
$GL_{n-1}(\Q_p)$, it is known that Hom$_{H'(\Q_p)}(\Pi_p\otimes \mu_p),\C) \neq 0)$.  Since both $\pi_p$ and $\mu_p$ are tempered,
it follows from a result of Sakellaridis and Venkatesh (quoted as Theorem A.1 in \cite{iz}) that the local integral
$$\int_{H'(\Q_p)} f(h')f'(h') dh' \neq 0$$
for some matrix coefficients $f, f'$ of $\pi_p$ and $\mu_p$, respectively.  It follows that (the dual of) $\mu_p$ is weakly contained in
the restriction to $H'(\Q_p)$ of $\pi_p$.

Now we assume $\mu_p$ is induced from an irreducible representation of $GL_{n-1}(\Z_p)\cdot Z$,
where $Z$ is the center of $H'(\Q_p)$; by the theory of types, there are supercuspidal representations with this property.  Choose an irreducible component
$\mu_{\infty}$ of the restriction to $H'(\R)$ of (the finite-dimensional representation) $\pi_{\infty}$. Then we can apply
Lemma 2.15 of \cite{weizhang1} to obtain an automorphic representation $\pi'$ of $H'$, whose archimedean component is
$\mu_{\infty}$ and whose $p$-component is supercuspidal (a twist of $\mu_p$), for which there exist $\phi \in \pi$ and
$\phi' \in \pi'$ such that $P(\phi,\phi') \neq 0$.  This implies the non-vanishing of the $L$-value by the main result of \cite{weizhang},
provided we know that $\pi'$ is unramified at all non-split places.
\end{proof}

\begin{lem}\label{neq1}
Let $\pi$ be an irreducible representation of $H(\A)$ when $n = 1$; in other words, $\pi$ is a character of $U(1)(\A)$. Then $L(s,\pi,Ad) = L(s,\varepsilon_{\K})$. In particular, $L(1,\pi_f,Ad) \sim_{\Q(\varepsilon_\K)} (2\pi i) \mathcal{G}(\varepsilon_{\K,f})$.
\end{lem}
\begin{proof}
This is an obvious calculation: The Lie algebra of $H$ is one-dimensional and the adjoint
representation of ${}^LG$ is trivial on the Langlands dual group $\C^{\times}$ and is given by
the quadratic character $\varepsilon_\K$.
\end{proof}

\begin{thm}\label{whittakeradjoint}  Let $\pi$ be an irreducible unitary cuspidal automorphic representation of $H(\A)$ which is cohomological with respect to $E^{\sf unt}_\mu$. Assume $\pi$ satisfies hypothesis \ref{hypotheses}. Then there is a non-zero complex constant $a(\pi_{\infty})$, depending only on $\pi_{\infty}$, and an integer $g(n)\in\Z$, such that
$$p(BC(\pi)) \sim_{\Q(\varepsilon_\K) \Q(\pi)} a(\pi_{\infty})\mathcal{G}(\varepsilon_{\K,f})^{g(n)}L(1,\pi_f,Ad).$$
\end{thm}
\begin{proof}
We prove this by induction on $n$.  When $n = 1$ the Whittaker period $p(BC(\pi))$ is rational and the claim follows
from \ref{neq1}.  \\Suppose the theorem is known for $n - 1$.  Choose $\pi$ and let $\pi'$ be an automorphic
representation of $H'(\A)$ satisfying \ref{nozero} relative to $\pi$. As $\pi$ and $\pi'$ both satisfy hypothesis \ref{hypotheses}, $\Pi=BC(\pi)$ and $\Pi'=BC(\pi')$ are cuspidal automorphic representations matching the conditions of Thm.\ \ref{thm:whittaker-periods-general}. In particular, $p(BC(\pi))$ and $p(BC(\pi'))$ exist. By the very construction of $\pi'$, Lem.\ \ref{lem:coeff} and Lem.\ \ref{critpoints}, $s_0=\tfrac12$ is critical for $L(s,BC(\pi)\times BC(\pi'))$, so $p(BC(\pi)_\infty,BC(\pi')_\infty)$ is well-defined. Now, let
$$\Lambda_W(\pi,\pi') := c_{\pi_{\infty},\pi'_{\infty}}\mathcal{G}(\varepsilon_{\K,f})^{\lceil\frac{n+1}{2}\rceil}(2\pi i)^{\frac{n(n+1)}{2}}\cdot \frac{p(BC(\pi))p(BC(\pi'))p(BC(\pi)_\infty,BC(\pi')_\infty)}{L(1,\pi_f,Ad)L(1,\pi'_f,Ad)}.$$

We apply \eqref{IINH} and Thm.\ \ref{thm:whittaker-periods-general} to the pair $(\pi,\pi')$ and conclude that $\Lambda_W(\pi,\pi') \in \bar{\Q}$. By \cite{haradj}, Prop.\ 2.6, ${}^\sigma\Pi$ (resp.\ ${}^\sigma\Pi'$) descends to ${}^\sigma\!\pi$ (resp.\ ${}^\sigma\!\pi'$), which are cuspidal automorphic representations satisfying hypothesis \ref{hypotheses}. Hence, using Prop.\ \ref{nozero}, \eqref{IINH2} and Thm.\ \ref{thm:whittaker-periods-general} again,
\begin{equation}\label{eq:Lambda}
\sigma(\Lambda_W(\pi,\pi')) = \Lambda_W({}^\sigma\!\pi,{}^\sigma\!\pi'),
\end{equation}
for all $\sigma \in {\rm Aut}(\C)$. But by induction, we can rewrite
$$\Lambda_W(\pi,\pi') \sim_{\Q(\varepsilon_\K)\Q(\pi')}  \frac{p(BC(\pi))}{L(1,\pi_f,Ad)}\cdot[a(\pi'_{\infty})c_{\pi_{\infty},\pi'_{\infty}}\mathcal{G}(\varepsilon_{\K,f})^{\lceil\frac{n+1}{2}\rceil+g(n-1)}(2\pi i)^{\frac{n(n+1)}{2}}\cdot p(BC(\pi)_\infty,BC(\pi')_\infty)].$$
It follows from \eqref{eq:Lambda} that
$$p(BC(\pi)) \sim_{\Q(\varepsilon_\K) \Q(\pi) \Q(\pi')} L(1,\pi_f,Ad)\cdot[a(\pi'_{\infty})c_{\pi_{\infty},\pi'_{\infty}}\mathcal{G}(\varepsilon_{\K,f})^{\lceil\frac{n+1}{2}\rceil+g(n-1)}(2\pi i)^{\frac{n(n+1)}{2}}\cdot p(BC(\pi)_\infty,BC(\pi')_\infty)].$$
The term in brackets on the right-hand side depends only on $\pi_{\infty}$ (since $\pi'_{\infty}$ is an irreducible component of the restriction to $H'(\R)$ of $\pi_\infty$) and the degree $n$, whereas the left-hand side depends only on $\pi$, so we conclude by induction.
\end{proof}

\begin{cor}\label{Ladjoint}
Let $\Pi$ and $\Pi'$ be cuspidal automorphic representations of $G(\A_{\K})$ and
$G'(\A_{\K})$ which are cohomological with respect to $E_{\mu}$ and $E_{\lambda}$, respectively.  We assume
$\Pi = BC(\pi)$ and $\Pi' = BC(\pi')$ where $\pi$ and $\pi'$ are irreducible unitary cuspidal automorphic representations of the definite
unitary groups $H(\A)$ and $H'(\A)$, respectively. Suppose moreover that both $\pi$ and $\pi'$ satisfy hypothesis \ref{hypotheses} and that the coefficients $E_\mu$ and $E_\lambda$ satisfy the equivalent conditions of Lem.\ \ref{lem:coeff}.
%\begin{enumerate}
%\item  $\mu_i - \mu_{i+1} \geq 2$ for all $i$ and $\lambda_j - \lambda_{j+1} \geq 2$ for all $j$.
%\item Both $\Pi$ and $\Pi'$ satisfy Hypothesis \ref{descent}.
%\end{enumerate}
Then for every critical point $s_0 = \tfrac12 + m$ of $L(s,\Pi\times\Pi')$ with $m\geq 0$, there is a non-zero complex constant $a(m,\Pi_{\infty},\Pi'_{\infty})$,
depending only on the archimedean components of $\Pi$ and $\Pi'$, and an integer $a(n)$, depending only on $n$, such that,
$$L(\tfrac12 +m,\Pi_f \times \Pi'_f) \ \sim_{\Q(\varepsilon_\K) \Q(\pi)\Q(\pi')} a(m,\Pi_\infty,\Pi'_\infty) \mathcal{G}(\varepsilon_{\K,f})^{a(n)}L(1,\pi_f,Ad)L(1,\pi'_f,Ad).$$
Equivalently, for every critical point $s_0 = n-1 + m$ of $L(s,R(M(\Pi)\otimes M(\Pi')))$ with $m\geq 0$,
$$L(n-1+m,R(M(\Pi) \otimes M(\Pi'_f))) \ \sim_{\Q(\varepsilon_\K) \Q(\pi)\Q(\pi')} a(m,\Pi_\infty,\Pi'_\infty) \mathcal{G}(\varepsilon_{\K,f})^{a(n)}L(1,\pi_f,Ad)L(1,\pi'_f,Ad).$$
\end{cor}
\begin{proof}
Indeed, it follows from the Theorems \ref{whittakeradjoint} and \ref{thm:whittaker-periods-general} and the fact that $\Q(\Pi_f)\subseteq\Q(\pi)$ (resp.\ $\Q(\Pi'_f)\subseteq\Q(\pi')$) that we can define $a(m,\Pi_{\infty},\Pi'_{\infty}): = a(\pi_{\infty})a(\pi'_{\infty})p(m,\Pi_{\infty},\Pi'_{\infty}) \in \C^{\times}$ and $a(n):=g(n)+g(n-1)\in\Z$.
\end{proof}

\subsection{Generalizations}\label{generalizations}

The methods of this paper apply to pairs $\Pi, \Pi'$ of representations when the
critical values of their $L$-functions can be related directly to cup products of
the cohomology classes they define.  This is only possible when the coefficients
satisfy the inequalities of Lemma \ \ref{lem:coeff}, a condition that is equivalent to relations
on the Hodge types of the corresponding motives that are summarized as \ref{tmm}.

We have seen in \ref{cplus2} that the Deligne periods of tensor products of motives satisfying these hypotheses
can be expressed, up to a certain power of $2 \pi i$, as products of terms denoted
$P_{\leq r}(M)$ and  $P_{\leq r'}(M')$, each occuring to the first power (for $1 \leq r \leq n-1$ and $1 \leq r' \leq n-2$).
Deligne periods of tensor products are calculated more generally in \S 1.4 of
 \cite{haradj}.  As explained in \cite{haradj}, the motivic periods that occur in this calculation
 belong to a {\it tableau}, and it follows without difficulty that in all cases these Deligne
 periods are products of certain integral powers of the same $P_{\leq r}(M)$ and  $P_{\leq r'}(M')$; the
 powers depend on the relative positions of the Hodge types of $M$ and $M'$.  In
 Optimistic Comparison \ \ref{comparison} we argue that $P_{\leq r}(M)$ and  $P_{\leq r'}(M')$
 can be identified with certain Petersson (square) norms of normalized holomorphic
 automorphic forms $P^{(i)}(\Pi)$ and $P^{(j)}(\Pi')$ .  Thus we would expect that the critical values of $L(s,\Pi\times \Pi')$ can always
 be expressed, up to algebraic factors, as powers of $2 \pi i$ multiplied by powers of these Petersson inner products.

 In fact, one can always find integers $N > n$ and Hecke characters $\chi_i, i = 1,\dots N-n$;
 $\chi'_j, j = 1,\dots, N - n$, all obtained by base change from $U(1)$, such that the (tempered) Eisenstein representations
 $$\Sigma = \Pi\boxplus \chi_1  \boxplus \chi_2 \boxplus \dots \boxplus \chi_{N-n};\quad\quad \Sigma' = \Pi'\boxplus \chi'_1  \boxplus \chi'_2 \boxplus \dots \boxplus \chi'_{N-n}$$
 are cohomological and have coefficients satisfying  the inequalities of Lemma \ \ref{lem:coeff}.
 (In alluding to ``base change" we are ignoring the parity issue that we have already seen in \ref{sect:Pi'}; all
 the parameters of the Eisenstein representations have to be either integers or half-integers, depending on the parity of $N$.
 But this can always be arranged.)
 Suppose for the moment that the assertion of Theorem \ \ref{thm:whittaker-periods} was valid for the
 pair $(\Sigma,\Sigma')$.  Then we would find that, up to algebraic factors, the critical values of
 $L(s,\Sigma\times \Sigma')$ were powers of $2 \pi i$ multiplied by
 $$\prod_{r = 1}^{N-1} P^{(r)}(\Sigma)\cdot \prod_{r' = 1}^{N-2}P^{(r')}(\Sigma).$$
 On the other hand, the $L$-function factors
 $$L(s,\Sigma\times \Sigma') = L(s,\Pi\times \Pi')\cdot \prod_{i = 1}^{N-n}L(s,\Pi'\otimes \chi_i)\cdot \prod_{j = 1}^{N-n}L(s,\Pi\otimes \chi'_j)
 \cdot  \prod_{i,j = 1}^{N-n}L(s,\chi_i\cdot\chi_j).$$
 Now we have expressions of the last three factors on the right-hand side in terms of
 $ P^{(i)}(\Pi)$ and $ P^{(j)}(\Pi')$ and CM periods of $\chi_i$ and $\chi'_j$, and we can
 expect that $P^{(r)}(\Sigma)$ and $P^{(r')}(\Sigma)$ can also be expressed in terms
 of the periods of $\Pi$, $\Pi'$, and the auxiliary Hecke characters.  In this way one would obtain
 an expression of the critical values of interest, namely those of the remaining term $L(s,\Pi\times \Pi')$,
 as powers of $ P^{(i)}(\Pi)$ and $ P^{(j)}(\Pi')$, as expected.

 The problem is that the integral representation used in Theorem \ \ref{thm:whittaker-periods} does
 not converge when the automorphic form on $GL_N$ is an Eisenstein series.  In \cite{iy}, Ichino and Yamana
 obtain the Rankin-Selberg product for Eisenstein representations of $GL_N\times GL_{N-1}$ as a regularized
 period integral.  It seems to be difficult  to interpret this regularized integral as a regularized cup product
 in rational cohomology.

Alternatively, assuming both $\Sigma$ and $\Sigma'$ descend to definite unitary groups, which must certainly
be possible for appropriate choices of $\chi_i$ and $\chi_j'$, we could apply the method described in Section \ \ref{wei}
to obtain an expression for the central critical value $L(\frac 12,\Pi\times \Pi')$ analogous to that in Corollary \ \ref{Ladjoint},
provided the $\chi_i$ and $\chi'_j$ could be chosen, with the given infinity types, so that the central critical values
$L(\frac 12,\Pi'\otimes \chi_i)$ and $L(\frac 12,\Pi\otimes \chi'_j)$ were all non-zero.  It would then be possible
to deduce the expected expressions for more general central critical values from the expected generalization of recent results of
Harder and Raghuram \cite{harrag}.  Unfortunately, proving the existence
of such $\chi_i$ and $\chi'_j$ seems to be an extremely difficult problem.

Yet another method would be to ignore the cup product altogether.  The Rankin-Selberg integral
\begin{equation}\label{integ}  \phi\otimes \phi' \mapsto I_m(\phi,\phi') := \int_{G'(\K)\backslash G'(\A_{\K})} \phi(g')\phi'(g')\|\det (g)\|^m dg \end{equation}
defines a $(\g,K,G(\A_f))$-invariant bilinear pairing on the space $\Pi_{(K)}\otimes \Pi'_{(K')}$ of $K\times K'$-finite elements in $\Pi\otimes \Pi'$, indexed by the integer $m$ (where this is defined by analytic continuation of the integral where necessary).  On the other hand, its realization in cohomology determines
a natural $\Q(\Pi_f)\Q(\Pi'_f)$-rational structure on $\Pi_f \otimes \Pi'_f$, and therefore
$L(\Pi,\Pi') := {\rm Hom}_{(\g,K,G(\A_f))}(\Pi_{(K)} \otimes \Pi'_{(K')},\C)$ also has a natural $\Q(\Pi_f)\Q(\Pi'_f)$-rational structure.  Under the so-called
automatic continuity hypothesis, any element of $L(\Pi,\Pi')$ extends continuously to the smooth Fr\'echet completions
of moderate growth. In this case, it is known that $L(\Pi,\Pi')$ is of dimension $1$.  Thus, there is a Rankin-Selberg
period invariant $P(\Pi,\Pi') \in \C^{\times}$ such that $P(\Pi,\Pi')^{-1}I_m$ is $\Q(\Pi_f)\Q(\Pi'_f)$-rational; moreover,
these invariants can be defined consistently for all Aut$(\C)$-conjugates of $\Pi$, $\Pi'$.  Verification of Deligne's
conjecture then comes down to identifying these $P(\Pi,\Pi')$, whether or not the pairing $I_m$ can be
related to a cohomological cup product.  If $\Pi'$ is abelian automorphic, the $P(\Pi,\Pi')$ can be related
as above to Petersson norms of holomorphic forms on unitary groups and periods of Hecke characters.
It would seem that $P(\Pi,\Pi')$ factors in general as a product of a period invariant attached to $\Pi$ and
one attached to $\Pi'$, each independent of the other factor, but we see no way to prove this in general.

In a forthcoming joint paper with S. Yamana, we avoid these difficulties by applying a different combinatorial
argument.  We anticipate that our method will give a version of Theorem \ref{tensorproduct} without assuming
the inequalities of Lemma \ \ref{lem:coeff}.

\end{document}